\documentclass[openany]{now}
\usepackage{mathtools}
\usepackage{amssymb}
\usepackage{xfrac}
\usepackage{framed}

\title{A Primer on Reproducing Kernel Hilbert Spaces}

\author{
Jonathan H. Manton \\
The University of Melbourne \\
Victoria 3010 Australia \\
j.manton@ieee.org
\and
Pierre-Olivier Amblard \\
CNRS \\
Grenoble 38402 France \\
pierre-olivier.amblard@gipsa-lab.grenoble-inp.fr
}

\newcommand{\reals}{\mathbb{R}}
\newcommand{\X}{\mathbb{X}}
\newcommand{\cx}{\mathbb{C}}
\newcommand{\Hs}{\mathcal{H}}

\begin{document}

\frontmatter

\maketitle

\tableofcontents

\mainmatter

\begin{abstract}
Reproducing kernel Hilbert spaces are elucidated without assuming
prior familiarity with Hilbert spaces.  Compared with extant pedagogic
material, greater care is placed on motivating the definition of
reproducing kernel Hilbert spaces and explaining when and why these
spaces are efficacious.  The novel viewpoint is that reproducing
kernel Hilbert space theory studies extrinsic geometry, associating
with each geometric configuration a canonical overdetermined
coordinate system.  This coordinate system varies continuously with
changing geometric configurations, making it well-suited for studying
problems whose solutions also vary continuously with changing
geometry. This primer can also serve as an introduction to
infinite-dimensional linear algebra because reproducing kernel
Hilbert spaces have more properties in common with Euclidean spaces
than do more general Hilbert spaces.
\end{abstract}

% ----------------------------------------------------
%                       CHAPTER
% ----------------------------------------------------

\chapter{Introduction}
\label{ch:intro}

Hilbert space theory is a prime example in mathematics of a beautiful synergy
between symbolic manipulation and visual reasoning.  Two-dimensional and
three-dimensional pictures can be used to reason about infinite-dimensional
Hilbert spaces, with symbolic manipulations subsequently verifying the soundness
of this reasoning, or suggesting modifications and refinements. Visualising a
problem is especially beneficial because over half the human brain is involved
to some extent with visual processing.  Hilbert space theory is an invaluable
tool in numerous signal processing and systems theory
applications~\cite{Small:2011wg,Curtain:1995hv,Berlinet:2012uv}.

Hilbert spaces satisfying certain additional properties are known as Reproducing
Kernel Hilbert Spaces (RKHSs), and RKHS theory is normally described as a
transform theory between Reproducing Kernel Hilbert Spaces and positive
semi-definite functions, called kernels: every RKHS has a unique kernel, and
certain problems posed in RKHSs are more easily solved by involving the kernel. 
However, this description hides the crucial aspect that the kernel captures not
just intrinsic properties of the Hilbert space but also how the Hilbert space is
embedded in a larger function space, which is referred to here as its extrinsic
geometry. A novel feature of this primer is drawing attention to this extrinsic
geometry, and using it to explain why certain problems can be solved more
efficiently in terms of the kernel than the space itself.

Another novel feature of this primer is that it motivates and
develops RKHS theory in finite dimensions before considering infinite
dimensions.  RKHS theory is ingenious; the underlying definitions
are simple but powerful and broadly applicable.  These aspects are
best brought out in the finite-dimensional case, free from the
distraction of infinite-dimensional technicalities.  Essentially
all of the finite-dimensional results carry over to the
infinite-dimensional setting.

This primer ultimately aims to empower readers to recognise when and how RKHS
theory can profit them in their own work.  The following are three of the known
uses of RKHS theory.

\begin{enumerate}
\item If a problem involves a subspace of a function space, and if the subspace
(or its completion) is a RKHS, then the additional properties enjoyed by RKHSs
may help solve the problem. (Explicitly computing limits of sequences in Hilbert
spaces can be difficult, but in a RKHS the limit can be found
pointwise.)

\item Certain classes of problems involving positive semi-definite functions can
be solved by introducing an associated RKHS whose kernel is precisely the
positive semi-definite function of interest. A classic example, due to Parzen,
is associating a RKHS with a stochastic process, where the kernel of the RKHS is
the covariance function of the stochastic process (see \S\ref{sec:rsp}).

\item Given a set of points and a function specifying the desired distances
between points, the points can be embedded in a RKHS with the distances between
points being precisely as prescribed; see \S\ref{sec:GbD}. (Support vector
machines use this to convert certain nonlinear problems into linear problems.)
\end{enumerate}

In several contexts, RKHS methods have been described as providing a unified
framework~\cite{Yao:1967tc, Kailath:1971hk, Newton:2002uz, Saitoh:1988vg};
although a subclass of problems was solved earlier by other techniques, a RKHS
approach was found to be more elegant, have broader applicability, or offer new
insight for obtaining actual solutions, either in closed form or numerically.
Parzen describes RKHS theory as facilitating a coordinate-free
approach~\cite{Newton:2002uz}.  While the underlying Hilbert space certainly
allows for coordinate-free expressions, the power of a RKHS beyond that of a
Hilbert space is the presence of two coordinate systems: the pointwise
coordinate system coming from the RKHS being a function space, and a canonical
(but overdetermined) coordinate system coming from the kernel.  The pointwise
coordinate system facilitates taking limits while a number of geometric problems
have solutions conveniently expressed in terms of what we define to be the
canonical coordinate system. (Geometers may wish to think of a RKHS as a
subspace $V \subset \reals^X$ with pointwise coordinates being the extrinsic
coordinates coming from $\reals^X$ while the canonical coordinates are intrinsic
coordinates on $V$ relating directly to the inner product structure on $V$.)

The body of the primer elaborates on all of the points mentioned
above and provides simple but illuminating examples to ruminate on.
Parenthetical remarks are used to provide greater technical detail
that some readers may welcome. They may be ignored without compromising
the cohesion of the primer.  Proofs are there for those wishing to
gain experience at working with RKHSs; simple proofs are
preferred to short, clever, but otherwise uninformative proofs.
Italicised comments appearing in proofs provide intuition or
orientation or both.

This primer is neither a review nor a historical survey, and as such, many
classic works have not been discussed, including those by leading pioneers such
as Wahba~\cite{wahba1973interpolating,wahba1981spline}.

\paragraph*{Contributions} This primer is effectively in two parts. The first
part (\S\ref{ch:intro}--\S\ref{sec:asp}), written by the first author, gives a
gentle and novel introduction to RKHS theory. It also presents several
classical applications. The second part (\S\ref{sec:embR}--\S\ref{ch:appemb}),
with \S\ref{sec:embR} written jointly and \S\ref{ch:appemb} written by the
second author, focuses on recent developments in the machine learning
literature concerning embeddings of random variables.

\section{Assumed Knowledge}

Basic familiarity with concepts from finite-dimensional linear algebra is
assumed: vector space, norm, inner product, linear independence, basis,
orthonormal basis, matrix manipulations and so forth.

Given an inner product $\langle\cdot,\cdot\rangle$, the induced norm is
$\|x\| = \sqrt{\langle x, x \rangle}$.  Not every norm comes from an inner
product, meaning some norms cannot be written in this form. If a norm does come
from an inner product, the inner product can be uniquely determined from
the norm by the polarisation identity $4 \langle x, y \rangle = \|x+y\|^2 -
\|x-y\|^2$.  (A corresponding formula exists for complex-valued vector spaces.)

A metric $d(\cdot,\cdot)$ is a ``distance function'' describing the distance
between two points in a metric space. To be a valid metric, it must satisfy
several axioms, including the triangle inequality.  A normed space is
automatically a metric space by the correspondence $d(x,y) = \|x-y\|$.

\section{Extrinsic Geometry and a Motivating Example}
\label{sec:eg}

Differential geometry groups geometric properties into two kinds: intrinsic and
extrinsic. Intrinsic properties depend only on the space itself, while extrinsic
properties depend on precisely how the space is embedded in a larger space. A
simple example in linear algebra is that the orientation of a straight line
passing through the origin in $\reals^2$ describes the \emph{extrinsic} geometry
of the line.

The following observation helps motivate the development of finite-dimensional
RKHS theory in \S\ref{sec:finRKHS}.  Let
\begin{equation}
L(\theta) = \{(t\cos\theta, t\sin\theta) \mid t \in \reals\} \subset \reals^2
\end{equation}
denote a straight line in $\reals^2$ passing through the origin and intersecting
the horizontal axis at an angle of $\theta$ radians; it is a one-dimensional
subspace of $\reals^2$.  Fix an arbitrary point $p = (p_1,p_2) \in \reals^2$ and
define $f(\theta)$ to be the point on $L(\theta)$ closest to $p$ with respect to
the Euclidean metric. It can be shown that
\begin{equation}
\label{eq:cs}
f(\theta) = (r(\theta) \cos\theta, r(\theta) \sin\theta), \qquad
	r(\theta) = p_1 \cos\theta + p_2 \sin\theta.
\end{equation}
Visualising $f(\theta)$ as the projection of $p$ onto $L(\theta)$ shows that
$f(\theta)$ depends \emph{continuously} on the orientation of the line. While
(\ref{eq:cs}) verifies this continuous dependence, it resorted to introducing an
\emph{ad hoc} parametrisation $\theta$, and different values of $\theta$ (e.g.,
$\theta$, $\pi + \theta$ and $2\pi + \theta$) can describe the same line.

\begin{framed} \noindent
Is there a more natural way of representing
$L(\theta)$ and $f(\theta)$, using linear algebra?
\end{framed}

A first attempt might involve using an orthonormal basis vector to represent
$L(\theta)$. However, there is no \emph{continuous} map from the line
$L(\theta)$ to an orthonormal basis vector $v(\theta) \in L(\theta)$.
(This should be self-evident with some thought, and follows rigorously from the
Borsuk-Ulam theorem.) Note that $\theta \mapsto (\cos\theta,\sin\theta)$ is not
a well-defined map from $L(\theta)$ to $\reals^2$ because $L(0)$ and $L(\pi)$
represent the same line yet $(\cos 0, \sin 0) \neq (\cos \pi, \sin \pi)$.

RKHS theory uses not one but two vectors to represent $L(\theta)$. Specifically,
it turns out that the kernel of $L(\theta)$, in matrix form, is
\begin{equation}
K(\theta) = \begin{bmatrix} \cos\theta \\ \sin\theta \end{bmatrix}
\begin{bmatrix} \cos\theta & \sin\theta \end{bmatrix} = \begin{bmatrix}
\cos^2\theta & \sin\theta \cos\theta \\ \sin\theta \cos\theta & \sin^2\theta
\end{bmatrix}.
\end{equation}
The columns of $K(\theta)$ are
\begin{equation}
k_1(\theta) = \cos\theta \begin{bmatrix} \cos\theta \\ \sin \theta
\end{bmatrix},\qquad k_2(\theta) = \sin\theta \begin{bmatrix} \cos\theta \\
\sin \theta \end{bmatrix}.
\end{equation}
Note that $L(\theta)$ is spanned by $k_1(\theta)$ and $k_2(\theta)$, and
moreover, both $k_1$ and $k_2$ are well-defined (and continuous) functions of
$L$; if $L(\theta) = L(\phi)$ then $k_1(\theta) = k_1(\phi)$ and $k_2(\theta) =
k_2(\phi)$.  To emphasise, although $\theta$ is used here for convenience to
describe the construction, RKHS theory defines a map from $L$ to $k_1$
and $k_2$ that does not depend on any \emph{ad hoc} choice of parametrisation.
It is valid to write $k_1(L)$ and $k_2(L)$ to show they are functions of $L$
alone.

Interestingly, $f$ has a simple representation in terms of the
kernel:
\begin{equation}
f(L) = p_1\, k_1(L) + p_2\, k_2(L).
\end{equation}
Compared with (\ref{eq:cs}), this is both simple and natural, and does not
depend on any \emph{ad hoc} parametrisation $\theta$ of the line $L$. In
summary,
\begin{itemize}
  \item the kernel represents a vector subspace by a possibly overdetermined
  (i.e., linearly dependent) ordered set of vectors, and the correspondence from
  a subspace to this ordered set is continuous;
  \item this \emph{continuous} correspondence cannot be achieved with an ordered
  set of basis (i.e., linearly independent) vectors;
  \item certain problems have solutions that depend continuously on the subspace
  and can be written elegantly in terms of the kernel:
  \begin{equation}
  \mathrm{subspace} \rightarrow \mathrm{kernel} \rightarrow
  \mathrm{solution}.
  \end{equation}
\end{itemize}

The above will be described in greater detail in \S\ref{sec:finRKHS}.

\paragraph*{Remark}
The above example was chosen for its simplicity. Ironically, the general problem
of projecting a point onto a subspace is not well-suited to the RKHS framework
for several reasons, including that RKHS theory assumes there is a norm only on
the subspace; if there is a norm on the larger space in which the subspace sits
then it is ignored.  A more typical optimisation problem benefitting from RKHS
theory is finding the minimum-norm function passing through a finite number of
given points; minimising the norm acts to regularise this interpolation
problem; see \S\ref{sec:ip}.

\section{Pointwise Coordinates and Canonical Coordinates}
\label{sec:cc}

Aimed at readers already familiar with Hilbert space theory, this section
motivates and defines two coordinate systems on a RKHS.

A separable Hilbert space $\mathcal{H}$ possesses an orthonormal basis
$e_1,e_2,\cdots \in \mathcal{H}$. An arbitrary element $v \in \mathcal{H}$ can
be expressed as an infinite series $v = \sum_{i=0}^\infty \alpha_i e_i$ where
the ``coordinates'' $\alpha_i$ are given by $\alpha_i = \langle v , e_i
\rangle$. A classic example is using a Fourier series to represent a periodic
function. The utility of such a construction is that an arbitrary element of
$\mathcal{H}$ can be written as the limit of a linear combination of a
manageable set of fixed vectors.

RKHS theory generalises this ability of writing an arbitrary element of a
Hilbert space as the limit of a linear combination of a manageable set of fixed
vectors. If $\mathcal{H} \subset \reals^T$ is a (not necessarily separable) RKHS
then an arbitrary element $v \in \mathcal{H}$ can be expressed as the limit of a
sequence $v_1,v_2,\cdots \in \mathcal{H}$ of vectors, each of which is a finite
linear combination of the vectors $\{K(\cdot,t) \mid t \in T\}$, where $K\colon
T \times T \rightarrow \reals$ is the kernel of $\mathcal{H}$. It is this
ability to represent an arbitrary element of a RKHS $\mathcal{H}$ as the limit
of a linear combination of the $K(\cdot,t)$ that, for brevity, we refer to as
the presence of a canonical coordinate system. The utility of this canonical
coordinate system was hinted at in \S\ref{sec:eg}.

There is another natural coordinate system: since an element $v$ of a RKHS
$\mathcal{H} \subset \reals^T$ is a function from $T$ to $\reals$, its $t$th
coordinate can be thought of as $v(t)$. The relationship between this pointwise
coordinate system and the aforementioned canonical coordinates is that $v(t) =
\langle v, K(\cdot,t) \rangle$. Note though that whereas an arbitrary linear
combination of the $K(\cdot,t)$ is guaranteed to be an element of $\mathcal{H}$,
assigning values arbitrarily to the $v(t)$, i.e., writing down an arbitrary
function $v$, may not yield an element of $\mathcal{H}$; canonical coordinates
are intrinsic whereas pointwise coordinates are extrinsic. The utility of the
pointwise coordinate system is that limits in a RKHS can be determined
pointwise: if $v_k$ is a Cauchy sequence, implying there exists a $v$ satisfying
$\|v_k - v \| \rightarrow 0$, then $v$ is fully determined by $v(t) = \lim_k
v_k(t)$.

% ----------------------------------------------------
%                       CHAPTER
% ----------------------------------------------------

\chapter{Finite-dimensional RKHSs}
\label{sec:finRKHS}

Pedagogic material on RKHSs generally considers infinite-dimensional spaces from
the outset\footnote{The tutorial~\cite{steinke2008kernels} considers
finite-dimensional spaces but differs markedly from the presentation here. An
aim of~\cite{steinke2008kernels} is explaining the types of regularisation
achievable by minimising a RKHS norm.}. Starting with finite-dimensional spaces
though is advantageous because the remarkable aspects of RKHS theory are already
evident in finite dimensions, where they are clearer to see and easier to study.
The infinite-dimensional theory is a conceptually straightforward extension of
the finite-dimensional theory.

Although elements of a RKHS must be functions, there is a canonical
correspondence between $\reals^n$ and real-valued functions on
$\{1,2,\cdots,n\}$ given by the rule that $(x_1,\cdots,x_n) \in
\reals^n$ is equivalent to the function $f\colon\{1,2,\cdots,n\}\rightarrow
\reals$ satisfying $f(i) = x_i$ for $i=1,\cdots,n$.  To exemplify, the point
$(5,8,4) \in \reals^3$ is canonically equivalent to the function
$f\colon\{1,2,3\} \rightarrow \reals$ given by $f(1)=5$, $f(2)=8$ and $f(3)=4$.
For simplicity and clarity, this primer initially works with $\reals^n$.

\paragraph*{Remark}
The term finite-dimensional RKHS is potentially ambiguous; is merely the RKHS
itself finite-dimensional, or must the embedding space also be finite
dimensional? For convenience, we adopt the latter interpretation. While this
interpretation is not standard, it is consistent with our emphasis on the
embedding space playing a significant role in RKHS theory.  Precisely, we say a
RKHS is finite dimensional if the set $X$ in \S\ref{sec:idRKHS} has finite
cardinality.

\section{The Kernel of an Inner Product Subspace}

Central to RKHS theory is the following question.
Let $V \subset \reals^n$ be endowed with an inner product.
How can this configuration be described efficiently?
The configuration involves three aspects:
the vector space $V$, the orientation of $V$ in $\reals^n$, and
the inner product on $V$. Importantly, the inner product is not defined on the
whole of $\reals^n$, unless $V = \reals^n$. (An alternative viewpoint that
will emerge later is that RKHS theory studies possibly degenerate inner products
on $\reals^n$, where $V$ represents the largest subspace on which the inner
product is not degenerate.)

One way of describing the configuration is by writing down a basis
$\{v_1,\cdots,v_r\} \subset V \subset \reals^n$ for $V$ and the
corresponding Gram matrix $G$ whose $ij$th element is $G_{ij} =
\langle v_i, v_j \rangle$.  Alternatively, the configuration is
completely described by giving an orthonormal basis $\{u_1,\cdots,u_r\}
\subset V \subset \reals^n$; the Gram matrix associated with an
orthonormal basis is the identity matrix and does not need
to be stated explicitly.  However, these representations do not satisfy the
following requirements because there is no unique choice for a
basis, even an orthonormal basis.

\begin{description}
\item[One-to-one] The relationship between a subspace $V$ of
$\reals^n$ and its representation should be one-to-one.

\item[Respect Topology] If $(V_1,\langle \cdot , \cdot \rangle_1),
(V_2,\langle \cdot , \cdot \rangle_2), \cdots$ is a sequence of
inner product spaces ``converging'' to $(V,\langle \cdot, \cdot
\rangle)$ then the representations of $(V_1,\langle \cdot , \cdot
\rangle_1), (V_2,\langle \cdot , \cdot \rangle_2), \cdots$ should
``converge'' to the representation of $(V,\langle \cdot, \cdot
\rangle)$, and \emph{vice versa}.

\item[Straightforward Inverse] It should be easy to deduce $V$ and
its inner product from its representation.
\end{description}

The lack of a canonical basis for $V$ can be overcome by considering
spanning sets instead.  Whereas a basis induces a coordinate system,
a spanning set that is not a basis induces an overdetermined
coordinate system due to linear dependence.

RKHS theory produces a unique ordered spanning set $k_1,\cdots,k_n$ for $V
\subset \reals^n$ by the rule that $k_i$ is the unique vector in
$V$ satisfying $\langle v, k_i \rangle = e_i^\top v$ for all $v \in
V$.  Here and throughout, $e_i$ denotes the vector whose elements
are zero except for the $i$th which is unity; its dimension should
always be clear from the context. In words, taking the inner product
with $k_i$ extracts the $i$th element of a vector.  This representation
looks \emph{ad hoc} yet has proven to be remarkably useful, in part
because it unites three different perspectives given by the three
distinct but equivalent definitions below.

\begin{definition}
\label{def:kernel}
Let $V \subset \reals^n$ be an inner product space. The \emph{kernel}
of $V$ is the unique matrix $K = [k_1\ k_2\ \cdots\ k_n]\in \reals^{n
\times n}$ determined by any of the following three equivalent
definitions.
\begin{enumerate}
\item $K$ is such that each $k_i$ is in $V$ and $\langle v, k_i \rangle = e_i^\top v$ for all $v \in V$.
\item $K = u_1 u_1^\top
+ \cdots + u_r u_r^\top$ where $u_1,\cdots,u_r$ is an orthonormal basis for $V$.
\item $K$ is such that the $k_i$ span $V$ and $\langle k_j, k_i \rangle = K_{ij}$.
\end{enumerate}
\end{definition}

The third definition is remarkable despite following easily from
the first definition because it shows that $K \in \reals^{n \times
n}$ is the Gram matrix corresponding to the vectors $k_1,\cdots,k_n$.
The kernel simultaneously encodes a set of vectors that span $V$
and the corresponding Gram matrix for those vectors!

The name \emph{reproducing kernel} can be attributed either to the
kernel $K$ reproducing itself in that its $ij$th element $K_{ij}$
is the inner product of its $j$th and $i$th columns $\langle k_j,
k_i \rangle$, or to the vectors $k_i$ reproducing an arbitrary
vector $v$ by virtue of $v = (\langle v,k_1 \rangle,\cdots, \langle
v, k_n \rangle)$.

\begin{example}
\label{ex:Qinv}
Equip $V = \reals^n$ with an inner product $\langle u,v
\rangle = v^\top Q u$ where $Q$ is symmetric (and positive definite).
The equation $\langle v, k_i \rangle = e_i^\top v$ implies
$k_i = Q^{-1} e_i$, that is, $K = Q^{-1}$.
Alternatively, an eigendecomposition $Q = XDX^\top$ yields an
orthonormal basis for $V$ given by the scaled columns of $X$, namely,
$\{XD^{-\frac12}e_1,\cdots,XD^{-\frac12}e_n\}$.  Therefore $K =
\sum_{i=1}^n (XD^{-\frac12}e_i)(XD^{-\frac12}e_i)^\top = XD^{-1}X^\top =
Q^{-1}$.  Observe that $\langle k_j, k_i \rangle = \langle Q^{-1}e_j,
Q^{-1}e_i \rangle =  e_i^\top Q^{-1} e_j = e_i^\top K e_j = K_{ij}$.
\end{example}

\begin{example}
Let $V \subset \reals^2$ be the subspace spanned by the vector $(1,1)$
and endowed with the inner product giving the vector $(1,1)$ unit norm.
Then $\{(1,1)\}$ is an orthonormal basis for $V$ and therefore
$K = \begin{bmatrix} 1 \\ 1 \end{bmatrix} \begin{bmatrix} 1 & 1 \end{bmatrix}
= \begin{bmatrix} 1 & 1 \\ 1 & 1 \end{bmatrix}$.
\end{example}

\begin{example}
The configuration having $K = \begin{bmatrix} 1 & 0 \\ 0 & 0
\end{bmatrix}$ as its kernel can be found as follows. Since $K$ is
two-by-two, $V$ must sit inside $\reals^2$. Since $V$ is the
span of $k_1 = [1\ 0]^\top$ and $k_2 = [0\ 0]^\top$, $V = \reals \times \{0\}$.
The vector $k_1$ has unit norm in $V$ because $\langle k_1, k_1
\rangle = K_{11} = 1$.
\end{example}

Before proving the three definitions of $K$ are equivalent, several
remarks are made.  Existence and uniqueness of $K$ follows
most easily from definition one because the defining equations
are linear.  From definition one alone it would seem remarkable
that $K$ is always positive semi-definite, denoted $K \geq 0$, yet
this fact follows immediately from definition two. Recall that a
positive semi-definite matrix is a symmetric matrix whose eigenvalues
are greater than or equal to zero. That $K$ is unique is not clear
from definition two alone. Definition three gives perhaps the most
important characterisation of $K$, yet from definition three
alone it is not at all obvious that such a $K$ exists.
Definition three also implies that any element
of $V$ can be written as a linear combination $K\alpha$ of the
columns of $K$, and $\langle K \alpha, K \beta \rangle = \beta^\top K
\alpha$. The reader is invited to verify this by writing $K\beta$
as $\beta_1 k_1 + \cdots + \beta_n k_n$. 

\begin{lemma}
\label{lem:Kunique}
Given an inner product space $V \subset \reals^n$, there is
precisely one $K = [k_1,\cdots,k_n] \in \reals^{n \times n}$ 
for which each $k_i$ is in $V$ and satisfies
$\langle v, k_i \rangle = e_i^\top v$ for all $v \in V$.
\end{lemma}
\begin{proof}
\emph{This is an exercise in abstract linear algebra that is
the generalisation of the linear equation $Ax=b$ having a
unique solution if $A$ is square and non-singular.}
Let $v_1,\cdots,v_r$ be a basis for $V$. Linearity implies
the constraints on $k_i$ are equivalent to requiring
$\langle v_j, k_i \rangle = e_i^\top v_j$ for $j=1,\cdots,r$.
Let $L$ denote the linear operator taking $k \in V$ to
$(\langle v_1, k \rangle, \langle v_2, k \rangle, \cdots,
\langle v_r, k \rangle)$. Both the domain $V$ and the range
$\reals^r$ of $L$ have dimension $r$. \emph{The linear operator
$L$ is equivalent to a square matrix.}
Now, $L$ is injective because if $L(k) = L(\tilde k)$ then
$L(k-\tilde k) = 0$, that is, $\langle v, k-\tilde k \rangle = 0$
for all $v \in V$ (because $v_1,\cdots,v_r$ is a basis), and
in particular, $\langle k - \tilde k, k - \tilde k \rangle = 0$,
implying $k=\tilde k$. \emph{The linear operator $L$ is 
non-singular.} Therefore, $L$ is also surjective, and
in particular, $L(k_i) = (e_i^\top v_1, \cdots, e_i^\top v_r)$ has
one and only one solution.
\end{proof}

Definition three gives a non-linear characterisation of $K$ requiring
$k_i$ to extract the $i$th element of the vectors $k_j$ whereas
definition one requires $k_i$ to extract the $i$th element of any
vector $v$.  However, definition three also requires the
$k_j$ to span $V$, therefore, by writing an arbitrary vector $v$
as a linear combination of the $v_j$, it becomes evident that
definition three satisfies definition one.

\begin{lemma}
If $K$ is such that the $k_i$ span $V$ and $\langle k_j, k_i \rangle
= K_{ij}$ then $\langle v, k_i \rangle = e_i^\top v$ for all $v \in V$.
\end{lemma}
\begin{proof}
Fix a $v \in V$.
Since $v$ is in $V$ and the $k_i$ span $V$, there is a vector
$\alpha$ such that $v = K\alpha$. Also, $k_i = Ke_i$.
Therefore, $\langle v, k_i \rangle = \langle K\alpha, Ke_i \rangle
= e_i^\top K \alpha = e_i^\top v$, as required. (Recall from earlier that
$\langle k_j, k_i \rangle
= K_{ij}$ implies $\langle K\alpha,K\beta \rangle = \beta^\top K \alpha$.)
\end{proof}

The existence of a $K$ satisfying definition three is implied by the
existence of a $K$ satisfying definition one.

\begin{lemma}
If $K$ is such that each $k_i$ is in $V$ and
$\langle v, k_i \rangle = e_i^\top v$ for all $v
\in V$ then the $k_i$ span $V$ and $\langle k_j, k_i \rangle =
K_{ij}$.
\end{lemma}
\begin{proof}
That $\langle k_j, k_i \rangle = K_{ij}$ follows immediately from
$K_{ij}$ being the $i$th element of $k_j$, namely, $e_i^\top k_j$.
If the $k_i$ do not span $V$ then there is a non-zero $k \in V$ which
is orthogonal to each and every $k_i$, that is, $\langle k, k_i \rangle = 0$.
Yet this implies $e_i^\top k = 0$ for all $i$, that is, $k=0$, a contradiction.
\end{proof}

If the columns of $U \in \reals^{n \times r}$ form an orthonormal basis for 
$V \subset \reals^n$ then an arbitrary vector in $V$ can be written as
$U\alpha$, and moreover, $\langle U\alpha, U\beta \rangle = \beta^\top \alpha$.
Referring to definition two, if the $u_i$ are the columns of $U$, then
$u_1u_1^\top + \cdots + u_r u_r^\top = UU^\top$. Armed with these facts, showing
definition two satisfies definition one becomes straightforward.

\begin{lemma}
If $u_1,\cdots,u_r$ is an orthonormal basis for $V$ then $K = u_1
u_1^\top + \cdots + u_r u_r^\top$ satisfies $\langle v, k_i \rangle =
e_i^\top v$ for all $v \in V$.
\end{lemma}
\begin{proof}
Let $U = [u_1,\cdots,u_r]$, so that $K = UU^\top$ and $k_i = Ke_i = UU^\top e_i$.
An arbitrary
$v \in V$ can be represented as $v = U\alpha$. Therefore
$\langle v, k_i \rangle = \langle U\alpha, UU^\top e_i \rangle =
e_i^\top U \alpha = e_i^\top v$, as required.
\end{proof}

Since definition one produces a unique $K$, and definitions two and
three both produce a $K$ satisfying definition one, the equivalence
of the three definitions has been proven.

The next pleasant surprise is that for every positive semi-definite
matrix $K \in \reals^{n \times n}$ there is an inner product space
$V \subset \reals^n$ whose kernel is $K$. The proof relies on the
following.

\begin{lemma}
\label{lem:aKa}
If $K \in \reals^{n \times n}$
is positive semi-definite then $\alpha^\top K \alpha = 0$ implies
$K \alpha = 0$, where $\alpha$ is a vector in $\reals^n$.
\end{lemma}
\begin{proof}
Eigendecompose $K$ as $K = U D U^\top$. Let $\beta = U^\top \alpha$.
Then $\alpha^\top K \alpha = 0$ implies $\beta^\top D \beta = 0$.
Since $D \geq 0$, $D \beta$ must be zero. Therefore $K \alpha =
U D U^\top \alpha = U D \beta = 0$.
\end{proof}

\begin{lemma}
\label{lem:VfromK}
Let $V = \operatorname{span}\{k_1,\cdots,k_n\}$ be the space spanned
by the columns $k_1,\cdots,k_n$ of a positive semi-definite matrix
$K \in \reals^{n \times n}$. There exists an inner product
on $V$ satisfying $\langle k_j, k_i \rangle = K_{ij}$.
\end{lemma}
\begin{proof}
Let $u,v \in V$. Then there exist vectors $\alpha, \beta$ such
that $u = K\alpha$ and $v = K\beta$. Define $\langle u,v\rangle$
to be $\beta^\top K \alpha$. To show this is well-defined, let
$u = K\tilde\alpha$ and $v = K\tilde\beta$ be possibly different
representations. Then $\beta^\top K \alpha - \tilde\beta^\top K \tilde\alpha
= (\beta - \tilde\beta)^\top K \alpha + \tilde\beta^\top K (\alpha - \tilde\alpha)
= 0$ because $K\alpha = K\tilde\alpha$, $K\beta = K\tilde\beta$ and $K=K^\top$.
Clearly $\langle \cdot,\cdot \rangle$ so defined is bilinear.
To prove $\langle \cdot,\cdot \rangle$ is positive definite,
assume $\langle K\alpha, K\alpha \rangle = \alpha^\top K \alpha = 0$. 
By Lemma~\ref{lem:aKa} this implies $K\alpha=0$, as required.
\end{proof}

Given $K$, there is a \emph{unique} inner product space whose kernel is $K$.

\begin{lemma}
\label{lem:Kinject}
Let $V_1 \subset \reals^n$ and $V_2 \subset \reals^n$ be two inner
product spaces having the same kernel $K$. Then $V_1$ and $V_2$ are
identical spaces: $V_1 = V_2$ and their inner products are the same.
\end{lemma}
\begin{proof}
The columns of $K$ span both $V_1$ and $V_2$, hence $V_1 = V_2$.
For the same reason, the inner products on $V_1$ and $V_2$ are
uniquely determined from the Gram matrix $K$ corresponding to $k_1,\cdots,k_n$.
Since $V_1$ and $V_2$ have the same Gram matrix, their inner products
are identical too. (Indeed, the inner product must be given
by $\langle K\alpha, K\beta \rangle = \beta^\top K \alpha$.)
\end{proof}

To summarise, for a fixed $n$, there is a bijective correspondence
between inner product spaces $V \subset \reals^n$ and positive
semi-definite matrices $K \in \reals^{n \times n}$.  Given $V \subset
\reals^n$ there is precisely one kernel $K \geq 0$
(Definition~\ref{def:kernel} and Lemma~\ref{lem:Kunique}). Given a
$K \geq 0$, there is precisely one inner product space $V \subset
\reals^n$ for which $K$ is its kernel (Lemmata~\ref{lem:VfromK}
and~\ref{lem:Kinject}).  Recalling the criteria listed earlier,
this correspondence is One-to-one and has a Straightforward Inverse.
That it Respects Topology is discussed next.

\section{Sequences of Inner Product Spaces}

A sequence of one-dimensional vector spaces in $\reals^2$ can be
visualised as a collection of lines passing through the origin and
numbered $1,2,\cdots$. It can be recognised visually when such a
sequence converges. (This corresponds to the topology of the Grassmann
manifold.) The kernel representation of an inner product space
induces a concept of convergence for the broader situation of a
sequence of \emph{inner product} subspaces, possibly of \emph{differing
dimensions}, by declaring that the limit $V_\infty \subset \reals^n$
of a sequence of inner product spaces $V_1,V_2,\cdots \subset
\reals^n$ is the space whose kernel is $K_\infty = \lim_{n \rightarrow
\infty} K_n$, if the limit exists.  This makes the kernel representation
of a subspace Respect Topology. The pertinent question is whether
the induced topology is sensible in practice.

\begin{example}
\label{ex:Kinf}
Define on $\reals^2$ a sequence of inner products $\langle u, v
\rangle_n = v^\top Q_n u$ where $Q_n = \operatorname{diag}\{1,n^2\}$
is a diagonal matrix with entries $1$ and $n^2$.  An orthonormal
basis for the $n$th space is $\{(1,0),(0,n^{-1})\}$.  
It seems acceptable, at least in certain situations,
to agree that the limit of this sequence
of two-dimensional spaces is the one-dimensional subspace of
$\reals^2$ spanned by $(1,0)$.  This accords with defining convergence
via kernels.  Let $K_n$ be the kernel of the $n$th space: $K_n
= \operatorname{diag}\{1,n^{-2}\}$. Then $K_\infty = \lim_{n
\rightarrow \infty} K_n = \operatorname{diag}\{1,0\}$. The space
$V_\infty$ having kernel $K_\infty$ is the subspace of $\reals^2$
spanned by $(1,0)$, where the inner product is such that $(1,0)$
has unit norm.
\end{example}

The example above illustrates a general principle: since $K$ can
be defined in terms of an orthonormal basis as $K = u_1u_1^\top +
\cdots + u_r u_r^\top$, a sequence of higher-dimensional inner product
spaces can converge to a lower-dimensional inner product space if
one or more of the orthonormal basis vectors approaches the zero
vector.
This need not be the only sensible topology, but it is a topology
that finds practical use, as later verified by examples.

\paragraph*{Remark}
Recall from Example~\ref{ex:Qinv} that if $V=\reals^n$ then $K=Q^{-1}$. If $Q_n$
is a sequence of positive definite matrices becoming degenerate in the limit,
meaning one or more eigenvalues goes to infinity, then $Q_n$ does not
have a limit.  However, the corresponding eigenvalues of $K_n =
Q_n^{-1}$ go to zero, hence the kernel $K_n$ may have a well-defined limit. In
this sense, RKHS theory encompasses degenerate inner products. A way of
visualising this is presented in \S\ref{sec:vis}.

\section{Extrinsic Geometry and Interpolation}

Finite-dimensional RKHS theory studies subspaces $V$ of $\reals^n$
rather than abstract vector spaces.  If it is only known that $Z$
is an abstract vector space then there is no way of knowing what
the elements of $Z$ look like. The best that can be done is assert
the existence of a basis $\{b_1,\cdots,b_r\}$.  By comparison,
knowing $V \subset \reals^n$ allows working with $V$ using
\emph{extrinsic coordinates} by writing an element of $V$ as a
vector in $\reals^n$.

Writing $V \subset \reals^n$ may also signify the importance of the
orientation of $V$ inside $\reals^n$.  If $V$ and $W$ are $r$-dimensional
linear subspaces of $\reals^n$ then their \emph{intrinsic geometry}
is the same but their \emph{extrinsic geometry} may differ; $V$ and
$W$ are equivalent (precisely, isomorphic) as vector spaces but
they lie inside $\reals^n$ differently unless $V=W$.

The usefulness of extrinsic geometry is exemplified by considering
the interpolation problem of finding a vector $x \in V \subset
\reals^n$ of smallest norm and some of whose coordinates $e_i^\top x$
are specified.  This problem is posed using extrinsic geometry. It
can be rewritten intrinsically, without reference to $\reals^n$,
by using the linear operator $L_i\colon V \rightarrow \reals$ defined
to be the restriction of $x \mapsto e_i^\top x$.  However, the extrinsic
formulation is the more amenable to a unified approach.

\begin{example}
\label{ex:Interp}
Endow $V \subset \reals^n$ with an inner product.  Fix
an $i$ and consider how to find $x \in V$ satisfying $e_i^\top x = 1$
and having the smallest norm. Geometrically, such an $x$ must be
orthogonal to any vector $v \in V$ satisfying $e_i^\top v = 0$, for
otherwise its norm could be decreased. This approach allows $x$
to be found by solving a system of linear equations,
however, going further seems to require choosing a basis for $V$.
\end{example}

RKHS theory replaces an \emph{ad hoc} choice of basis for $V \subset
\reals^n$ by the particular choice $k_1,\cdots,k_n$ of spanning vectors 
for $V$.  Because the kernel representation Respects Topology,
the vectors $k_1,\cdots,k_n$ vary continuously with changing $V$.
The solution to the interpolation problem should also vary continuously
as $V$ changes. There is thus a chance that the solution can be
written elegantly in terms of $k_1,\cdots,k_n$.

\begin{example}
\label{ex:Interp2}
Continuing the example above, let $K=[k_1,\cdots,k_n]$ be the kernel
of $V$. Let $x = K\alpha$ and $v = K\beta$. The constraints $e_i^\top x = 1$
and $\langle x,v \rangle = 0$ whenever $e_i^\top v = 0$ become $k_i^\top \alpha = 1$
and $\beta^\top K \alpha = 0$ whenever $\beta^\top K e_i = 0$.
(Recall that $i$ is fixed.) The latter requirement is satisfied by
$\alpha = c e_i$ where $c \in \reals$ is a scalar. Solving $k_i^\top \alpha =1$
implies $c = K_{ii}^{-1}$. Therefore, $x = K_{ii}^{-1}k_i$.
Note $\|x\|^2 = \langle x, x \rangle = K_{ii}^{-1}$.
\end{example}

In the above example, as $V$ changes, both the kernel $K$ and the
solution $x$ change.  Yet the relationship between $x$ and $K$
remains constant: $x = K_{ii}^{-1}k_i$.  This is further evidence
that the representation $K$ of the extrinsic geometry is a useful
representation.

There is a geometric explanation for the columns of $K$ solving the
single-point interpolation problem.  Let $L_i\colon V \rightarrow
\reals$ denote the $i$th coordinate function: $L_i(v) = e_i^\top v$.
That $\langle v, k_i \rangle = L_i(v)$ means $k_i$ is the gradient
of $L_i$. In particular, the line determined by $k_i$ meets the
level set $\{ v \mid L_i(v) = 1 \}$ at right angles, showing that
$k_i$ meets the orthogonality conditions for optimality.
Turning this around leads to yet another definition of $K$.

\begin{lemma}
\label{lem:geomK}
The following is a geometric definition of the $k_i$ 
that is equivalent to Definition~\ref{def:kernel}.  Let $H_i = \{z
\in \reals^n \mid e_i^\top z = 1\}$ be the hyperplane consisting of
all vectors whose $i$th coordinate is unity.  If $V \cap H_i$ is
empty then define $k_i = 0$. Otherwise, let $\tilde k_i$ be the
point in the intersection $V \cap H_i$ that is closest to the origin.
Define $k_i$ to be $k_i = c \tilde k_i$
where $c = \langle \tilde k_i, \tilde k_i \rangle^{-1}$.
\end{lemma}
\begin{proof}
Since $K$ is unique (Lemma~\ref{lem:Kunique}), it suffices to prove
the $k_i$ defined here satisfy definition one of
Definition~\ref{def:kernel}. If $V \cap H_i$ is empty then $e_i^\top v=0$
for all $v \in V$ and $k_i = 0$ satisfies definition one.  Assume
then that $V \cap H_i$ is non-empty. It has a closest point, $\tilde
k_i$, to the origin. As $e_i^\top \tilde k_i = 1$, $\tilde k_i$ is
non-zero and $c$ is well-defined.  Also, $\langle w, \tilde k_i
\rangle = 0$ whenever $w \in V$ satisfies $e_i^\top w = 0$, for otherwise
$\tilde k_i$ would not have the smallest norm.  Let $v \in V$ be
arbitrary and define $w = v - a \tilde k_i$ where $a = e_i^\top v$.
Then $e_i^\top w = e_i^\top v - a = 0$.  Thus $\langle v, k_i \rangle =
\langle w + a \tilde k_i, c \tilde k_i \rangle = a c \langle \tilde
k_i, \tilde k_i \rangle = a = e_i^\top v$, as required.
\end{proof}

The $c$ in the lemma scales $\tilde k_i$ so that $\langle k_i, k_i
\rangle = e_i^\top k_i$.  It is also clear that if $k_i \neq 0$ then
$K_{ii} \neq 0$, or in other words, if $K_{ii} = 0$ then the $i$th
coordinate of every vector $v$ in $V$ is zero and the interpolation
problem has no solution.

\section{Visualising RKHSs}
\label{sec:vis}

It is more expedient to understand RKHS theory by focusing on the
columns $k_i$ of the kernel $K$ rather than on the matrix $K$ itself;
$K$ being a positive semi-definite matrix corresponding to the Gram
matrix of the $k_i$ is a wonderful bonus.  The indexed set of vectors
$k_1,\cdots,k_n$ changes in a desirable way as $V \subset \reals^n$
changes.  This can be visualised explicitly for low-dimensional examples.

The inner product on a two-dimensional vector space $V$ can be
depicted by drawing the ellipse $\{v \in V \mid \langle v,v \rangle
= 1\}$ on $V$ because the inner product is uniquely determined from
the norm, and the ellipse uniquely determines the norm.  Based on
Lemma~\ref{lem:geomK}, the columns $k_1$ and $k_2$ of the kernel
$K$ corresponding to $V \subset \reals^2$ can be found geometrically,
as explained in the caption of Figure~\ref{fig:findK}.

\begin{figure}[h]
\centering
\includegraphics[width=0.7\linewidth]{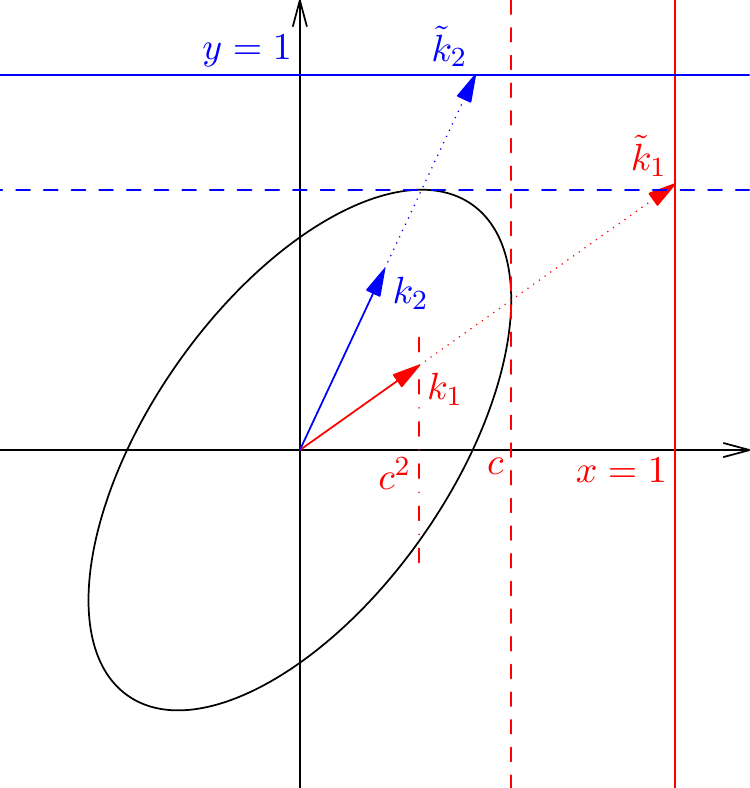}
\caption{
The ellipse comprises all points one unit from the origin. It
determines the chosen inner product on $\reals^2$.  The vector
$\tilde k_1$ is the closest point to the origin on the vertical
line $x=1$.  It can be found by enlarging the ellipse until it first
touches the line $x=1$, or equivalently, as illustrated, it can be
found by shifting the line $x=1$ horizontally until it meets the
ellipse tangentially, represented by the dashed vertical line, then
travelling radially outwards from the point of intersection until
reaching the line $x=1$.  The vector $k_1$ is a scaled version of
$\tilde k_1$. If the dashed vertical line intersects the $x$-axis
at $c$ then $k_1 = c^2 \tilde k_1$. Equivalently, $k_1$ is such
that its tip intersects the line $x=c^2$.  The determination of
$k_2$ is analogous but with respect to the horizontal line $y=1$.
}
\label{fig:findK}
\end{figure}

Figures~\ref{fig:Kpic} and~\ref{fig:Kgraph} reveal how $k_1$ and
$k_2$ vary as the inner product on $V = \reals^2$ changes.
Figure~\ref{fig:Kgraph} shows that $k_1$ and $k_2$ vary smoothly
as the inner product changes. By comparison, an orthonormal basis
formed from the principal and minor axes of the ellipse must have
a discontinuity somewhere, because after a 180 degree rotation, the
ellipse returns to its original shape yet, for example,
$\{e_1,\frac12e_2\}$ rotated 180 degrees is $\{-e_1,-\frac12e_2\}$,
which differs from $\{e_1,\frac12e_2\}$.

\begin{figure}[h]
\centering
\includegraphics[width=\linewidth]{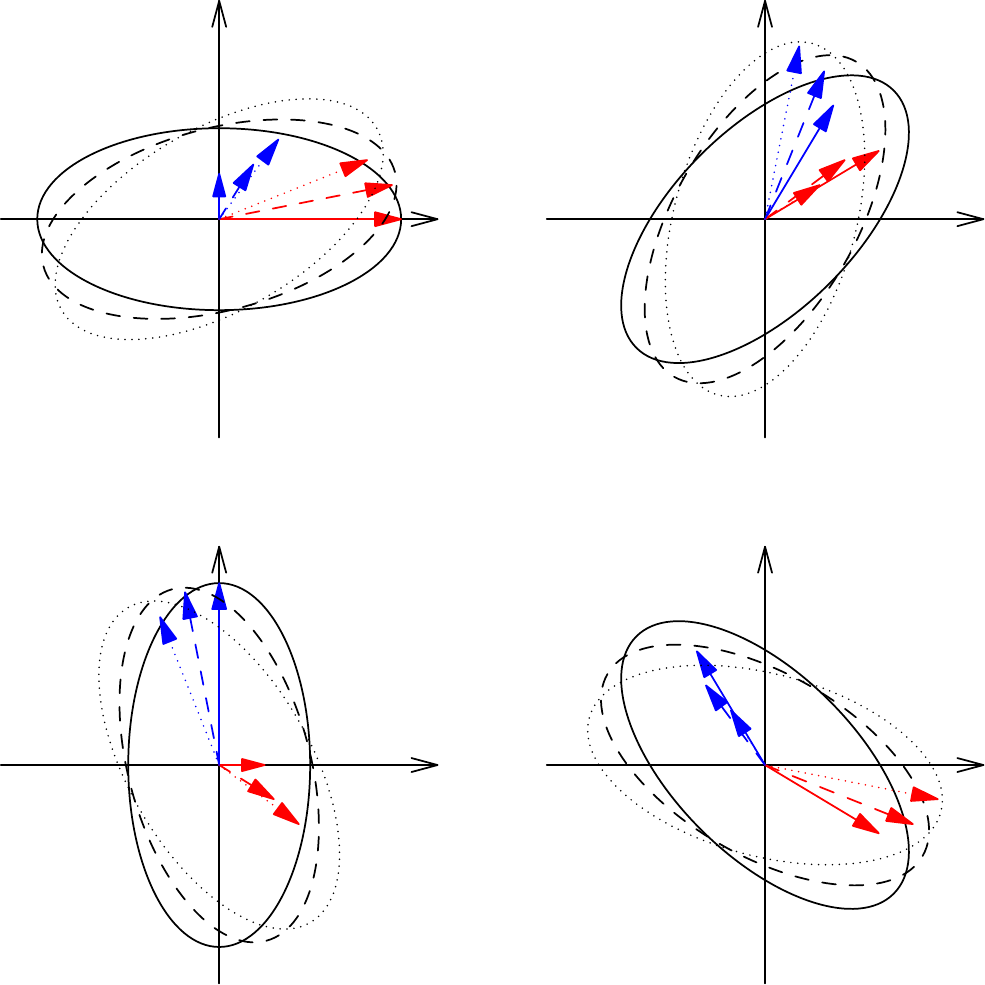}
\caption{
Shown are
the vectors $k_1$ (red) and $k_2$ (blue) corresponding to rotated
versions of the inner product $\langle u,v \rangle = v^\top Q u$
where $Q = \operatorname{diag}\{1,4\}$. The magnitude and angle
of $k_1$ and $k_2$ are plotted in Figure~\ref{fig:Kgraph}.
}
\label{fig:Kpic}
\end{figure}

\begin{figure}[h]
\centering
\includegraphics[width=\linewidth]{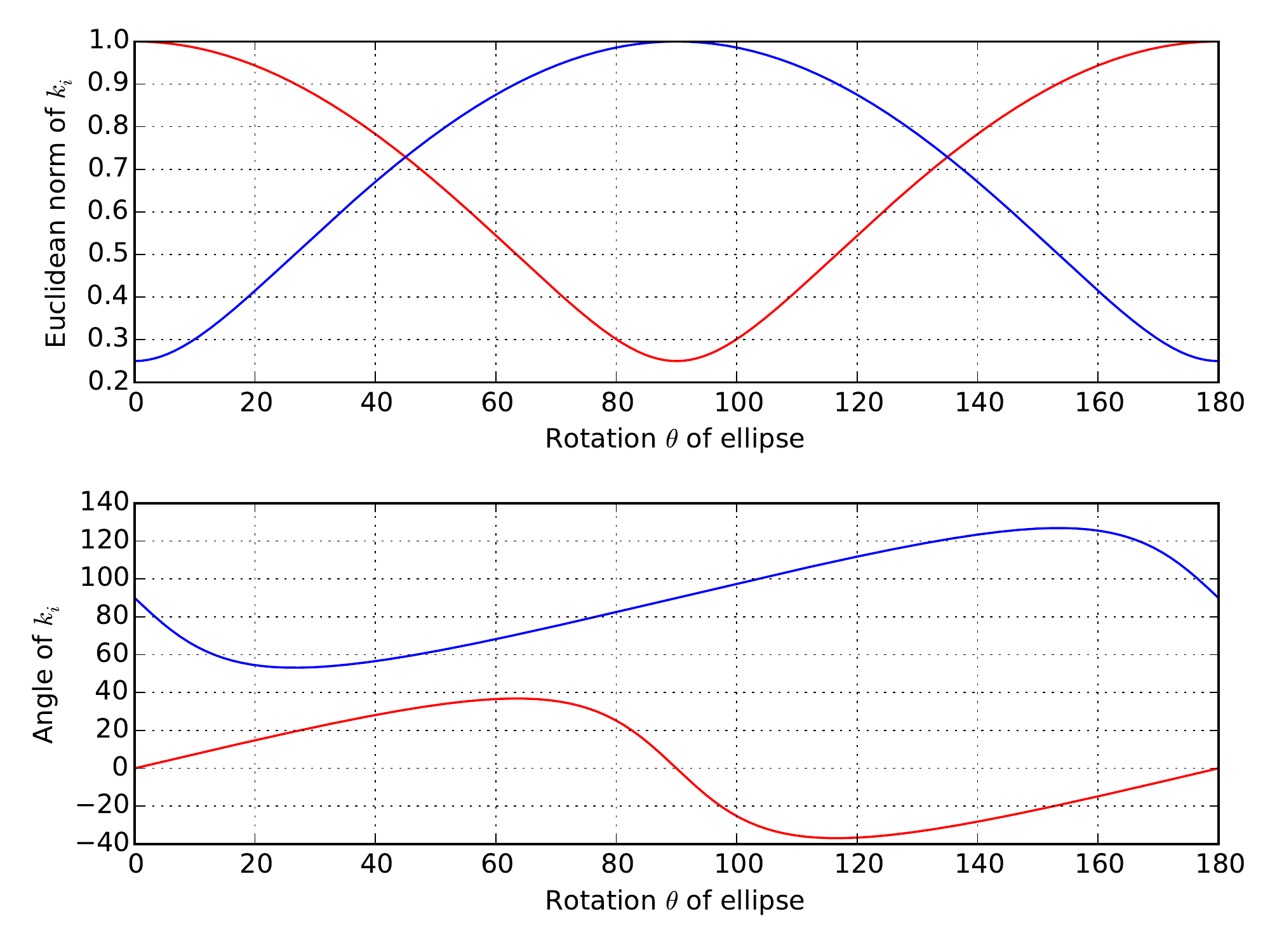}
\caption{
Plotted are the magnitude and angle of $k_1$ (red) and $k_2$ (blue)
corresponding to rotated
versions of the inner product $\langle u,v \rangle = v^\top Q u$
where $Q = \operatorname{diag}\{1,4\}$, as in Figure~\ref{fig:Kpic}.
}
\label{fig:Kgraph}
\end{figure}

Figure~\ref{fig:K1in2} illustrates the kernel of various one-dimensional
subspaces in $\reals^2$. As $V$ is one-dimensional, $k_1$ and $k_2$
must be linearly dependent since they span $V$. Together,
$k_1$ and $k_2$ describe not only $V$ and its inner product, but
also the orientation of $V$ in $\reals^2$.

\begin{figure}[h]
\centering
\includegraphics[width=0.7\linewidth]{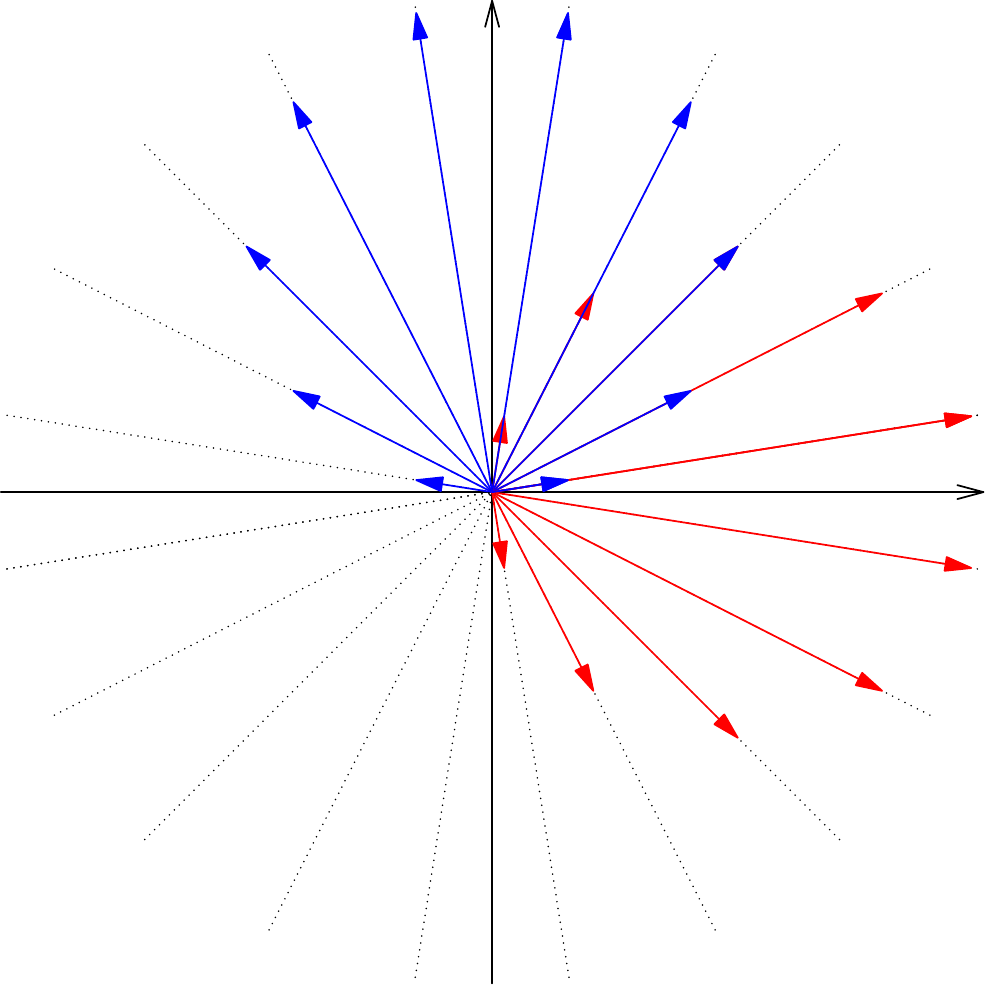}
\caption{
Shown are $k_1$ (red) and $k_2$ (blue) for various one-dimensional
subspaces (black) of $\reals^2$. In all cases, the inner product
is the standard Euclidean inner product. Although not shown,
$k_2$ is zero when $V$ is horizontal, and $k_1$ is zero
when $V$ is vertical. Therefore, the magnitude of the red vectors 
increases from zero to a maximum then decreases back
to zero. The same occurs for the blue vectors.
}
\label{fig:K1in2}
\end{figure}

Figure~\ref{fig:K2to1} portrays how a sequence of two-dimensional
subspaces can converge to a one-dimensional subspace.  From the
perspective of the inner product, the convergence is visualised by
an ellipse degenerating to a line segment.  From the perspective of the
kernel $K = [k_1\ k_2]$, the convergence is visualised by $k_1$ and
$k_2$ individually converging to vectors lying in the same subspace.

\begin{figure}[h]
\centering
\includegraphics[width=0.7\linewidth]{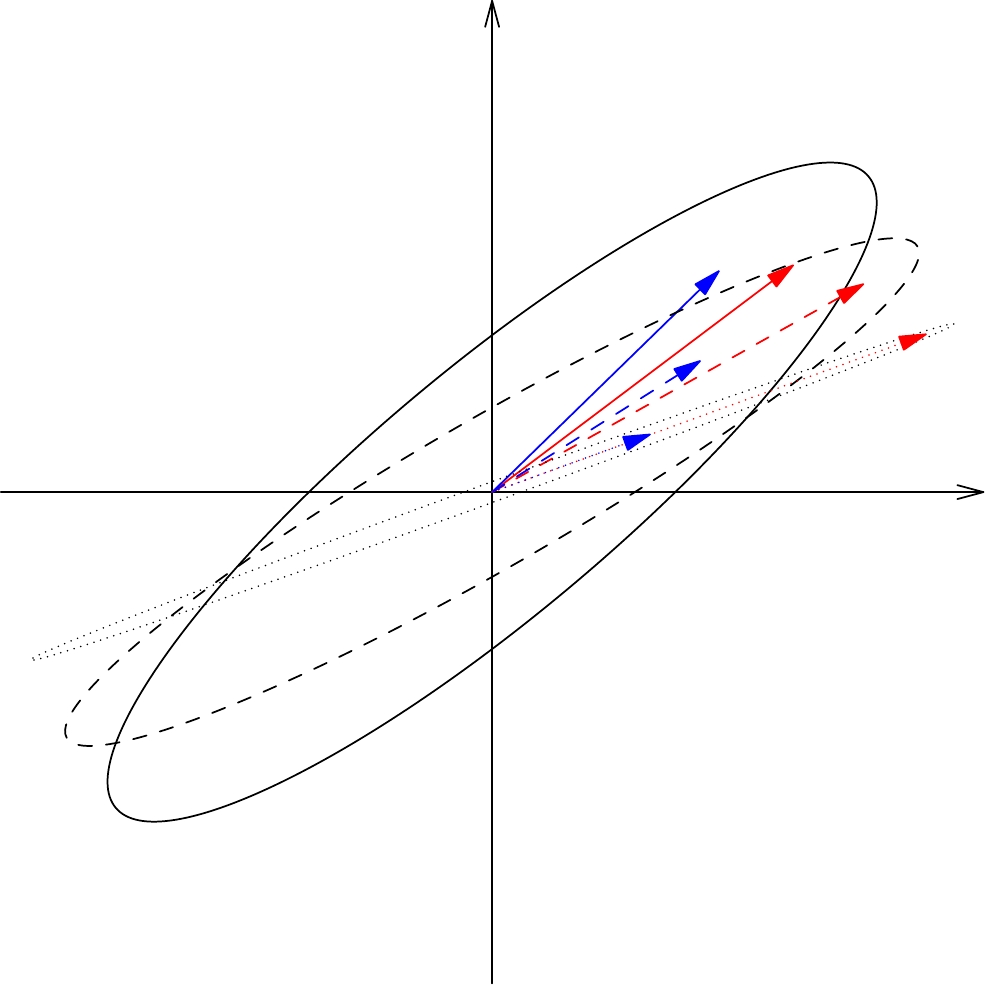}
\caption{
Illustration of how, as the ellipse gets narrower, the two-dimensional
inner product space $V = \reals^2$ converges to a one-dimensional
inner product space. The kernels of the subspaces are represented
by red ($k_1$) and blue ($k_2$) vectors.
}
\label{fig:K2to1}
\end{figure}

The figures convey the message that the $k_i$ are a spanning set
for $V$ that vary continuously with changing $V$. Precisely how
they change is of more algebraic than geometric importance; the
property $\langle k_j, k_i \rangle = K_{ij}$ is very convenient to
work with algebraically.

\section{RKHSs over the Complex Field}

RKHS theory extends naturally to complex-valued vector spaces. In
finite dimensions, this means considering subspaces $V \subset
\cx^n$ where $V$ is equipped with a complex-valued inner product
$\langle \cdot, \cdot \rangle \colon V \times V \rightarrow \cx$.
The only change to Definition~\ref{def:kernel} is replacing
$K=u_1u_1^\top + \cdots + u_ru_r^\top$ by $K=u_1u_1^H + \cdots + u_ru_r^H$
where $H$ denotes Hermitian transpose.

The kernel $K$ will always be Hermitian $(K = K^H)$ and positive
semi-definite.  Swapping the order of an inner product introduces
a conjugation: $\langle u, v \rangle = \overline{ \langle v, u
\rangle }$. Beyond such minor details, the real-valued and
complex-valued theories are essentially the same.

% ----------------------------------------------------
%                       CHAPTER
% ----------------------------------------------------

\chapter{Function Spaces}

Typifying infinite-dimensional vector spaces are function spaces.
Given an arbitrary set $X$, let $\reals^X = \{ f \colon X \rightarrow \reals\}$
denote the set of all functions from $X$ to $\reals$. It is usually given a
vector space structure whose simplicity belies its usefulness: vector-space
operations are defined \emph{pointwise}, meaning scalar multiplication $\alpha
\cdot f$ sends the function $f$ to the function $g$ given by $g(x) =
\alpha\,f(x)$, while vector addition sends $f+g$ to the function $h$ defined by
$h(x) = f(x) + g(x)$.

The space $\reals^X$ is often too large to be useful on its own, but it contains
useful subspaces, such as the spaces of all continuous, smooth or analytic
functions.  (This necessitates $X$ having an appropriate structure; for example,
it does not make sense for a function $f\colon X \rightarrow \reals$ to be
continuous unless $X$ is a topological space.)  Note that to be a subspace, as
opposed to merely a subset, the restrictions placed on the functions must be
closed with respect to vector-space operations: if $f$ and $g$ are in the subset
and $\alpha$ is an arbitrary scalar then $f+g$ and $\alpha\,f$ must also be in
the subset.  While the set of all continuous functions $f \colon [0,1]
\rightarrow \reals$ on the interval $[0,1] \subset \reals$ is a vector subspace
of $\reals^{[0,1]}$, the set of all functions $f \colon [0,1] \rightarrow
\reals$ satisfying $f(1) \leq 1$ is not.

\section{Function Approximation}

Function approximation is an insightful example of how function
spaces facilitate geometric reasoning for solving algebraic
problems~\cite{Luenberger:1969wv}.

\begin{example}
\label{ex:Fapprox}
Denote by $C_2[0,1]$ the space of continuous functions
$f \colon [0,1] \rightarrow \reals$ equipped with the inner product
$\langle f, g \rangle = \int_0^1 f(x) g(x)\,dx$.
The square of the induced norm is $\|f\|^2 = \langle f, f \rangle =
\int_0^1 \left[ f(x) \right]^2\,dx$.
Let $V$ denote the subspace of $C_2[0,1]$ spanned by the functions
$g_0(x) = 1$, $g_1(x)=x$ and $g_2(x)=x^2$.
In words, $V$ is the space of all polynomials of degree at most two.
Fix an $f$, say, $f(x) = e^x-1$. Consider finding $g \in V$ that
minimises $\| f - g \|^2$.
Figure~\ref{fig:Fapprox} visualises this in two ways.
Graphing $f$ does not help in finding $g$, whereas treating
$f$ and $g$ as elements of a function space suggests $g$ can be
found by solving the linear equations $\langle f-g, v \rangle = 0$
for all $v \in V$.

An orthogonal basis for $V$ is $p_0(x) = 1$, $p_1(x) = 2x-1$
and $p_2(x) = 6x^2 - 6x + 1$. (These are the shifted
Legendre polynomials.) Note $\|p_i\|^2 = (2i+1)^{-1}$.
Write $g$ as $g = \alpha_0 p_0 + \alpha_1 p_1 + \alpha_2 p_2$.
Attempting to solve $\langle f-g,v \rangle = 0$ for $v \in \{p_0, p_1, p_2\}$
leads to the equations $\langle \alpha_i p_i, p_i \rangle = \langle f, p_i
\rangle$. Since
$\langle f, p_0 \rangle = e-2$,
$\langle f, p_1 \rangle = 3-e$ and
$\langle f, p_2 \rangle = 7e-19$,
the coefficients $\alpha_i$ are found to be
$\alpha_0 = e-2$, $\alpha_1 = 9 - 3e$ and $\alpha_2 = 35e-95$.
\end{example}

\begin{figure}[h]
\centering
\includegraphics[width=0.7\linewidth]{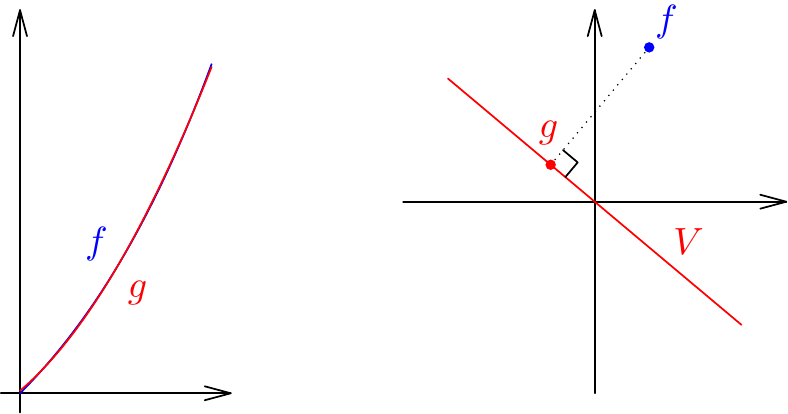}
\caption{
On the left is the direct way of visualising the function approximation
problem: given a function $f$ (blue), find a function $g$ (red)
that best approximates $f$, where $g$ is constrained to
a subclass of functions. Shown here is the second-order polynomial
approximation of $f(x) = e^x-1$ found in Example~\ref{ex:Fapprox}.
It approximates $f$ well, in that the graphs of $f$ and $g$
overlap.  On the right is a geometric visualisation of the
function approximation problem, made possible by representing
functions not by their graphs (left) but as elements of a
vector space (right). The subspace $V$ is the subclass of functions
in which $g$ must lie. Minimising the square of the norm $\|f-g\|$ means
finding the point on $V$ that is closest to $f$. As the norm
comes from an inner product, the point $g$ must be such that $f-g$ is
perpendicular to every function in $V$.
}
\label{fig:Fapprox}
\end{figure}

The simple optimisation problem in Example~\ref{ex:Fapprox} could have been
solved using calculus by setting to zero the derivatives of $\| f - \sum_i
\alpha_i p_i \|^2$ with respect to the $\alpha_i$.  Importantly though, setting
the derivatives to zero in general does not guarantee finding the optimal
solution because the optimal solution need not
exist\footnote{See~\cite{sussmann1997300} for a historical account relating to
existence of optimal solutions.}. When applicable, vector space methods go
beyond calculus in that they can guarantee the optimal solution has been found. 
(Calculus examines local properties whereas being a global optimum is a global
property.)

Infinite-dimensional inner product spaces can exhibit behaviour not
found in finite-dimensional spaces.  The projection of $f$ onto $V$
shown in Figure~\ref{ex:Fapprox} suggesting every minimum-norm
problem has a solution is not necessarily true in infinite dimensions,
as Example~\ref{ex:l2} will exemplify.  Conditions must be imposed
before the geometric picture in Figure~\ref{fig:Fapprox}
accurately describes the infinite-dimensional function approximation
problem.

When coming to terms with infinite-dimensional spaces, one of the
simplest (Hilbert) spaces to contemplate is $l^2$, the space of
square-summable sequences. Elements of $l^2$ are those sequences
of real numbers for which the sum of the squares of each term of
the sequence is finite.  In symbols, $x = (x_1,x_2,\cdots)$ is in
$l^2$ if and only if $\sum_{i=1}^\infty x_i^2 < \infty$. For example,
$(1,\frac12,\frac14,\frac18,\cdots)$ is in $l^2$.  The inner product
on $l^2$ is implicitly taken to be $\langle x, y \rangle =
\sum_{i=1}^\infty x_iy_i$. Note that the square of the induced norm
is $\|x\|^2 = \sum_{i=1}^\infty x_i^2$ and hence the condition for
$x$ to be in $l^2$ is precisely the condition that it have finite
$l^2$-norm.

Approximating an element of $l^2$ by an element of a subspace $V
\subset l^2$ can be considered a function approximation problem
because there is a canonical correspondence between $l^2$ and a
subclass of functions from $\{1,2,\cdots\}$ to $\reals$, just as
there is a canonical correspondence between $R^n$ and functions
from $\{1,\cdots,n\}$ to $\reals$. (The subclass comprises those
functions $f$ for which $\sum_{i=1}^\infty f(i)^2 < \infty$.)

\begin{example}
\label{ex:l2}
Let $V \subset l^2$ be the vector space of sequences with only a
finite number of non-zero terms, thought of as a subspace of $l^2$.
There is no element $g$ of $V$ that minimises $\| f-g \|^2$ when
$f = (1,\frac12,\frac14,\cdots) \in l^2$. Given any approximation
$g$ of $f$, it can always be improved by changing a zero term in
$g$ to make it match the corresponding term in $f$. A rigorous proof
goes as follows. Assume to the contrary there exists a $g \in V$
such that $\| f-\tilde g \| \geq \| f - g \|$ for all $\tilde g \in
V$.  Let $i$ be the smallest integer such that $g_i = 0$. It 
exists because only a finite number of terms of $g$ are non-zero.
Let $\tilde g$ be a replica of $g$ except for setting its $i$th
term to $\tilde g_i = f_i$. Let $e$ be the zero sequence except for
the $i$th term which is unity. Then $f-g = (f-\tilde g) + f_i e$.
Since $\langle f - \tilde g, e \rangle = 0$, $\| f - g \|^2 = \| f
- \tilde g\|^2 + f_i^2 \|e\|^2$, implying $\| f - \tilde g \| < \|
f - g \|$ and contradicting the existence of an optimal element $g$.
\end{example}

Reconciling the \emph{algebra} in Example~\ref{ex:l2}
proving the non-existence of an optimal approximation, with the
\emph{geometry} depicted in Figure~\ref{fig:Fapprox} suggesting
otherwise, requires understanding the interplay between
vector space structures and topological structures induced
by the norm.

Before delving into that topic, it is penetrating to ascertain whether
RKHS theory may be beneficially applied to the function approximation
problem.  A novel feature of this primer is considering RKHS theory
from the ``dynamic'' perspective of how $K$ changes in response to
changing $V$. The solution of the interpolation problem in
Examples~\ref{ex:Interp} and~\ref{ex:Interp2} is continuous in $K$,
and in particular, a solution can be found even for a lower-dimensional
$V \subset \reals^n$ by working in $\reals^n$ and taking limits as
the inner product becomes degenerate, much like in Figure~\ref{fig:K2to1}.
The function approximation problem lacks this property because
enlarging $V$ to include $f$ will mean the optimal solution is $g=f$
no matter how close the inner product is to being degenerate. Since
the solution of the function approximation problem does not Respect
Topology it is unlikely RKHS theory will be of assistance.  (It is
irrelevant that the space used in Example~\ref{ex:Fapprox} is not
a RKHS; the function approximation problem could have been posed
equally well in a genuine RKHS.)

No theory is a panacea for all problems. RKHS theory has broad but
not universal applicability.

\section{Topological Aspects of Inner Product Spaces}
\label{sec:closed}

Reproducing kernel Hilbert spaces have additional properties that more
general Hilbert spaces and other inner product spaces do not necessarily enjoy.
These properties were chosen to make RKHSs behave more like finite-dimensional
spaces.  Learning what can ``go wrong'' in infinite dimensions helps put RKHS
theory in perspective.

Some subspaces should not be drawn as sharp subspaces, as in
Figure~\ref{fig:Fapprox}, but as blurred subspaces denoting the existence of
points not lying in the subspace yet no more than zero distance away from it.
The blur represents an infinitesimally small region as measured by the norm, but
need not be small in terms of the vector space structure. Every point in $l^2$
is contained in the blur around the $V$ in Example~\ref{ex:l2} despite $V$ being
a very small subspace of $l^2$ in terms of cardinality of basis vectors.  (A
Hamel basis for $l^2$ is uncountable while a Hamel basis for $V$ is countable.
Note that in an infinite-dimensional Hilbert space, every Hamel basis is
uncountable.)

Since the blur is there to represent the geometry coming from the
norm, the blur should be drawn as an infinitesimally small
region around $V$. An accurate way of depicting $V \subset l^2$
in Example~\ref{ex:l2} is by using a horizontal plane for $V$ and
adding a blur of infinitesimal height to capture the rest of $l^2$.
This visually implies correctly that there is no vector in $l^2$
that is orthogonal to every vector in $V$. Any $f \in l^2$ that is
not in $V$ makes an infinitesimally small angle with $V$, for
otherwise the vectors $f, 2f, 3f, \cdots$ would move further and
further away from $V$.

As these facts are likely to generate more questions than answers,
it is prudent to return to the beginning and work slowly towards
deriving these facts from first principles.

An inner product induces a norm, and a norm induces a topology.
It suffices here to understand a topology as a rule for determining
which sequences converge to which points. A norm coming from an inner
product determines the inner product uniquely.  It is pragmatic to
think of the norm as the more dominant structure.  The presence of
an inner product means the square of the norm is ``quadratic'', a
very convenient property.

Inner product spaces therefore have two structures: a vector space
structure and a norm whose square is ``quadratic''. The axioms of
a normed space ensure these two structures are compatible.
For example, the triangle inequality $\| x + y \| \leq \|x\|
+ \|y\|$ involves both the norm and vector addition.  Nevertheless,
the two structures capture different aspects that are conflated in
finite dimensions.

A norm induces a topology, and in particular, a norm determines
what sequences have what limit points. Topology is a weaker concept
than that of norm because different norms may give rise to the same
topology. In fact, in \emph{finite} dimensions, every norm gives
rise to the same topology as every other norm: if $x_k \rightarrow
x$ with respect to one norm $\| \cdot \|$, meaning $\| x_k - x \|
\rightarrow 0$, then $x_k \rightarrow x$ with respect to any other
norm.

The first difference then in infinite dimensions is that different
norms can induce different topologies.

\begin{example}
Let $\|\cdot\|_2$ denote the standard norm on $l^2$ and let
$\|x\|_\infty = \sup_i |x_i|$ be an alternative norm.  Let $x^{(i)}$
denote the sequence $(\frac1i,\cdots,\frac1i,0,0,\cdots)$ whose
first $i^2$ terms are $\frac1i$ and the rest zero.  Then
$x^{(1)},x^{(2)},\cdots$ is a sequence of elements of $l^2$.  Since
$\| x^{(i)} \|_\infty = \frac1i \rightarrow 0$, the sequence converges
to the zero vector with respect to the norm $\|\cdot\|_\infty$.
However, $\| x^{(i)} \|_2 = 1$ does not converge to zero and hence
$x^{(i)}$ does not converge to the zero vector with respect to the
norm $\|\cdot\|_2$. The two norms induce different topologies.
\end{example}

A second difference is that an infinite dimensional subspace need
not be topologically \emph{closed}.  A subspace $V$ of a normed
space $W$ is closed in $W$ if, for any convergent sequence
$v_1,v_2,\cdots \rightarrow w$, where the $v_i$ are in $V$ but $w$
need only be in $W$, it is nevertheless the case that $w$ is in $V$.
(Simply put, $V$ is closed if every limit point of every sequence in $V$ is
also in $V$.)  The subspace $V$ in Example~\ref{ex:l2} is not closed
because the sequence $(1,0,\cdots)$, $(1,\frac12,0,\cdots)$,
$(1,\frac12,\frac14,0, \cdots)$ is in $V$ but its limit
$(1,\frac12,\frac14,\frac18,\cdots) \in l^2$ is not in $V$.

A closed subspace can be drawn as a sharp subspace but any other
subspace should have an infinitesimal blur around it denoting the
existence of elements not in the subspace but which are infinitesimally
close to it, a consequence of $v_1,v_2,\cdots \rightarrow w$ meaning
$\| v_i - w \| \rightarrow 0$.  

In finite dimensions, if $V \subset \reals^n$ and $v_1,v_2,\cdots$
is a sequence of elements of $V$ converging to a point $v \in
\reals^n$, it is visually clear from a diagram that the limit point
$v$ must lie in $V$.

\begin{lemma}
Every finite-dimensional subspace $V \subset W$ of an inner product 
space $W$ is closed.
\end{lemma}
\begin{proof}
Let $v_1,v_2,\cdots \rightarrow w$ be a convergent sequence,
with $v_i$ in $V$ and $w$ in $W$.
Let $u_1,\cdots,u_r$ be an orthonormal basis for $V$. 
Let $\tilde w = w - \langle w, u_1 \rangle
u_1 - \cdots - \langle w, u_r \rangle u_r$.
\emph{Here, $w-\tilde w$ is the projection of $w$ onto $V$.
In finite dimensions, the projection is always well-defined.}
Then $\tilde w$ is orthogonal to $u_1,\cdots,u_r$. 
By Pythagoras' theorem, $\| \tilde w \| \leq \| w - v_i \|$.
By definition, $v_i \rightarrow w$ implies $\| w - v_i \| \rightarrow 0$,
therefore, $\| \tilde w \|$ must equal zero, implying $w$ lies in $V$.
\end{proof}

A key ingredient in the above proof is that a vector $w \notin V$
can be projected onto $V$ and hence is a non-zero distance away
from $V$.  In Example~\ref{ex:l2} though, there is no orthogonal
projection of $f$ onto $V$.

\begin{lemma}
Define $f$ and $V$ as in Example~\ref{ex:l2}. There is no vector
$\tilde f \in V$ such that $\langle f-\tilde f, v \rangle = 0$ for
all $v \in V$.
\end{lemma}
\begin{proof}
Let $\tilde f \in V$ be arbitrary. Let $i$ be the smallest integer
such that the $i$th term of $\tilde f$ is zero. Then the $i$th term
of $f-\tilde f$ is non-zero and the inner product of
$f - \tilde f$ with the element $(0,\cdots,0,1,0,\cdots,0) \in V$
having 1 for its $i$th term, is non-zero.
\end{proof}

The space $l^2$, as a vector space, is larger than it looks. The
sequences $x^{(i)} = (0,\cdots,0,1,0,\cdots,0)$, with the unit
appearing in the $i$th term, do not form a (Hamel) basis for $l^2$.
The $x^{(i)}$ span $V$ in Example~\ref{ex:l2} but do not come close
to spanning $l^2$. The $x^{(i)}$ are countable in number whereas a
basis for $l^2$ must necessarily be uncountable. This stems from
the span of a set of vectors being the vectors representable by a
\emph{finite} linear combination of vectors in the set.  Elements
in $l^2$ having infinitely many non-zero terms cannot be written
as finite linear combinations of the $x^{(i)}$.  (Since infinite
summations involve taking limits, only finite linear combinations
are available in spaces with only a vector space structure.  A prime
mathematical reason for introducing a topology, or \emph{a fortiori}
a norm, is to permit the use of infinite summations.)

From the norm's perspective though, $l^2$ is not large: any point
in $l^2$ is reachable as a limit of a sequence of points in the
span of the $x^{(i)}$. Precisely, $(f_1,f_2,\cdots) \in l^2$ is the
limit of $(f_1,0,\cdots)$, $(f_1,f_2,0,\cdots)$, $\cdots$.  This
shows \emph{the norm is coarser than the vector space structure.}
The vector space structure in Example~\ref{ex:l2} places $f$ away
from $V$ yet the norm places $f$ no more than zero distance from
$V$.

To demonstrate the phenomenon in Example~\ref{ex:l2} is not esoteric,
the following is a more familiar manifestation of it. Once again,
the chosen subspace $V$ is not closed.

\begin{example}
Define $C_2[0,1]$ as in Example~\ref{ex:Fapprox}, but this time
take $V$ to be the subspace of all polynomials on the unit interval
$[0,1]$. Then there is no polynomial $g \in V$ that is closest to
$f(x) = \sin(x)$ because $f$ can be approximated arbitrarily
accurately on $[0,1]$ by including more terms of the Taylor series
approximation $\sin(x) = 1-\frac{x^3}{3!} + \frac{x^5}{5!} - \cdots$.
\end{example}

To summarise, it is sometimes possible to extend an inner product
on an infinite dimensional space $V$ to an inner product on $V
\oplus \operatorname{span}\{f\}$ so that there is a
sequence in $V$ converging to $f$. Here, $f$ is a new vector being
added to form a larger space.  The affine spaces $cf + V$ for $c
\in \reals$ would then be infinitesimally close to each other,
crammed one on top of the other.  Similarly, $l^2$ is obtained from
the $V$ in Example~\ref{ex:l2} by adding basis vectors that are all
crammed flat against $V$, hence the earlier suggestion of visualising
$l^2$ as an infinitesimal blur around $V$.

\section{Evaluation Functionals}

Evaluation functionals play a crucial role in RKHS theory. In finite
dimensions, they went unnoticed as the linear functionals $v \mapsto
e_i^\top v$.

Let $V$ be a subspace of $\reals^X$ where $X$ is an arbitrary set.
An element $f$ of $V$ is therefore a function $f \colon X \rightarrow
\reals$ and can be evaluated pointwise. This means mathematically
that for any $x \in X$, there is a function $l_x \colon V \rightarrow
\reals$ given by $l_x(f) = f(x)$. To become familiar with this
definition, the reader is invited to verify $l_x$ is linear.

If there is a norm on $V$ then it can be asked whether the norm
``controls'' the pointwise values of elements of $V$. The following
examples are intended to convey the concept prior to a rigorous
definition.

\begin{example}
\label{ex:tent}
Define $C_2[0,1]$ as in Example~\ref{ex:Fapprox}.  Let $f(x)$ be
the tent function given by $f(x) = 8x+4$ for $x \in [-\frac12,0]$,
$f(x) = 4-8x$ for $x \in [0,\frac12]$ and $f(x) = 0$ otherwise.
Let $f_i(x) = i\,f\!\left(i^2(x-\frac12)\right)$ for $x \in [0,1]$.
Then $\| f_i \| = \frac1i$ decreases to zero yet $f_i(\frac12)$
diverges to infinity.  No matter how small the norm is, there are
elements of $C_2[0,1]$ taking on arbitrarily large values pointwise.
\end{example}

The next example uses a version of the Cauchy-Schwarz inequality:
\begin{equation}
\label{eq:CSi}
\left( \int_a^b g(t)\,dt \right)^2 \leq (b-a)\,\int_a^b g(t)^2\,dt.
\end{equation}
It is valid for any real-valued $a$ and $b$, not just when $a < b$.

\begin{example}
Let $C^1[0,1]$ be the space of continuously differentiable functions.
(If $f \in C^1[0,1]$ then $f'(x)$ exists and is continuous.) Define
a norm by $\|f\|^2 = \int_0^1 f(x)^2\,dx + \int_0^1 f'(x)^2\,dx$.
It is shown below that $\max_x | f(x) | \leq 2 \| f \|$. The
norm controls the values of $f(x)$ for all $x$.

Let $c = \max_x |f(x)|$. By considering $-f$ instead of $f$ if
necessary, it may be assumed that $c = \max_x f(x)$. If $f(x) \geq
\frac{c}2$ for all $x$ then $\|f\| \geq \frac{c}2$ and thus $\max_x
| f(x) | \leq 2 \| f \|$.  Alternatively, there must exist a $y \in
[0,1]$ such that $f(y)=\frac{c}2$.  Let $x \in [0,1]$ be such that
$f(x)=c$.  Then $2\int_y^x f'(t)\,dt = c$ because $f(x) = f(y) +
\int_y^x f'(t)\,dt$.  Applying (\ref{eq:CSi}) with $g(t)=f'(t)$
yields $c^2 \leq 4(x-y)\int_y^x f'(t)^2\,dt$, from which it follows
that $c^2 \leq 4 \int_0^1 f'(t)^2\,dt$, that is, $\|f\| \geq
\frac{c}2$.  Thus, $\max_x | f(x) | \leq 2 \| f \|$.
\end{example}

Linear functionals, such as $l_x$, fall into two categories, bounded
and unbounded, based on whether $c_x = \sup_{f \in V} |l_x(f)| \cdot
\| f \|^{-1}$ is finite or infinite. If $l_x$ is bounded then
$|l_x(f)| \leq c_x\,\|f\|$ for all $f$, demonstrating that the norm
being small implies $f(x) = l_x(f)$ is small in magnitude.

A key requirement for $V \subset \reals^X$ to be a RKHS is for the
$l_x$ to be bounded for all $x$. This implies that if $f_n \in V$
is a sequence converging in norm, meaning $\| f_n - f \| \rightarrow
0$, then the sequence also converges pointwise, meaning $f_n(x)
\rightarrow f(x)$.  This requirement was not mentioned earlier
because $l_x$ is automatically bounded when $V$ is finite dimensional.

% ----------------------------------------------------
%                       CHAPTER
% ----------------------------------------------------

\chapter{Infinite-dimensional RKHSs}
\label{sec:idRKHS}

The finite-dimensional RKHS theory introduced earlier suggests the
general theory should take the following form. Fix an arbitrary set
$X$. Endow a subspace $V$ of the function space $\reals^X$ with an
inner product.  RKHS theory should associate with the tuple
$(V,\langle\cdot,\cdot\rangle,X)$ a set of vectors in $\reals^X$
that span $V$, that vary continuously as $(V,\langle\cdot,\cdot\rangle)$
changes, and that encode the inner product on $V$. Furthermore, by
analogy with Definition~\ref{def:kernel}, an appropriate set should
be the vectors $k_x$ satisfying $\langle f, k_x \rangle = f(x)$ for
all $f \in V$ and $x \in X$.

The above can be achieved with only several minor adjustments.
Requiring the span of the $k_x$ to be $V$ is too demanding.  For
example, a basis for $l^2$ would necessarily be uncountable and cannot even be
constructed explicitly, whereas the countable set of vectors $(1,0,\cdots)$,
$(0,1,0,\cdots)$, $\cdots$ suffices for working with $l^2$ if taking limits is
acceptable.  Although seemingly wanting to reconstruct $V$ by taking the
topological closure of the span of the $k_x$, this would necessitate endowing
$\reals^X$ with a topology, and there might not be a topology on $\reals^X$
compatible with all possible choices of $(V,\langle\cdot,\cdot\rangle)$.
Instead, the mathematical process of completion (\S\ref{sec:completions}) can be
used to achieve the same aim of adding to the span of the $k_x$ any
additional vectors that ``should be'' infinitesimally close to the $k_x$. A
Hilbert space is an inner product space that is complete, hence the above can be
summarised by saying RKHS theory requires $(V,\langle\cdot,\cdot\rangle)$ to be
a Hilbert space.

There may not exist a $k_x$ such that $\langle f, k_x \rangle =
f(x)$ for all $f \in V$ because the linear functional $f \mapsto
\langle f, k_x \rangle$ is always bounded whereas Example~\ref{ex:tent}
showed $f \mapsto f(x) = l_x(f)$ need not be bounded. Another
advantage of requiring $V$ to be a Hilbert space is that any bounded
linear functional can be written as an inner product. (This is the
Riesz representation theorem.) Therefore, a necessary and sufficient
condition for $\langle f, k_x \rangle = f(x)$ to have a solution
$k_x$ is for $l_x$ to be bounded. Requiring $l_x$ to be bounded for
all $x \in X$ is the second requirement RKHS
theory places on $(V,\langle\cdot,\cdot\rangle)$.

RKHSs are similar to Euclidean spaces.
A feature of Euclidean space is the presence of coordinate functions
$\pi_i\colon \reals^n \rightarrow \reals$ sending $(x_1,\cdots,x_n)$
to $x_i$. These coordinate functions are continuous.  A
RKHS replicates this: if $V \subset \reals^X$ is a RKHS, the
coordinate functions $\pi_x \colon V \rightarrow \reals$ sending
$f$ to $f(x)$ are continuous by definition. (Recall that a linear functional is
continuous if and only if it is bounded.)

\section{Completions and Hilbert Spaces}
\label{sec:completions}

If $V \subset W$ is not closed (\S\ref{sec:closed}) then there are points
in $W$ not in $V$ but infinitesimally close to $V$. The closure
$\bar V$ of $V$ in $W$ is the union of these infinitesimally close
points and $V$.  Perhaps a new $W$ can be found though for
which $\bar V \subset W$ is not closed?

Given a normed space $V$, there is a unique (up to an isometric
isomorphism) normed space $\hat V \supset V$, called
the \emph{completion} of $V$, such that every point
in $\hat V$ is infinitesimally close to $V$ and there is no normed
space $W$ for which $\hat V \subset W$ is not closed. (Here,
the norm on $W$ must agree with the norm on $\hat V$ which must
agree with the norm on $V$.) This means
there is one and only one way to enlarge $V$ maximally by adding
infinitesimally close points.  Once this is done, no more points
can be added infinitesimally close to the enlarged space $\hat V$.

This fact permits a refinement of the previous visualisation of a
non-closed subspace as being blurred. If $V \subset W$ is not
closed then there are dents on the surface of $V$.  Pedantically,
these dents are straight line scratches since $V$ is a vector space.
Filling in these scratches closes $V$ in $W$. Choosing a different
$W$ cannot re-open these scratches but may reveal scratches elsewhere.
The completion $\hat V$ results from filling in all possible scratches.

Scratches can be found without having to construct a $W$.  A sequence
$v_n$ in a normed space $V$ is \emph{Cauchy} if, for all $\epsilon
> 0$, there exists a positive integer $N$ such that $\| v_n - v_m
\| < \epsilon$ whenever $n,m \geq N$. Every convergent sequence is
a Cauchy sequence, and a sequence that is Cauchy ``should'' converge;
if a Cauchy sequence in $V$ does not converge then it ``points''
to a scratch.

\paragraph*{Remark}
The distance between successive terms of a convergent sequence must
go to zero: $1,\frac12,\frac13,\cdots \rightarrow 0$ implies
$|\frac1n - \frac1{n+1}| \rightarrow 0$.  The converse is false
though: successive terms of $s_n = \sum_{i=1}^n i^{-1}$ become
infinitesimally close but too slowly to prevent $s_n$ from diverging
to infinity.  Hence Cauchy imposed the stronger requirement of
non-successive terms becoming infinitesimally close.

To demonstrate that Cauchy sequences point to scratches, let $v_n
\in V$ be a non-convergent Cauchy sequence. Enlarge $V$ to $V \oplus
\operatorname{span}\{f\}$ where $f$ is a new vector. The aim is to
construct a norm on $V \oplus \operatorname{span}\{f\}$ agreeing
with the norm on $V$ and placing $f$ precisely at the limit of the
Cauchy sequence, that is, $v_n \rightarrow f$.  Extending the norm
requires defining $\| c f + g \|$ for all $c \in \reals$ and $g \in
V$. Since the aim is for $v_n \rightarrow f$, the obvious choice
to try is $\| c f + g \| = \lim_{n \rightarrow \infty} \| c v_n +
g \|$. It can be shown the proposed norm really is a norm; it
satisfies the requisite axioms. That it extends the original norm
is clear; when $c=0$, the new and old norms agree.  Finally,
$v_n \rightarrow f$ because $\lim_n \| f - v_n \| = \lim_n \lim_m
\| v_m - v_n \| = 0$, the last equality due to $v_n$ being Cauchy.

Once a scratch is filled in, it cannot be re-filled due to uniqueness
of limits in normed spaces.  If an attempt was made to place $g \in
V \oplus \operatorname{span}\{f,g\}$ in the location pointed to by
$v_n$, then $v_n \rightarrow g$, yet from above, $v_n \rightarrow
f$, hence $f$ and $g$ must be the same point.

Textbooks explain how $V$ is completed by creating a new space
comprising all Cauchy sequences in $V$ then quotienting by declaring
two Cauchy sequences as equivalent if they converge to the same
limit. This is not explained here since the description above conveys
adequately the underlying principles.

If the norm on $V$ comes from an inner product
then the norm on the completion $\hat V$ also comes from
an inner product, that is, the completion of an inner product
space is itself an inner product space.

An inner product space that is complete (with respect to the norm
induced by the inner product) is called a \emph{Hilbert space}.

A completion is needed for reconstructing $V \subset \reals^X$ from
its kernel $K$ because in general the kernel can describe only a
dense subspace of the original space.  Any space containing this
dense subspace and contained in its completion would have the same
kernel.  A unique correspondence is achieved by insisting $V$ is
complete; given $K$, the space $V$ is the completion of the subspace
described by $K$.

It would be remiss not to mention another reason Cauchy sequences
and completions are important.  Completeness ensures existence
of solutions to certain classes of problems by preventing the
solution from having been accidentally or deliberately removed from
the space.

\begin{example}
Let $x^{(i)} = (0,\cdots,0,1,0,\cdots)$ be the element of $l^2$ having unity as
its $i$th term. Let $f = (1,\frac12,\frac13,\cdots) \in l^2$.  Then $f$ is the
unique solution in $l^2$ to the system of equations $\langle x^{(i)}, f \rangle
= i^{-1}$ for $i=1,2,\cdots$.
Let $V$ be a subspace of $l^2$ omitting $f$ but containing all the $x^{(i)}$,
such as the $V$ in Example~\ref{ex:l2}. Treating $V$ as a vector space, there is
no solution $\tilde f \in V$ to $\langle x^{(i)}, \tilde f \rangle = i^{-1}$ for
$i=1,2,\cdots$.
\end{example}

If a space is complete, an existence proof of a solution
typically unfolds as follows.  Assume the problem
is to prove the existence of a solution $f$ to a differential
equation.  Construct a sequence of approximate solutions $f_n$ by
mimicking how differential equations are solved numerically, with decreasing
step size.
The hope is that $f_n \rightarrow f$, but since $f$ cannot be exhibited
explicity, it is not possible to prove $f_n \rightarrow f$ directly.
Instead, $f_n$ is shown to be Cauchy and therefore, by completeness,
has a limit $\tilde f$. Lastly, $\tilde f$ is verified by some limit
argument to be the solution, thus proving the existence of $f =
\tilde f$. For details, see~\cite{Brown:2004gy}.

Two examples of determining whether a space is complete are now
given.  The proof that $l^2$ is complete is quite standard and
proceeds along the following lines. If $f_1,f_2,\cdots \in l^2$ is
Cauchy then the $i$th term of each $f_j$ form a Cauchy sequence of
real numbers. Since the real numbers are complete --- they are the
completion of the rational numbers --- the $f_j$ converge pointwise
to a limit $f$. The norm of $f$ can be shown to be finite and hence
$f$ is in $l^2$. Finally, it can be shown that $\| f_i - f \|
\rightarrow 0$, proving the $f_i$ have a limit in $l^2$.

The space $C_2[0,1]$ in Example~\ref{ex:Fapprox} is not complete.
This can be seen by considering a sequence of continuous functions
that better and better approximate a square wave. The limit ``should''
exist but a square wave is not a continuous function and hence does
not lie in $C_2[0,1]$. The completion of $C_2[0,1]$ is the space
known as $L_2[0,1]$.  Although $L_2$ spaces are often treated as
if they were function spaces, technically they are not. A meaning
cannot be ascribed to the pointwise evaluation of a function in
$L_2$ because an element of $L_2$ is actually an equivalence class
of functions, any two members of which may differ pointwise on a
set of measure zero. Conditions for when the completion of a function space
remains a function space are given in \S\ref{sec:compfs}.

\section{Definition of a RKHS}

The following definition of a RKHS differs in two inconsequential
ways from other definitions in the literature. Normally, RKHSs over
the complex field are studied because every real-valued RKHS extends
canonically to a complex-valued RKHS.  Often, the evaluation
functionals are required to be continuous whereas here they are
required to be bounded.  A linear
operator is bounded if and only if it is continuous.

\begin{definition}
\label{def:rkhs}
Let $X$ be an arbitrary set and denote by $\reals^X$ the vector space of
all functions $f\colon X \rightarrow \reals$ equipped with
pointwise operations.  A subspace $V \subset
\reals^X$ endowed with an inner product is a \emph{reproducing
kernel Hilbert space (RKHS)} if $V$ is complete and, for every $x
\in X$, the evaluation functional $f \mapsto f(x)=l_x(f)$ on $V$
is bounded.
\end{definition}

The kernel of a RKHS exists and is unique.

\begin{definition}
\label{def:Kfn}
If $V \subset \reals^X$ is a RKHS then its \emph{kernel} is the
function $K\colon X \times X \rightarrow \reals$ satisfying
$\langle f, K(\cdot,y) \rangle = f(y)$ for all $f \in V$ and $y \in
X$. Here, $K(\cdot,y)$ denotes the function $x \mapsto K(x,y)$ and
is an element of $V$. 
\end{definition}

If $V$ is a subspace of $\reals^X$ where $X=\{1,\cdots,n\}$ then
finite dimensionality implies it is complete and its evaluation
functionals bounded. Thus, $V$ is a RKHS by
Definition~\ref{def:rkhs}. The kernel in Definitions~\ref{def:Kfn}
and~\ref{def:kernel} are equivalent: $K(i,j) = K_{ij}$.

\paragraph*{Remark}
The set $X$ in Definition~\ref{def:rkhs} is arbitrary. It need not have a
topology and, in particular, there are no continuity requirements on the
elements of $V$ or the kernel $K$.

\section{Examples of RKHSs}

Examples of RKHSs can be found in~\cite{Kailath:1971hk,
Berlinet:2012uv, Aronszajn:1950tq, Hille:1972wi}, among
other places.

\begin{example}[Paley-Wiener Space]
\label{ex:pws}
Let $V$ consist of bandlimited functions $f
\colon \reals \rightarrow \reals$ expressible as
$f(t) = \frac1{2\pi} \int_{-a}^a F(\omega) e^{\jmath \omega t}\,d\omega$
where $F(\omega)$ is square-integrable.
Endow $V$ with the inner product
$\langle f, g \rangle = \int_{-\infty}^{\infty} f(t) g(t)\,dt$.
Then $V$ is a RKHS with kernel~\cite{Yao:1967tc,
Larkin:1970do}
\begin{equation}
\label{eq:sinc}
K(t,\tau) = \frac{ \sin(a(t-\tau)) }{ \pi(t-\tau) }.
\end{equation}
\end{example}
Using the Fourier transform and its inverse,
the kernel (\ref{eq:sinc}) can be derived as follows
from the requirement $\langle f, K(\cdot,\tau) \rangle = f(\tau)$.
\begin{align*}
f(\tau) &= \frac1{2\pi} \int_{-a}^{a} \left(
    \int_{-\infty}^{\infty} f(t) e^{-\jmath \omega t}\,dt
\right) e^{\jmath \omega \tau}\,d\omega \\
&= \int_{-\infty}^\infty f(t) \left( \frac1{2\pi}
    \int_{-a}^a e^{\jmath\omega(\tau-t)}\,d\omega \right)\,dt \\
&= \int_{-\infty}^\infty f(t) \left( \frac1{2\pi}
    \frac{2}{\tau - t} \sin(a(\tau - t)) \right)\,dt \\
&= \int_{-\infty}^\infty f(t) K(t,\tau) \,dt.
\end{align*}
Completeness of $V$ in Example~\ref{ex:pws} can be proven essentially
as follows.  If $f_n \in V$ is a Cauchy sequence then its Fourier
transform $F_n$ is Cauchy too. Therefore, $F_n$ converges to a
square-integrable function $F$ and $f(t) = \frac1{2\pi} \int_{-a}^a
F(\omega) e^{\jmath \omega t}\,d\omega$ is in $V$. It can be verified
$f_n \rightarrow f$.

\begin{example}
\label{ex:wp}
Let $V$ consist of functions of the form $f\colon [0,1] \rightarrow \reals$
where $f(t)$ is absolutely continuous, its derivative $f'(t)$ (which
exists almost everywhere) is square-integrable, and $f(0)=0$.
Then $V$ is a RKHS when endowed with the inner product
$\langle f,g \rangle = \int_0^1 f'(t) g'(t)\,dt$. The
associated kernel is $K(t,s) = \min\{t,s\}$.
\end{example}

Readers may recognise $K(t,s) = \min\{t,s\}$ as the covariance
function of a Wiener process. If $g(t) = K(t,s)$ then
$g'(t) = 1$ when $t<s$ and $g'(t) = 0$ when $t>s$. Therefore,
\[
\langle f, K(\cdot,s) \rangle = \int_0^s f'(t)\, dt = f(s).
\]

Despite this primer focusing on real-valued functions, the following
classic example of a RKHS comprises complex-valued functions.

\begin{example}[Bergman Space]
\label{ex:berg}
Let $S = \{z \in \mathbb{C} \mid |z|<1\}$ denote the unit disk in
the complex plane. Let $V$ be the vector space of all analytic and
square-integrable functions $f\colon S \rightarrow \mathbb{C}$.
Equip $V$ with the inner product $\langle f, g \rangle = \int_S
f(z) \overline{g(z)}\,dz$. Then $V$ is a RKHS with kernel
\begin{equation}
\label{eq:bergman}
K(z,w) = \frac1\pi\frac1{(1-z\bar w)^2}.
\end{equation}
\end{example}

The above is an example of a Bergman space and its associated kernel.
Choosing domains other than the unit disk produce other Bergman
spaces and associated Bergman kernels~\cite{bergman1970kernel}.
(Generally the Bergman kernel cannot be determined explicitly though.)

\section{Basic Properties}

The fundamental property of RKHSs is the bijective correspondence
between RKHSs and positive semi-definite functions: the kernel of
a RKHS is positive semi-definite and every positive semi-definite
function is the kernel of a unique RKHS~\cite{Aronszajn:1950tq}.

A symmetric function $K \colon X \times X \rightarrow \reals$ is positive
semi-definite (equivalently, of positive type~\cite{Mercer:1909gf})
if, $\forall r \geq 1$, $\forall c_1,\cdots,c_r \in \reals$,
$\forall x_1,\cdots,x_r \in X$,
\begin{equation}
\label{eq:psd}
\sum_{i=1}^r \sum_{j=1}^r c_i c_j K(x_i,x_j) \geq 0.
\end{equation}
Unlike in the complex-valued case when the complex version of (\ref{eq:psd})
forces $K$ to be symmetric~\cite{Berlinet:2012uv}, in the real-valued case, it
is necessary to insist explicitly that $K$ be symmetric: $K(x,y) = K(y,x)$.

The kernel of a RKHS is automatically symmetric;
choosing $f(y) = K(y,x)$ in Definition~\ref{def:Kfn} shows
\begin{equation}
\label{eq:KKK}
\langle K(\cdot,x), K(\cdot,y) \rangle = K(y,x),\qquad
    x,y \in X.
\end{equation}
The symmetry of inner products implies $K(x,y) = K(y,x)$, as claimed.

Confusingly, positive semi-definite functions are often referred to in the
literature as positive definite functions.  The terminology adopted here agrees
with the finite-dimensional case when $K$ is a matrix.

A number of different topologies are commonly placed on function
spaces. The strong topology comes from the norm: the sequence $f_n$
converges strongly to $f$ if $\| f_n - f \| \rightarrow 0$. It
converges weakly if $\langle f_n, g \rangle \rightarrow \langle f,
g \rangle$ for all $g \in V$. (This is weaker because the ``rate
of convergence'' can be different for different $g$.) Pointwise
convergence is when $f_n(x)$ converges to $f(x)$ for all $x \in X$.
In general, the pointwise topology is unrelated to the strong and
weak topologies.  In a RKHS however, strong convergence implies
pointwise convergence, as does weak convergence. Writing $f_n \rightarrow f$
refers to strong convergence.

If $K$ is the kernel of $V \subset \reals^X$, let $V_0 =
\operatorname{span}\{ x \mapsto K(x,y) \mid y \in X\}$.  In words,
$V_0$ is the space spanned by the functions $K(\cdot,y)$ as $y$
ranges over $X$. Clearly $V_0 \subset V$. In the finite-dimensional
case, $V_0$ would equal $V$. In general, $V_0$ is only dense in
$V$. Any $f \in V$ can be approximated arbitrarily accurately
by finite linear combinations of the $K(\cdot,y)$. Moreover, $V$
is the completion of $V_0$ and hence can be recovered uniquely from
$V_0$, where the inner product on $V_0$ is uniquely determined by
(\ref{eq:KKK}). Any Cauchy sequence in $V_0$ converges \emph{pointwise}
to an element of the RKHS $V$.  (It converges strongly by definition
of completion, and strong convergence in a RKHS implies pointwise
convergence.)

\paragraph*{Remark}
Proofs have been omitted partly because they can be found in more
traditional introductory material, including~\cite{Aronszajn:1950tq,
Hille:1972wi, Saitoh:1988vg}, and partly because the above results
should not be surprising once finite-dimensional RKHSs are understood.

\section{Completing a Function Space}
\label{sec:compfs}

If $V$ is incomplete but meets the other conditions in
Definition~\ref{def:rkhs}, a naive hope is for the completion of
$V$ to be a RKHS. This hope is close to the mark: the completion
of $V$ might not be a function space on $X$ but can be made into a
function space on a set larger than $X$.

First a subtlety about completions. Although it is common to speak
of \emph{the} completion $\hat V$ of a normed space $V$, and to
think of $V$ as a subset of $\hat V$, it is more convenient to refer
to any normed space $W$ as a version of the completion of $V$, or
simply, the completion of $V$, provided three conditions are met:
$V$ can be identified with a subspace $V' \subset W$ (that is, $V$
and $V'$ are isometrically isomorphic); every Cauchy sequence in
$V'$ converges to an element of $W$; every element of $W$ can be
obtained in this way.  The usual construction of $\hat V$ produces
a quotient space of Cauchy sequences, where $V$ is identified with
the set $V' \subset \hat V$ consisting of (equivalence classes of)
Cauchy sequences converging to an element of $V$.  The consequence
for RKHS theory is that even though $V \subset \reals^X$ is a
function space, the usual construction produces a completion $\hat
V$ whose elements are not functions.

Whether there exists a version of the completion that is a subset
of $\reals^X$ will be addressed in two parts.  Assuming a RKHS
completion exists, a means for constructing it will be derived,
then necessary and sufficient conditions will be found for the
RKHS completion to exist.

Assume $V_0 \subset \reals^X$ has a completion $V \subset \reals^X$
that is a RKHS containing $V_0$. Let $f_n$ be a Cauchy sequence in
$V_0$. Its limit $f$ can be found pointwise: $f(x) = \langle f,
K(\cdot,x) \rangle = \lim_n \langle f_n, K(\cdot,x) \rangle = \lim_n
f_n(x)$. Hence $V$ can be reconstructed as the set of all functions
on $X$ that are pointwise limits of Cauchy sequences in $V_0$.

This construction can fail in two ways to produce a completion of
an arbitrary $V_0 \subset \reals^X$. The pointwise limit might not
exist, or the pointwise limit might be the same for Cauchy sequences
with distinct limits. Being in a vector space, the latter is
equivalent to the existence of a Cauchy sequence $f_n$ not converging
to zero, meaning $\lim_n \|f_n\| \neq 0$, but whose pointwise limit
is zero: $f_n(x) \rightarrow 0$ for $x \in X$.

\begin{proposition}
\label{pr:comp}
An arbitrary inner product space $V_0 \subset \reals^X$ has a RKHS
completion $V$, where $V_0 \subset V \subset \reals^X$, if and only
if
\begin{enumerate}
\item the evaluation functionals on $V_0$ are bounded;
\item if $f_n \in V_0$ is a Cauchy sequence converging pointwise to
zero then $\|f_n\| \rightarrow 0$.
\end{enumerate}
\end{proposition}
\begin{proof}
Assume $V$ exists.  Condition (1) holds because the evaluation
functionals on $V$ are bounded.  If $f_n \in V_0$ is Cauchy then
$f_n \rightarrow f$ for some $f$ in $V$.  Since strong convergence
implies pointwise convergence, $f(x)=\lim_n f_n(x)$ for all $x \in X$.
Thus, if $f_n$ converges pointwise to zero then its strong limit $f$ must be
zero, in which case $\| f_n \| = \| f - f_n \| \rightarrow 0$, showing condition
(2) holds.

Conversely, assume (1) and (2) hold. Let $f_n \in V_0$ be Cauchy.
The evaluation functionals being bounded implies $f_n(x)$ is Cauchy
for each $x \in X$, hence $f(x) = \lim_n f_n(x)$ exists. Let $V$
be the set of all such pointwise limits $f$.  Consider endowing $V$
with a norm satisfying $\| f \| = \lim_n \| f_n \|$ where $f_n$ is
a Cauchy sequence converging pointwise to $f$.  Condition (2) ensures
different Cauchy sequences converging pointwise to the same limit
yield the same norm: if $f_n(x) \rightarrow f(x)$ and $g_n(x)
\rightarrow f(x)$ then $f_n(x) - g_n(x) \rightarrow 0$, hence $|
\|f_n\| - \|g_n\| | \leq \| f_n - g_n \| \rightarrow 0$. Omitted
are the routine proofs showing $\|\cdot\|$ satisfies the axioms of
a norm and the parallelogram law, hence giving $V$ an inner product
agreeing with the original inner product on $V_0 \subset V$.

To verify $V$ is the completion of $V_0$, let $f_n \in V_0$ be
Cauchy, with pointwise limit $f \in V$.  The sequence $f_m - f_n$
in $m$ is a Cauchy sequence in $V_0$ converging pointwise to $f-f_n$.
Therefore, $\lim_n \| f - f_n \| = \lim_n \lim_m \|f_m - f_n\| =
0$, proving the pointwise limit of a Cauchy sequence in $V_0$ is
also the strong limit.  Therefore, any $f \in V$ can be written as a
pointwise limit and hence as a strong limit of a Cauchy sequence
in $V_0$, and any Cauchy sequence in $V_0$ converges pointwise and
hence in the strong limit to an element of $V$.  \end{proof}

The following example is taken from~\cite[pp.  349--350]{Aronszajn:1950tq}
and shows condition (2) in the above proposition is not automatically
true.  Complex-valued functions are used because analytic functions
are better behaved than real-analytic functions.

Define $S$ and $\langle\cdot,\cdot\rangle$ as in Example~\ref{ex:berg}
and let $\|\cdot\|$ denote the norm induced from the inner product.
The sequence $a_n = 1-(\frac12)^n$ belongs to $S$. The function
(known as a Blaschke product) $f(z) = \prod_{n=1}^\infty \frac{a_n
- z}{1 - a_n z}$ is analytic and bounded on $S$. It satisfies
$f(a_n)=0$ for all $n$. (An analytic function vanishing on a
convergent set of points need itself only be zero if the limit point
is within the domain of definition. Here, the domain is $S$ and
$a_n \rightarrow 1 \notin S$.) There exists a sequence of polynomials
$f_n(z)$ for which $\lim_{n \rightarrow \infty} \| f_n - f \| = 0$.

These ingredients are assembled into an example by defining $X =
\{a_1,a_2,\cdots\} \subset S$ and taking $V_0$ to be the restriction
to $X$ of all polynomials $p\colon S \rightarrow \mathbb{C}$, keeping
the above norm. The $f_n$ above are a Cauchy sequence of polynomials
on $S$ and \emph{a fortiori} on $X$. Pointwise they converge
to zero: $f_n(x) \rightarrow f(x) = 0$ for $x \in S$.
However, $\|f_n\| \rightarrow \|f\| \neq 0$. Condition (2) in
Proposition~\ref{pr:comp} is not met.

Put simply, $X$ did not contain enough points to distinguish the
strong limit of the Cauchy sequence $f_n$ from the zero function.
In this particular example, enlarging $X$ to contain a sequence and
its limit point would ensure condition (2) is met.

This principle holds generally.  Assume $V_0 \subset \reals^X$ fails
to meet condition (2).  Augment $X$ to become $X' = X \sqcup \hat
V$, the disjoint union of $X$ and the completion $\hat V$ of $V_0$.
Extend each $f \in V_0$ to a function $\tilde f \in \reals^{X'}$
by declaring $\tilde f(x) = f(x)$ if $x \in X$ and $\tilde f(x) =
\langle f, x \rangle$ if $x \in \hat V$. (The inner product is the
one on $\hat V$, and $f \in V_0$ is identified with its corresponding
element in $\hat V$.) Let $V_0'$ be the image of this linear embedding
$f \mapsto \tilde f$ of $V_0$ into $\reals^{X'}$.  The inner product
on $V_0$ gets pushed forwards to an inner product on $V_0'$, that
is, $\langle \tilde f, \tilde g \rangle$ is defined as $\langle f,
g \rangle$. To show condition (2) of Proposition~\ref{pr:comp} is
met, let $\tilde f_n \in V_0'$ be a Cauchy sequence converging
pointwise to zero. Then $\tilde f_n(x) \rightarrow 0$ for all $x
\in X'$, and in particular, $\langle f_n, g \rangle \rightarrow 0$
for all $g \in \hat V$. This implies $f=0$ where $f \in \hat V$ is
the limit of $f_n$. Therefore, $\|f_n\| \rightarrow \|f\| = 0$, as
required.

\section{Joint Properties of a RKHS and its Kernel}

The correspondence between a RKHS $V \subset \reals^X$ and its
kernel $K$ induces a correspondence between certain properties of
$V$ and of $K$. Modifying or combining RKHSs to form new ones
can translate into familiar operations on the corresponding
kernels.  A small selection of examples is presented below.

\subsection{Continuity}

If $X$ is a metric space then it can be asked whether all the
elements of a RKHS $V \subset \reals^X$ are continuous.  Since
$K(\cdot,y)$ is an element of $V$, a necessary condition is for the
functions $x \mapsto K(x,y)$ to be continuous for all $y \in X$.
An arbitrary element of $V$ though is a limit of a Cauchy sequence
of finite linear combinations of the $K(\cdot,y)$. An extra
condition is required to ensure this limit is continuous.

\begin{proposition}
The elements of a RKHS $V \subset \reals^X$ are continuous,
where $X$ is a metric space, if and only if
\begin{enumerate}
\item $x \mapsto K(x,y)$ is continuous for all $y \in X$; and
\item for every $x \in X$ there is a radius $r>0$ such that
$y \mapsto K(y,y)$ is bounded on the open ball $B(x,r)$.
\end{enumerate}
\end{proposition}
\begin{proof}
See~\cite[Theorem 17]{Berlinet:2012uv}.
\end{proof}

The necessity of the second condition is proved by assuming there
is an $x \in X$ and a sequence $x_n \in B(x,\frac1n)$ satisfying
$K(x_n,x_n) \geq n$. The functions $K(\cdot,x_n)$ grow without
limit: $\| K(\cdot,x_n) \| = K(x_n,x_n) \rightarrow \infty$.  By a
corollary of the Banach-Steinhaus theorem, a weakly convergent
sequence is bounded in norm, therefore, there exists a $g \in V$
such that $\langle g, K(\cdot,x_n) \rangle \nrightarrow \langle g,
K(\cdot,x) \rangle$.  This $g$ is not continuous: $g(x_n) \nrightarrow
g(x)$.

\subsection{Invertibility of Matrices}
\label{sec:invmat}

Equation (\ref{eq:psd}) asserts the
matrix $A_{ij} = K(x_i,x_j)$ is positive semi-definite.
If this matrix $A$ is singular then there exist constants
$c_i$, not all zero, such that (\ref{eq:psd}) is zero. Since
\begin{align*}
\sum_{i=1}^r \sum_{j=1}^r c_i c_j K(x_i,x_j)
&= \sum_{i=1}^r \sum_{j=1}^r c_i c_j
    \langle K(\cdot,x_j), K(\cdot,x_i) \rangle \\
&= \left\langle \sum_{j=1}^r c_j K(\cdot,x_j), \sum_{i=1}^r c_i K(\cdot,x_i)
    \right\rangle \\
&= \left\| \sum_{i=1}^r c_i K(\cdot,x_i) \right\|^2,
\end{align*}
the matrix $A$ is non-singular if and only if $\sum_{i=1}^r c_i
K(\cdot,x_i) = 0$ implies $c_1=\cdots=c_r=0$.  Now, $\sum_{i=1}^r
c_i K(\cdot,x_i) = 0$ if and only if $\sum_{i=1}^r c_i f(x_i) = 0$
for all $f \in V$, where $V$ is the RKHS whose kernel is $K$.
Whenever $V$ is sufficiently rich in functions, there will be no
non-trivial solution to $\sum_{i=1}^r c_i f(x_i) = 0$ and every
matrix of the form $A_{ij} = K(x_i,x_j)$ is guaranteed to be
non-singular. For example, since the corresponding RKHS
contains all polynomials, the kernel (\ref{eq:bergman}) is
strictly positive, and for any collection of points
$x_1,\cdots,x_r$ in the unit disk, the matrix $A_{ij} =
(1-x_i \bar x_j)^{-2}$ is non-singular.

\subsection{Restriction of the Index Set}

If $V$ is a RKHS of functions on the interval $[0,2]$ then a new
space $V'$ of functions on $[0,1]$ is obtained by restricting
attention to the values of the functions on $[0,1]$.  Since two or
more functions in $V$ might collapse to the same function in $V'$,
the norm on $V$ does not immediately induce a norm on $V'$.  Since
a kernel $K$ on $X$ restricts to a kernel $K'$ on $X'$, it is natural
to ask if there is a norm on $V'$ such that $K'$ is the kernel of
$V'$. (The norm would have to come from an inner product.)

When interpolation problems are studied later, it will be seen that
$K(\cdot,x)$ solves the one-point interpolation problem; it is the
function of smallest norm satisfying $f(x) = K(x,x)$.  This suggests
the norm of $f$ should be the smallest norm of all possible functions
in $V$ collapsing to $f$.  This is precisely the case and is now
formalised.

If $X' \subset X$ then any $f \colon X \rightarrow \reals$ restricts
to a function $f|_{X'} \colon X' \rightarrow \reals$ given by
$f|_{X'}(x) = f(x)$ for $x \in X'$.  If $V \subset \reals^X$ is a
RKHS then a new space $V' \subset \reals^{X'}$ results from restricting
each element of $V$ to $X'$.  Precisely, $V' = \{f|_{X'} \mid f \in
V\}$.  Define on $V'$ the norm
\begin{equation}
\label{eq:finfg}
\| f \| = \inf_{\substack{g \in V \\ g|_{X'} = f}} \| g \|.
\end{equation}

\paragraph*{Remark}
In (\ref{eq:finfg}), $\inf$ can be replaced by $\min$ because $V$
is complete.

A kernel $K\colon X \times X \rightarrow \reals$ restricts to a
kernel $K'\colon X' \times X' \rightarrow \reals$ where $K'(x,y) =
K(x,y)$ for $x,y \in X'$. As proved in~\cite[p. 351]{Aronszajn:1950tq}
and~\cite[Theorem 6]{Berlinet:2012uv}, the kernel of $V'$ is $K'$.

\subsection{Sums of Kernels}

If $K_1$ and $K_2$ are kernels on the same set $X$ then
$K = K_1+K_2$ is positive semi-definite and hence a kernel
of a RKHS $V \subset \reals^X$. There is a relatively
straightforward expression for $V$ in terms of the
RKHSs $V_1$ and $V_2$ whose kernels are $K_1$ and $K_2$
respectively.

The space itself is $V = V_1 \oplus V_2$, that is,
\begin{equation}
V = \{ f_1 + f_2 \mid f_1 \in V_1,\ f_2 \in V_2\}.
\end{equation}
The norm on $V$ is given by a minimisation:
\begin{equation}
\label{eq:normsum}
\| f \|^2 = \inf_{\substack{f_1 \in V_1 \\ f_2 \in V_2 \\ f_1 + f_2 = f}}
    \| f_1 \|^2 + \| f_2 \|^2.
\end{equation}
The three norms in (\ref{eq:normsum}) are defined on $V$, $V_1$ and $V_2$, in
that order.
Proofs of the above can be found in~\cite[p. 353]{Aronszajn:1950tq}
and~\cite[Theorem 5]{Berlinet:2012uv}.

If there is no non-zero function belonging to both $V_1$ and $V_2$
then $f \in V$ uniquely decomposes as $f = f_1 + f_2$.  To see this,
assume $f_1 + f_2 = f_1' + f_2'$ with $f_1,f_1' \in V_1$ and $f_2,f_2'
\in V_2$.  Then $f_1 - f_1' = f_2 - f_2'$. As the left-hand side
belongs to $V_1$ and the right-hand side to $V_2$, the assertion
follows. In this case, (\ref{eq:normsum}) becomes $\|f\|^2 = \|f_1\|^2
+ \|f_2\|^2$, implying from Pythagoras' theorem that $V_1$ and $V_2$
have been placed at right-angles to each other in $V$.

% ----------------------------------------------------
%                       CHAPTER
% ----------------------------------------------------

\chapter{Geometry by Design}
\label{sec:GbD}

The delightful book~\cite{Jiri2010thirty} exemplifies the advantages
of giving a geometry to an otherwise arbitrary collection of points
$\{p_t \mid t \in T\}$. Although this can be accomplished by 
a recipe for evaluating the distance $d(p_t,p_\tau)$ between pairs
of points, a metric space lacks the richer structure of Euclidean
space; linear operations are not defined, and there is no inner
product or coordinate system.

\section{Embedding Points into a RKHS}
\label{sec:eprk}

Assume a recipe has been given for evaluating $\langle p_t, p_\tau
\rangle$ that satisfies the axioms of an inner product.  If the
$p_t$ are random variables then the recipe might be $\langle p_t,
p_\tau \rangle = E[p_t p_\tau]$. It is not entirely trivial to
construct an inner product space and arrange correctly the $p_t$
in it because the recipe for $\langle p_t, p_\tau \rangle$ may imply
linear dependencies among the $p_t$. It is not possible to
take by default $\{p_t \mid t \in T\}$ as a basis for the space.

RKHSs have the advantage of not requiring the kernel to be non-singular.
The set $\{p_t \mid t \in T\}$ can serve directly as an ``overdetermined
basis'', as now demonstrated. Note $T$ need not be a finite set.

Define $K(\tau,t) = \langle p_t, p_\tau \rangle$. If the inner
product satisfies the axioms required of an inner product then
$K\colon T \times T \rightarrow \reals$ is positive semi-definite
and therefore is the kernel of a RKHS $V \subset \reals^T$.
Just like an element of $\reals^2$ can be drawn as a point
and given a label, interpret the element $K(\cdot,t)$ of $V$
as a point in $\reals^T$ labelled $p_t$. In this way, the $p_t$ have
been arranged in $V$. They are arranged correctly because
\begin{equation}
\label{eq:Kpt}
\langle K(\cdot,t), K(\cdot,\tau) \rangle = K(\tau, t) =
\langle p_t, p_\tau \rangle.
\end{equation}

Not only have the points $p_t$ been arranged correctly in the RKHS, they have
been given coordinates: since $p_\tau$ is represented by $K(\cdot,\tau)$, its
$t$th coordinate is $\langle K(\cdot,\tau), K(\cdot, t) \rangle = K(t,\tau)$. 
This would not have occurred had the points been placed in an arbitrary Hilbert
space, and even if the Hilbert space carried a coordinate system, the assignment
of coordinates to points $p_t$ would have been \emph{ad hoc} and lacking the
reproducing property.

Even if the $p_t$ already belong to a Hilbert space,
embedding them into a RKHS will give them a pointwise coordinate
system that is consistent with taking limits (i.e., strong convergence implies
pointwise convergence). See \S\ref{sec:sle} and \S\ref{sec:asp} for examples of
why this is beneficial.

Let $\Hs$ be a Hilbert space, $\{p_t \mid t \in T\} \subset \Hs$
a collection of vectors and $V \subset \reals^T$ the RKHS whose
kernel is $K(\tau,t) = \langle p_t, p_\tau \rangle$.  Then $V$ is
a replica of the closure $U$ of the subspace in $\Hs$ spanned by
the $p_t$; there is a unique isometric isomorphism
$\phi\colon U \rightarrow V$ satisfying $\phi(p_t) = K(\cdot,t)$.
An isometric isomorphism preserves the linear structure and
the inner product. In particular, $\langle \phi(p), \phi(q)
\rangle = \langle p, q \rangle$.  The $t$th coordinate of $u \in
U$ is defined to be the $t$th coordinate of $\phi(u)$, namely,
$\langle \phi(u), K(\cdot,t) \rangle$.

\section{The Gaussian Kernel}
\label{sec:gk}

Points in $\reals^n$ can be mapped
to an infinite-dimensional RKHS by declaring the inner
product between $x,y \in \reals^n$ to be
\begin{equation}
\label{eq:Gkern}
\langle x, y \rangle = \exp\left\{-\frac{1}{2}\|x-y\|^2\right\}.
\end{equation}
As in \S\ref{sec:eprk}, 
the RKHS is defined via its kernel $K(x,y) = \langle x, y \rangle$,
called a Gaussian kernel. The point $x \in \reals^n$ is
represented by the element $K(\cdot,x)$ in the RKHS,
a Gaussian function centred at $x$.

The geometry described by an inner product is often easier to understand
once the distance between points has been deduced.
Since $\langle K(\cdot,x), K(\cdot,x) \rangle = K(x,x)=1$,
every point $x \in \reals^n$ is mapped to a unit length vector in the RKHS.
The squared distance between $x,y \in \reals^n$ is
\begin{align}
\MoveEqLeft \nonumber
\langle K(\cdot,x) - K(\cdot,y), K(\cdot,x) - K(\cdot,y) \rangle \\
    \nonumber
    &= \langle K(\cdot,x), K(\cdot,x) \rangle - 2 \langle K(\cdot,x),
        K(\cdot, y) \rangle + \langle K(\cdot,y), K(\cdot,y) \rangle \\
    \label{eq:sqdist}
    &= K(x,x) - 2K(x,y) + K(y,y) \\
    &= 2\left( 1 - \exp\left\{-\frac{1}{2}\|x-y\|^2\right\} \right).
\end{align}
As $y$ moves further away from $x$ in $\reals^n$, their
representations in the infinite-dimensional RKHS also move
apart, but once $\|x-y\| > 3$, their representations are close
to being as far apart as possible, namely, the representations
are separated by a distance close to 2.

Although the maximum straight-line distance between two
representations is 2, a straight line in $\reals^n$ does not
get mapped to a straight line in the RKHS. The original
geometry of $\reals^n$, including its linear structure, is
completely replaced by (\ref{eq:Gkern}). The ray $t \mapsto tx$,
$t \geq 0$, is mapped to the curve $t \mapsto K(\cdot,tx)$. The
curve length between $K(\cdot,0)$ and $K(\cdot,tx)$ is
\begin{align*}
\MoveEqLeft
\lim_{N \rightarrow \infty} \sum_{i=0}^{N-1}
    \left\| K\left(\cdot,\frac{i+1}Ntx\right)
    - K\left(\cdot,\frac{i}Ntx\right) \right\| \\
    &= \lim_{N \rightarrow \infty} \sum_{i=0}^{N-1}
        \sqrt{2\left( 1 - \exp\left\{-\frac{1}{2N^2}\|tx\|^2\right\} \right)} \\
    &= \lim_{N \rightarrow \infty} 
    \sqrt{2N^2\left( 1 - \left[1 -\frac{1}{2N^2}\|tx\|^2 + \cdots\right]
        \right)} \\
    &= t\|x\|.
\end{align*}

Calculations similar to the above can be carried out for any given
kernel $K$, at least in theory.

% ----------------------------------------------------
%                       CHAPTER
% ----------------------------------------------------

\chapter{Applications to Linear Equations and Optimisation}

Linearly constrained norm-minimisation problems in Hilbert spaces
benefit from the norm coming from an inner product.  The inner
product serves as the derivative of the cost function, allowing the
minimum-norm solution to be expressed as the solution to a linear
equation. This equation, known as the normal equation, can be written
down directly as an orthogonality constraint, as shown on the right
in Figure~\ref{fig:Fapprox}.  Whenever the constraint set forms a
\emph{closed} subspace, a minimum-norm solution is guaranteed to
exist~\cite{Luenberger:1969wv}.

Reproducing kernel Hilbert spaces have additional benefits when the
constraints are of the form $f(x)=c$ because the intersection of a
finite or even infinite number of such pointwise constraints is a
\emph{closed} affine space. The minimum-norm solution is guaranteed
to exist and is expressible using the kernel of the RKHS.

Example~\ref{ex:Interp2} showed mathematically that the kernel
features in minimum-norm problems. A simple geometric explanation
comes from the defining equation $\langle f, K(\cdot,x) \rangle =
f(x)$ implying $K(\cdot,x)$ is orthogonal to every function
$f$ having $f(x)=0$. This is precisely the geometric constraint
for $K(\cdot,x)$ to be a function of smallest norm satisfying
$f(x)=c$. The actual value of $c$ is determined by substituting
$f=K(\cdot,x)$ into the defining equation, yielding the self-referential
$f(x) = K(x,x)$.

\section{Interpolation Problems}
\label{sec:ip}

The interpolation problem asks for the function $f \in V \subset
\reals^X$ of minimum norm satisfying the finite number of constraints
$f(x_i) = c_i$ for $i=1,\cdots,r$. Assuming for the moment there
is at least one function in $V$ satisfying the constraints,
if $V$ is a RKHS with kernel $K$ then a minimum-norm solution exists.

Let $U = \{f \in V \mid f(x_i) = c_i,\ i=1,\cdots,r\}$ and $W = \{
g \in V \mid g(x_i) = 0,\ i=1,\cdots,r\}$.  The subspace $W$ is
closed because $g(x) = 0$ is equivalent to $\langle g, K(\cdot,x)
\rangle = 0$ and the collection of vectors orthogonal to a
particular vector forms a closed subspace.  (Equivalently, the
boundedness of evaluation functionals implies they are continuous
and the inverse image of the closed set $\{0\}$ under a continuous
function is closed.) Being the intersection of a collection of
closed sets, $W$ is closed.  Thus $U$ is closed since it is a
translation of $W$, and a minimum-norm solution $f \in U$ exists
because $U$ is closed (\S\ref{sec:completions}).

A minimum-norm $f \in U$ must be orthogonal to all $g \in W$, for
else $f+\epsilon g \in U$ would have a smaller norm for some $\epsilon
\neq 0$. Since $W$ is the orthogonal complement of $O =
\operatorname{span}\{K(\cdot,x_1), \cdots, K(\cdot,x_r)\}$, and $O$
is closed, $O$ is the orthogonal complement of $W$.  The minimum-norm
solution must lie in $O$.

The minimum-norm $f$ can be found by writing $f = \sum_{j=1}^r
\alpha_j K(\cdot,x_j)$ and solving for the scalars $\alpha_j$ making
$f(x_i) = c_i$. If no such $\alpha_j$ exist then
$U$ must be empty.  Since $f(x_i) = \sum_{j=1}^r \alpha_j K(x_i,x_j)$,
the resulting linear equations can be written in matrix form as
$A\alpha = c$ where $A_{ij} = K(x_i,x_j)$ and the $\alpha$ and $c$
are vectors containing the $\alpha_j$ and $c_i$ in order.  A
sufficient condition for $A$ to be non-singular is for the kernel
to be strictly positive definite (\S\ref{sec:invmat}).
See also~\cite{Larkin:1970do, Yao:1967tc}.

It is remarked that the Representer Theorem in statistical learning
theory is a generalisation of this basic result.

Example~\ref{ex:pws} demonstrates that certain spaces of bandlimited
signals are reproducing kernel Hilbert spaces.  Reconstructing a
bandlimited signal from discrete-time samples can be posed as an
interpolation problem~\cite{Yao:1967tc}.

\section{Solving Linear Equations}
\label{sec:sle}

If the linear equation $Ax=b$ has more than one solution, the
Moore-Penrose pseudo-inverse $A^+$ of $A$ finds the solution $x =
A^+ b$ of minimum norm.  The involvement of norm minimisation
suggests RKHS theory might have a role in solving infinite-dimensional
linear equations.

The minimum-norm solution is that which is orthogonal to the
null-space of $A$. The $r$ rows of $A$ are orthogonal to the null-space,
therefore, the required solution is that which is expressible as a
linear combination of the rows of $A$. This motivates writing $Ax=b$
as $\langle x, a_i \rangle = b_i$, $i=1,\cdots,r$, where $a_i$ is
the $i$th row of $A$, and looking for a solution $x$ lying in $U =
\operatorname{span}\{a_1,\cdots,a_r\}$. 

If $U$ was a RKHS and the $a_i$ of the form $K(\cdot,i)$ then
$\langle x, a_i \rangle = b_i$ would be an interpolation problem
(\S\ref{sec:ip}). This form can always be achieved by embedding $U$
into a RKHS (\S\ref{sec:eprk}) because the embedding $\phi$ sends
$a_i$ to $K(\cdot,i)$ and preserves the inner product: $\langle x,
a_i \rangle = \langle \phi(x), \phi(a_i) \rangle = \langle \phi(x),
K(\cdot, i) \rangle$. The solution to $\langle \phi(x), K(\cdot,
i) \rangle = b_i$ is $\phi(x)=b$ and hence the solution to $\langle
x, a_i \rangle = b_i$ is $x = \phi^{-1}(b)$.

\begin{theorem}
\label{th:lineq}
Let $\Hs$ be a Hilbert space and $\{a_t \mid t \in T\} \subset \Hs$
a collection of vectors in $\Hs$. Denote the closure of the span
of the $a_t$ by $U$. The equations $\langle x, a_t \rangle = b_t$
for $t \in T$ have a solution $x \in U$ if and only if the function
$t \mapsto b_t$ is an element of the RKHS $V \subset \reals^T$ whose
kernel $K \colon T \times T \rightarrow \reals$ is $K(s,t) = \langle
a_t, a_s \rangle$.  If a solution $x$ in $U$ exists then it is
unique and has the smallest norm out of all solutions $x$ in
$\Hs$.  Moreover, $x = \phi^{-1}(b)$ where $b(t) = b_t$ and $\phi\colon
U \rightarrow V$ is the unique isometric isomorphism sending $a_t$
to $K(\cdot,t)$ for $t \in T$.
\end{theorem}
\begin{proof}
Follows from the discussion above; see 
also~\cite[Theorem 42]{Berlinet:2012uv}.
\end{proof}

The $\Hs$-norm of the solution $x$ equals the $V$-norm of $b$ because
$\phi$ in Theorem~\ref{th:lineq} preserves the norm.  Sometimes
this knowledge suffices. Otherwise, in principle, $x$ can be found
by expressing $b$ as a linear combination of the $K(\cdot,t)$. If
$b = \sum_k \alpha_k K(\cdot,t_k)$ then $x = \phi^{-1}(b) = \sum_k
\alpha_k a_{t_k}$ by linearity of $\phi$.  More generally, limits
of linear combinations can be used, including integrals and derivatives
of $K(\cdot,t)$; see~\cite[Section III-B]{Kailath:1971hk}
or~\cite[Theorem 43]{Berlinet:2012uv}. For an alternative,
see Theorem~\ref{th:sle}.

\begin{example}
\label{ex:axb}
The equation $Ax=b$ can be written as $\langle x, a_i \rangle =
b_i$ with respect to the Euclidean inner product. Then $K(i,j) =
\langle a_j, a_i \rangle$. In matrix form, $K = A A^\top$.  The map
$\phi$ sending $a_i = A^\top e_i$ to $K(\cdot,i) = A A^\top e_i$ is
$\phi(x) = Ax$. It is more useful though to think of $\phi$ as
sending $A^\top v$ to $AA^\top v$ for arbitrary $v$ because then
$\phi^{-1}$ is seen to send $AA^\top v$ to $A^\top v$.
Expressing $b$ as a linear combination of the
$K(\cdot,i)$ means finding the vector $\alpha$ such that $b = A A^\top
\alpha$.  Assuming $AA^\top$ is invertible, $\alpha = (A A^\top)^{-1} b$.
As $b = (AA^\top) (AA^\top)^{-1} \alpha$ and $\phi^{-1}$ changes the leading
$AA^\top$ to $A^\top$, $x = \phi^{-1}(b) = A^\top (A A^\top)^{-1} b$.
\end{example}

The kernel $K=AA^\top$ in the above example corresponds to
the inner product $\langle b, c \rangle = c^\top (AA^\top)^{-1} b$
by the observation in Example~\ref{ex:Qinv}. The solution
$x = A^\top (AA^\top)^{-1} b$ also involves $(AA^\top)^{-1}$, suggesting
$x$ can be written using the inner product of the RKHS.
This is pursued in \S\ref{sec:frames}.

\begin{example}
\label{ex:integral}
Consider the integral transform
\begin{equation}
\label{eq:fFint}
f(t) = \int_Z F(z) h(z,t)\,dz,\qquad t \in T,
\end{equation}
taking a function $F \in \reals^Z$ and returning a function $f \in
\reals^\top$.  Regularity conditions are necessary for the integral
to exist: assume $F(\cdot)$ and $h(\cdot,t)$ for all
$t \in T$ are square-integrable, as in \cite[Chapter 6]{Saitoh:1988vg}.
Let $\Hs$ be
the Hilbert space $L_2(Z)$ of square-integrable functions on $Z$.
Then (\ref{eq:fFint}) becomes $\langle F, h(\cdot,t) \rangle =
f(t)$ for $t \in T$. Define $K(s,t) = \langle h(\cdot,t), h(\cdot,s) \rangle
= \int_Z h(z,s) h(z,t)\,dz$ and let $V \subset \reals^\top$ be the
associated RKHS. Then (\ref{eq:fFint}) has a solution $F$ if
and only if $f \in V$.
\end{example}

\subsection{A Closed-form Solution}
\label{sec:frames}

The solution $x = A^\top (AA^\top)^{-1} b$ to $Ax=b$ in Example~\ref{ex:axb}
can be written as $x_j = \langle b, c_j \rangle$ where $c_j$ is the
$j$th column of $A$ and the inner product comes from the kernel
$K=AA^\top$, that is, $\langle b, c \rangle = c^\top (AA^\top)^{-1} b$.
Indeed, $\langle b, c_j \rangle = \langle b, A e_j \rangle =
e_j^\top A^\top(AA^\top)^{-1} b = e_j^\top x = x_j$, as claimed.

Similarly, under suitable regularity conditions, a closed-form
solution to (\ref{eq:fFint}) can be found as follows.  Let $V \subset
\reals^X$ be the RKHS with $K(s,t) = \int h(z,t) h(z,s)\,dz$ as
its kernel.  Any $f \in V$ therefore satisfies
\begin{equation}
\label{eq:fFker}
\begin{aligned}
f(t) &= \left\langle f, K(\cdot,t) \right\rangle \\
    &= \left\langle f, \int h(z,t) h(z,\cdot) \,dz \right\rangle \\
    &= \int h(z,t) \left\langle f, h(z,\cdot) \right\rangle\,dz \\
    &= \int h(z,t) F(z) \, dz,
\end{aligned}
\end{equation}
showing the integral equation $f(t) = \int F(z) h(z,t) \, dz$ has
$F(z) = \langle f, h(z,\cdot) \rangle$ as a solution.  This derivation
does not address whether the solution has minimum norm;
sometimes it does~\cite[Chapter 6]{Saitoh:1988vg}.

The following novel derivation builds on the idea in (\ref{eq:fFker}) and
culminates in Theorem~\ref{th:sle}. It relies on the concept of a Parseval
frame~\cite{han2007frames, Paulsen:rkhs}. A collection of vectors $c_s \in V$ is
a Parseval frame for $V \subset \reals^T$ if and only if $b = \sum_s \langle b,
c_s \rangle c_s$ for all $b \in V$. By~\cite[Theorem 3.12]{Paulsen:rkhs}, $c_s$
is a Parseval frame for $V$ if and only if $K = \sum_s c_s c_s^\top$. This
condition extends the second definition in Definition~\ref{def:kernel} because
an orthonormal basis is \emph{a fortiori} a Parseval frame. Note that a
solution to $\langle x, a_t \rangle = b_t$ can be found analogous to
(\ref{eq:fFker}) by substituting $K=\sum_s c_s c_s^\top$ into $b_t = \langle b,
Ke_t \rangle$ in the particular case when $T=\{1,2,\cdots,n\}$. A more general
strategy is sought though.

The difficulty of solving $\phi(x)=b$ in Theorem~\ref{th:lineq} using the kernel
directly is having to express $b$ as a linear combination of the $K(\cdot,t)$.
Expressing $b$ as a linear combination of a Parseval frame $c_s$ though is
trivial: $b = \sum_s \langle b, c_s \rangle c_s$.  For this to be useful, it
must be possible to evaluate $\phi^{-1}(c_s)$. It turns out that the choice $c_s
= \psi(v_s)$ achieves both objectives, where $v_s$ is an orthonormal basis for
$\Hs$ and $\psi\colon \Hs \rightarrow \reals^T$ is the map sending $x$ to $b$
where $b_t = \langle x, a_t \rangle$.

The map $\phi\colon U \rightarrow V$ in Theorem~\ref{th:lineq} is the
restriction of $\psi$ to $U$.  Denote by $P$ the orthogonal projection onto $U$.
If $b = \psi(x)$ for some $x \in \Hs$ then the minimum-norm solution of
$\psi(\hat x) = b$ is $\hat x = P(x)$. Equivalently, $\phi^{-1}(\psi(x)) =
P(x)$. Moreover, $\phi(P(x)) = \psi(x)$.

The inverse image $\phi^{-1}(c_s)$ of $c_s = \psi(v_s)$ can be deduced from
the following calculation, valid for any $x \in \Hs$.
\begin{equation}
\label{eq:pfder}
\begin{aligned}
\phi^{-1}(\psi(x)) &= \sum_s \langle P(x), v_s \rangle v_s \\
    &= \sum_s \langle P(x), P(v_s) \rangle v_s \\
    &= \sum_s \langle \phi(P(x)), \phi(P(v_s)) \rangle v_s \\
    &= \sum_s \langle \psi(x), \psi(v_s) \rangle v_s.
\end{aligned}
\end{equation}
In fact, this shows not only how $\phi^{-1}(c_s)$ can be found, but that the
minimum-norm solution $\hat x$ to $\psi(x)=b$ is given by
\begin{equation}
\hat x = \sum_s \langle b, c_s \rangle v_s,\quad c_s = \psi(v_s).
\end{equation}
Equivalently, $x$ is found ``elementwise'' by $\langle x, v_s \rangle = \langle
b, c_s \rangle$.  Note that $b = \sum_s \langle b, c_s \rangle c_s$, therefore,
$x$ is obtained by using the same linear combination but with $c_s$ replaced by
its preimage $v_s$.  Even though the $v_s$ need not lie in $U$, the linear
combination of them does.

That the $c_s$ form a Parseval frame only enters indirectly to guide the
derivation (\ref{eq:pfder}).  It helps explain why it works. For completeness,
the following verifies the $c_s$ form a Parseval frame.
\begin{align*}
e_t^\top \left( \sum_s c_s c_s^\top \right) e_\tau &=
    \sum_s (e_t^\top c_s) (e_\tau^\top c_s) \\
    &= \sum_s \langle v_s, a_t \rangle \langle v_s, a_\tau \rangle \\
    &= \left\langle \sum_s \langle a_t, v_s \rangle v_s, a_\tau \right\rangle \\
    &= \langle a_t, a_\tau \rangle 
    = \langle a_\tau, a_t \rangle
    = K(t,\tau).
\end{align*}

The above leads to the following elementary but new result.

\begin{theorem}
\label{th:sle}
With notation as in Theorem~\ref{th:lineq},
let $\{v_s \mid s \in S\}$ be an orthonormal basis for $\Hs$.
Define $c_s(t) = \langle v_s, a_t \rangle$. The minimum-norm
solution $x$ in Theorem~\ref{th:lineq} is the unique solution
to $\langle x, v_s \rangle = \langle b, c_s \rangle$.
\end{theorem}
\begin{proof}
As $v_s$ is an orthonormal basis, the solution $x$ to $\langle x,
v_s \rangle = \langle b, c_s \rangle$ is unique. It therefore
suffices to show $\langle x, v_s \rangle = \langle \phi(x), c_s
\rangle$ for $x \in U$. This follows from $\langle x, v_s \rangle
= \langle P(x), v_s \rangle = \langle x, P(v_s) \rangle = \langle
\phi(x), \phi(P(v_s)) \rangle = \langle \phi(x), c_s \rangle$ where
$P$ is projection onto $U$. That $c_s = \phi(P(v_s))$ can be derived
from $c_s(t) = \langle v_s, P(a_t) \rangle = \langle P(v_s), a_t
\rangle$.
\end{proof}

Theorem~\ref{th:sle} can be used to solve numerically integral
equations via series expansions (Example~\ref{ex:integral}). If $x
= \sum_s \alpha_s v_s$ then each $\alpha_s$ is found by evaluating
numerically the RKHS inner product $\langle b, c_s \rangle$. Here,
$c_s$ can be found by evaluating an integral numerically.

% ----------------------------------------------------
%                       CHAPTER
% ----------------------------------------------------

\chapter{Applications to Stochastic Processes}
\label{sec:asp}

Parzen is credited with introducing RKHS theory into
statistics~\cite{Parzen:1959uu, Parzen:1962ug, parzen1963probability,
Parzen:1970wz}.  He learnt about RKHSs by necessity, as a
graduate student, back when the theory was in its infancy, because
it was his prescribed minor thesis topic~\cite{Newton:2002uz}.
A decade later, Kailath demonstrated to the signal processing
community the relevance of RKHS theory in a series of lucid papers
concerning detection, filtering and parameter
estimation~\cite{Kailath:1971hk, Kailath:1972jo, Duttweiler:1973bq,
Duttweiler:1973io, Kailath:1975bj}.  Kailath and Parzen were
colleagues at Stanford~\cite{Newton:2002uz}.

Despite Hilbert spaces being linear spaces, RKHS theory is not
confined to studying \emph{linear} estimation theory; a clever use
of characteristic functions leads to a RKHS theory for optimal
nonlinear estimators~\cite{Hida:1967ux, Kallianpur:1970wu,
Duttweiler:1973bq, Stulajter:1978wt}. Robust estimation theory has
also been considered~\cite{Barton:1990ir}.

A RKHS approach usually entails finding an expression for the kernel.
If this cannot be done, Kailath observes it may
suffice to use a numerical scheme for computing the norms of
functions on a RKHS~\cite{Kailath:1972bv}.

The ideas underlying the application of RKHS theory to
stochastic processes are presented without regard to
technical considerations (such as measurability).

\section{Detection in Finite Dimensions}
\label{sec:difd}

Let $m \in \reals^n$ be a known vector called the message. Let $w \in
\reals^n$ be a realisation of Gaussian white noise, that is, $w \sim N(0,I)$.
The detection problem is deciding whether the received signal $y \in \reals^n$
represents only noise, $y = w$, or the message plus noise, $y = m+w$.

The standard mechanism for making such a decision is the likelihood ratio
test. The likelihood of $y$ being just noise is proportional to
$\exp\{-\frac12\|y\|^2\}$ while the likelihood of $y$ being the message plus
noise is proportional to $\exp\{-\frac12\|y-m\|^2\}$, where the constant of
proportionality is the same in both cases. The message is therefore deemed to
be present if and only if $\|y-m\| < \|y\|$.

The test $\|y-m\| \lessgtr \|y\|$ has a simple geometric explanation. Interpret
the case $y=w$ as sending the zero message: $y=0+w$. In $\reals^n$, label the
point $m$ as the message, label the origin as the zero message and label the
point $y$ as the received signal. Since the distribution of $w$ is radially
symmetric, the likelihood of $y = s+w$ is proportional to $\|y-s\|$ and decays
monotonically. Deciding between $s=0$ and $s=m$ comes down to deciding whether
$y$ is closer to $0$ or $m$.

The geometry suggests $\langle y,m \rangle - \frac12\langle m, m\rangle
\lessgtr 0$ as an equivalent formulation of the likelihood ratio test,
corresponding to projecting the signal $y$ onto the ``message space'' $m$.

\newcommand{\Sh}{\Sigma^{\frac12}}
\newcommand{\Shi}{\Sigma^{-\frac12}}

Generalising to coloured noise $N(0,\Sigma)$ is achieved by replacing $w$ by
$\Sh w$. Testing between $y = \Sh w$ and $y = m + \Sh w$ is achieved by
whitening the received signal, thus deciding between $\Shi y = w$ and $\Shi y =
\Shi m + w$. Specifically, the test is
\begin{equation}
\label{eq:Shi}
\langle \Shi y, \Shi m\rangle - \frac12\langle \Shi m, \Shi m\rangle \lessgtr 0,
\end{equation}
assuming of course $\Sigma$ is non-singular.

Intriguingly, (\ref{eq:Shi}) can be written as
\begin{equation}
\label{eq:ShiK}
\langle y, m \rangle_K -\frac12 \langle m, m \rangle_K \lessgtr 0,
\end{equation}
where $\langle \cdot, \cdot \rangle_K$ is the inner product on the RKHS whose
kernel is $K=\Sigma$. Indeed, from Example~\ref{ex:Qinv}, the RKHS whose kernel
is $K$ is $\reals^n$ equipped with the inner product $\langle x,y \rangle_K =
y^T \Sigma^{-1} x$.

\newcommand{\range}{\operatorname{range}}

The advantage of (\ref{eq:ShiK}) is it can be shown to remain valid when
$\Sigma$ is singular. Importantly, note that if $m$ does not lie in $\range(\Sh)
= \range(\Sigma)$ then the detection problem can be solved deterministically.
Otherwise, if $m$ is in $\range(\Sigma)$ then $y$ will also lie in
$\range(\Sigma)$, hence (\ref{eq:ShiK}) is meaningful because both $m$ and
$y$ will be elements of the RKHS.

The RKHS structure gives a useful link between the geometry $\langle \cdot,
\cdot \rangle_K$ and the statistics of the noise $\Sh w$. This manifests itself
in several ways, one of which is that the distribution of $\langle \Sh w, z
\rangle_K$ is zero-mean Gaussian with variance $\langle z, z \rangle_K$, and
moreover, the correlation between any two such Gaussian random variables is
given by
\begin{equation}
E[\, \langle \Sh w, x \rangle_K \, \langle \Sh w, z \rangle_K \,] =
  \langle x, z \rangle_K.
\end{equation}
In particular, with respect to $\langle \cdot, \cdot \rangle_K$, the
distribution of $\Sh w$ is spherically symmetric, hence the ratio test reduces
to determining whether $y$ is closer to $0$ or $m$ with respect to the RKHS norm
$\|\cdot\|_K$.

\section{The RKHS Associated with a Stochastic Process}
\label{sec:rsp}

A stochastic process $X$ is a collection of random variables $X(t)$
indexed by a parameter $t \in T$ often taken to be time. An archetypal
process is the Wiener process, described from first principles
in~\cite{Manton:2013kk}, among other introductory material.  A
sample path is a realisation $t \mapsto X(t)$ of $X$, also known
as a randomly chosen signal or waveform.

Despite RKHSs being function spaces and sample paths being functions,
a RKHS structure is not given to the space of all possible sample
paths.  Since linear filters form linear combinations of the $X(t)$,
the Hilbert space approach~\cite{Small:2011wg} to filtering endows the
space spanned by the individual $X(t)$ with the inner product
$\langle X(s), X(t) \rangle = E[X(s) X(t)]$.  This inner product
equates the statistical concept of conditional expectation with the
simpler geometric concept of projection. For example, uncorrelated
random variables are orthogonal to each other.

The RKHS approach goes one step further. The space spanned by the
$X(t)$ has a useful geometry but no convenient coordinate system.
A point in that space is nothing more than a point representing a
random variable, whereas a point in $\reals^n$ comes with $n$
coordinates describing the point's precise location.  In hindsight,
the RKHS approach can be understood as giving a convenient coordinate
system to the space of random variables.  (While other authors have
considered a RKHS approach to be coordinate free, we argue $\langle
f, K(\cdot,x)\rangle$ is precisely the $x$th coordinate of $f$.
RKHS theory is a blend of coordinate-free geometric reasoning with
algebraic manipulations in pointwise coordinates.)

The coordinate system introduced by the RKHS approach assigns to
each random variable $U$ the function $t \mapsto E[UX(t)]$. The
$t$th coordinate is $E[UX(t)]$. This provides explicitly a wealth
of information about $U$ in a form compatible with the geometry of
the space.

This coordinate system comes from taking the covariance function $R(t,s) =
E[X(t)X(s)]$, which is always positive semi-definite, to be the kernel of a RKHS
$V \subset \reals^T$. Recall from \S\ref{sec:GbD} that this construction
arranges the individual points $X(t)$ in $V$ according to the geometry $\langle
X(s), X(t) \rangle = E[X(s) X(t)]$. Precisely, the element $R(\cdot,s) \in V$
represents the random variable $X(s)$.  The space $V$ is the completion of the
space spanned by the $R(\cdot,s)$ for $s \in T$.  Any (mean-square) limit $U$ of
finite linear combinations of the $X(t)$ can therefore be represented in $V$ by
a limit $E[UX(t)]$ of corresponding finite linear combinations of the
$R(\cdot,s)$. This map $U \mapsto E[UX(t)]$ corresponds to the isometric
isometry $\phi$ in \S\ref{sec:eprk}. In other words, associated with the random
variable $U$ are the coordinates $t \mapsto E[UX(t)]$ describing the location of
the embedding of $U$ in the  RKHS $V$.

While the elements of $V$ and sample paths are both functions on
$T$, the former represents the coordinates $t \mapsto E[U X(t)]$
of a random variable $U$ while the latter represents a realisation
of $t \mapsto X(t)$. Especially since the former is deterministic
and the latter stochastic, there is not necessarily any relationship
between the two.  Surprisingly then, several relationships have
been found for Gaussian processes. In typical cases, the sample paths are
elements of the RKHS with probability zero or one, depending on the process
itself~\cite{Driscoll:1973un, Lukic:2001un}.  A
detection problem is non-singular if and only if the signal belongs
to $V$~\cite{Kailath:1971hk}.  (This is a manifestation of the
Cameron-Martin Theorem, with the RKHS $V$ being known as the
Cameron-Martin space~\cite{da2006introduction}.)

\section{Signal Detection}
\label{sec:sigdet}

Salient features of the RKHS approach to signal detection are
summarised below, the purpose being to add substance to the higher
level description above. Greater detail can be found
in~\cite{Kailath:1971hk}.

The signal detection problem is a hypothesis test for deciding
between the null hypothesis that the signal $X(t)$ is composed of
just noise --- $X(t) = N(t)$ --- versus the hypothesis that a known
signal $m(t)$ is present --- $X(t) = m(t) + N(t)$.  For concreteness,
$t$ is assumed to belong to the interval $T = [0,1]$. The noise process
$N(t)$ is assumed to be a zero-mean second-order Gaussian process.
Its covariance function $R(t,s) = E[N(s)N(t)]$ is assumed to be
continuous (and hence bounded) on $T \times T$. (This is equivalent
to $N(t)$ being mean-square continuous.)
Example calculations will take $N(t)$ to be
the standard Wiener process~\cite{Manton:2013kk} with $N(0)=0$
and covariance function $R(t,s) = \min\{s,t\}$.

Hypothesis testing generally reduces to comparing
the likelihood ratio against a threshold~\cite{poor1994introduction}.
Two methods for determining the likelihood ratio are discussed below.

\subsection{The Karhunen-Lo\'eve Approach}
\label{sec:KL}

The continuous-time detection problem has a different flavour from
the conceptually simpler discrete-time detection problem when $T$
is a finite set. If $m(t)$ varies faster than a typical sample path
of $N(t)$, and especially if $m(t)$ contains a jump, correct detection
is possible with probability one~\cite{poor1994introduction}.
Otherwise, if exact detection
is not possible, the continuous-time detection problem can be solved
as a limit of discrete-time detection problems; being continuous,
a sample path of $N(t)$ is fully defined once its values at rational
times $t$ are known.

A more practical representation of $N(t)$ using a countable
number of random variables is the Karhunen-Lo\'eve expansion.
Provided $R(t,s)$ is continuous on the unit square $T \times T$,
it has an eigendecomposition
\begin{equation}
\label{eq:Reig}
R(t,s) = \sum_{k=1}^\infty \lambda_k \psi_k(t) \psi_k(s)
\end{equation}
where the $\lambda_k \geq 0$ are the eigenvalues and the $\psi_k$
are the orthonormal eigenfunctions defined by
\begin{equation}
\label{eq:Rlp}
\int_0^1 R(t,s) \psi(s)\,ds = \lambda \psi(t).
\end{equation}
Orthonormality imposes the extra constraint
\[
\int_0^1 \psi_n(t) \psi_m(t)\,dt = \begin{cases}
    1 & m = n; \\
    0 & \textrm{otherwise}.
    \end{cases}
\]
This is a generalisation of the eigendecomposition
of a positive semi-definite matrix and is known
as Mercer's theorem~\cite{Mercer:1909gf}.

If $N(t)$ is the standard Wiener process then
$R(t,s) = \min\{s,t\}$ and
\begin{align*}
\psi_k(t) &= \sqrt2 \sin\left( \left(k-\tfrac12\right) \pi t \right), \\
\lambda_k &= \frac1{\left(k-\tfrac12\right)^2 \pi^2},
\end{align*}
as verified in part by
\begin{align*}
\MoveEqLeft
\int_0^1 \min\{s,t\} \sqrt2 \sin\left( \left(k-\tfrac12\right) \pi s \right)
        \, ds  \\
    &= \sqrt2\left[
\int_0^t s \sin\left( \left(k-\tfrac12\right) \pi s \right) \, ds 
+ t \int_t^1  \sin\left( \left(k-\tfrac12\right) \pi s \right) \, ds \right] \\
    &= \lambda_k \sqrt2 \sin\left( \left(k-\tfrac12\right) \pi t \right).
\end{align*}
Note the $\lambda_k$ decay to zero.
Define
\[
N_k = \int_0^1 N(t) \psi_k(t)\,dt.
\]
The orthonormality of the $\psi_k$ and $N(t)$ having zero mean
imply $E[N_k] = 0$, $E[N_k^2] = \lambda_k$ and $E[N_nN_m]=0$ for
$n \neq m$. Indeed,
\begin{align*}
E[N_n N_m] &= \int_0^1 \int_0^1 R(t,s) \psi_n(s) \psi_m(t)\,ds\,dt \\
&= \int_0^1 \lambda_n \psi_n(t) \psi_m(t)\,dt \\
&= \begin{cases}
    \lambda_n & m = n; \\
    0 & \textrm{otherwise}.
    \end{cases}
\end{align*}
Since $N(t)$ is a Gaussian process, the $N_k$ are actually independent
Gaussian random variables with zero mean and variance $\lambda_k$.

The Karhunen-Lo\'eve expansion of $N(t)$ is
\begin{equation}
\label{eq:KLe}
N(t) = \sum_{k=1}^\infty N_k \psi_k(t)
\end{equation}
where it is simplest to interpret the right-hand side as a way
of generating a Wiener process: generate $N_k \sim N(0,\lambda_k)$
at random then form the summation. See~\cite{wong2011stochastic} for a proof.

Assume the signal $m(t)$ looks sufficiently like the noise that it
too can be represented as $m(t) = \sum_{k=1}^\infty m_k \psi_k(t)$.
Substituting this expansion of $m(t)$ into $\int_0^1 m(t) \psi_n(t)\,dt$,
and recalling the orthonormality of the $\psi_k$, shows $m_k =
\int_0^1 m(t) \psi_k(t)\,dt$.

Transforming the received signal $X(t)$ in the same way leads to
a corresponding hypothesis test for distinguishing between $X_k =
N_k$ and $X_k = m_k + N_k$, where $X_k = \int_0^1 X(t) \psi_k(t)\,dt$.
Since the $\lambda_k$ decrease to zero, the likelihood ratio can
be approximated by using only a finite number of the $X_k$.
For details, see~\cite{poor1994introduction}.

The use of (Lebesgue measure) $ds$ in (\ref{eq:Rlp}) is \emph{ad
hoc}~\cite{Kailath:1971hk}. Changing it will result in a different discrete-time
approximation of the detection problem.  An equivalent hypothesis test will
still result, nevertheless, avoiding an \emph{ad hoc} choice is preferable.
Moreover, computing an eigendecomposition (\ref{eq:Rlp}) of $R(t,s)$ can be
difficult~\cite{Kailath:1971hk}. The RKHS approach avoids the need for computing
an eigendecomposition of $R(t,s)$.

\subsection{The RKHS Approach}

The logarithm of the likelihood ratio of $X(t) = m(t)+N(t)$ to $X(t)
= N(t)$ when $N(t)$ is a Gaussian process takes the form
\begin{equation}
\label{eq:h}
\int_0^1
X(t)\, dH(t) - \frac12 \int_0^1 m(t)\, dH(t)
\end{equation}
if there exists an $H(t)$ (of bounded variation) satisfying
\begin{equation}
\label{eq:Hreq}
\int_0^1 R(t,s) dH(s) = m(t).
\end{equation}
This is a consequence of a martingale representation theorem;
see~\cite[Theorem 2.3]{Pitcher:1960ub}, \cite[Proposition
VI.C.2]{poor1994introduction} and~\cite[Section 7.2]{wong2011stochastic}
for details. In other words, the linear operator $X \mapsto \int_0^1
X(t)\,dH(t)$ reduces the signal detection problem to the one-dimensional problem
of testing whether a realisation of $\int_0^1 X(t)\,dH(t)$ came
from $\int_0^1 N(t)\,dH(t)$ or $\int_0^1 m(t)+N(t)\,dH(t)$.

Directly obtaining a closed-form solution to (\ref{eq:Hreq}) can be difficult.
The RKHS approach replaces this problem by the sometimes simpler problem of
computing the inner product of a particular RKHS; as discussed later, the RKHS
approach leads to the mnemonic
\begin{equation}
\label{eq:LRmnemonic}
\langle X, m \rangle - \frac12 \langle m, m \rangle
\end{equation}
for the likelihood ratio, where the inner product is that of the RKHS whose
kernel is the covariance $R(t,s)$. (See Example~\ref{ex:Nwp}.)

Let $U = \int_0^1 N(t)\,dH(t)$. Being the limit of certain finite
linear combinations of the $N(t)$, it belongs to the Hilbert space
(of square-integrable random variables)
generated by the $N(t)$. Since each $N(t)$ is zero-mean Gaussian,
so is $U$. Its distribution is thus fully determined by its
``coordinates'' $t \mapsto E[UN(t)]$. Since
\begin{align*}
E[UN(t)] &= \int_0^1 E[N(s) N(t)]\,dH(s) \\
    &= \int_0^1 R(t,s)\,dH(s),
\end{align*}
if the requirement (\ref{eq:Hreq}) holds then $U = \int_0^1 N(t)\,dH(t)$
satisfies
\begin{equation}
\label{eq:Emt}
E[U N(t)] = m(t).
\end{equation}

The existence of an $H(t)$ satisfying (\ref{eq:Hreq}) is only a sufficient
condition for the signal detection problem to be non-singular. The existence of
a random variable $U$, belonging to the Hilbert space generated by the $N(t)$
and satisfying (\ref{eq:Emt}), is both necessary and
sufficient~\cite{Kailath:1971hk}. The discrepancy arises because a $U$
satisfying (\ref{eq:Emt}) need not be expressible in the form $U = \int_0^1
N(t)\,dH(t)$; see~\cite[p.\@ 44]{Pitcher:1960ub} for an example.

Let $V \subset \reals^T$ be the RKHS whose kernel is the noise
covariance function $R(t,s) = E[N(s) N(t)]$. Recall that the Hilbert
space generated by the $N(t)$ is isometrically isomorphic to $V$,
with $U \mapsto E[UN(\cdot)]$ an isometric isomorphism.  Therefore,
(\ref{eq:Emt}) has a solution if and only if $m$ is an element of
the RKHS $V$.

Condition (\ref{eq:Hreq}) requires expressing $m(t)$ as a linear combination of
the $R(\cdot,s)$, but with the caveat that the linear combination take the form
of a particular integral. The requirement that $m$ be an element of $V$ is that
$m$ is the limit of finite linear combinations of the $R(\cdot,s)$ where the
limit is with respect to the norm on $V$, defined implicitly by $\langle
R(\cdot,s), R(\cdot,t) \rangle = R(t,s)$.

If $U$ satisfies (\ref{eq:Emt}) then the likelihood ratio can be found as
follows. Write $U$ as $U(N)$ to emphasise that $U$ can be considered to be an
almost surely linear function of the $N(t)$. The likelihood ratio is $U(X) -
\frac12 U(m)$. This agrees with (\ref{eq:h}) if $U(N) = \int_0^1 N(t)\,dH(t)$.

\begin{example}
\label{ex:det}
With notation as above, assume $m(t) = \sum_i \alpha_i R(\cdot,s_i)$.
A solution to (\ref{eq:Hreq}) is obtained by choosing $H(s)$ to have increments
at the $s_i$ of height $\alpha_i$. This leads to $\int_0^1 X(t)\, dH(t) = \sum_i
\alpha_i\, X(s_i)$ and $\int_0^1 m(t)\,dH(t) = \sum_i \alpha_i\, m(s_i)$.
Alternatively, let $\phi$ be the linear isomorphism $U \mapsto E[UN(\cdot)]$, so
that (\ref{eq:Emt}) becomes $\phi(U) = m$. Then $U(N) = \phi^{-1}(m) = \sum_i
\alpha_i \phi^{-1}(R(\cdot,s_i)) = \sum_i \alpha_i N(s_i)$. Thus $U(X) = \sum_i
\alpha_i X(s_i)$ and $U(m) = \sum_i \alpha_i m(s_i)$, resulting in the same
likelihood ratio.
\end{example}

A more explicit expression can be given for the second term of the
likelihood ratio (\ref{eq:h}) by using the inner product on the RKHS $V$. Let
$U(f) = \int_0^1 f(t)\,dH(t)$ where $H(t)$ satisfies (\ref{eq:Hreq}), that is,
$U(R(t,\cdot))=m(t)$. Assume
$U(f)$ for $f \in V$ can be expressed as $\langle f, g \rangle$ where $g \in V$ is to be determined.
The requirement (\ref{eq:Hreq}) becomes $\langle
R(t,\cdot), g \rangle = m(t)$, whose unique solution is $g=m$. In particular,
the second term of the likelihood ratio is proportional to
\begin{equation}
\label{eq:normm}
\int_0^1 m(t)\,dH(t) = \langle m, m \rangle = \| m \|^2.
\end{equation}
This suggests, and it turns out to be true, that $U(m) = \langle m, m \rangle$
in general, justifying the second part of (\ref{eq:LRmnemonic}). (Even if $U$
cannot be written in the form $\int \cdot\,dH$, it can be approximated
arbitrarily accurately by such an expression~\cite{Pitcher:1960ub}.)

Equation (\ref{eq:normm}) can be verified to hold in Example~\ref{ex:det}.
The square of the RKHS norm of $m(t)$ is
\begin{align*}
\|m\|^2 &= \left\langle \sum_i \alpha_i R(\cdot, s_i), \sum_j \alpha_j R(\cdot, s_j) \right\rangle \\
    &= \sum_{i,j} \alpha_i \alpha_j R(s_j,s_i).
\end{align*}
The integral can be written as
\begin{align*}
\int_0^1 m(t)\,dH(t) &= \sum_j \alpha_j m(s_j) \\
    &= \sum_j \alpha_j \sum_i \alpha_i R(s_j,s_i).
\end{align*}
These two expressions are the same, as claimed in (\ref{eq:normm}).

The first term of (\ref{eq:h}) is more subtle. It naively equals $\langle X, m
\rangle$, leading to (\ref{eq:LRmnemonic}). The difficulty is that a realisation
of $X(t)$ is almost surely not an element of the RKHS $V$, making $\langle X, m
\rangle$ technically meaningless. As advocated in~\cite{Kailath:1971hk} though, the
mnemonic $\langle X, m \rangle$ is useful for guessing the solution then
verifying its correctness by showing it satisfies (\ref{eq:Emt}). Moreover, it
is intuitive to interpret the detection problem as $X \mapsto \langle X, m
\rangle$ since this projects the received signal $X$ onto the signal $m$ to be
detected. The RKHS inner product gives the optimal projection.

\begin{example}
\label{ex:Nwp}
Assume $N(t)$ is the standard Wiener process with $N(0)=0$ and
$R(t,s) = \min\{s,t\}$. The associated RKHS $V$ is given in
Example~\ref{ex:wp}. The detection problem is non-singular
if and only if $m \in V$. Assuming $m \in V$, the second
term of the likelihood ratio (\ref{eq:h}) is $\frac12\|m\|^2 = \frac12
\int_0^1 (m'(s))^2\,ds$. The mnemonic $\langle X, m \rangle$
suggests the first term of (\ref{eq:h}) is $\int_0^1 X'(s) m'(s)\,ds$. This is
technically incorrect because a Wiener process is nowhere differentiable, that is, $X'(t)$ is
not defined. However, assuming $m$ is sufficiently smooth and applying
integration by parts leads to the alternative expression $m'(1)X(1) - \int_0^1 
X(s) m''(s)\,ds$. If $X$ lies in $V$ then $\int_0^1 X'(s) m'(s)\,ds$ and
$m'(1)X(1) - \int_0^1 X(s) m''(s)\,ds$ are equal. The advantage of the latter is
it makes sense when $X$ is a realisation of a Wiener process. To verify this is
correct, let $U = m'(1) N(1) - \int_0^1 N(s) m''(s)\,ds$.  Then
\begin{align*}
E[UN(t)] &= m'(1) R(t,1) - \int_0^1 R(t,s) m''(s)\,ds \\
    &= m'(1) t - \int_0^t s\, m''(s)\,ds - t \int_t^1 m''(s)\,ds \\
    &= m'(1) t - (t m'(t) - m(t)) - t(m'(1) - m'(t)) \\
    &= m(t),
\end{align*}
verifying the correct linear operator has been found.
Here, integration by parts has been used, as has the
fact that $m(0) = 0$ because $m$ is assumed to be in $V$.
(If $m$ is not sufficiently smooth to justify integration by
parts, it is possible to write $U$ in terms of a stochastic
integral.)
\end{example}

It remains to demystify how the technically nonsensical $\langle X, m \rangle$
can lead to the correct linear operator being found. At the heart of the matter
is the concept of an abstract Wiener space. For the Wiener process in
Example~\ref{ex:Nwp} the setting is the following. Let $\mathcal{B}$ be the
Banach space of continuous functions $f \colon [0,1] \rightarrow \reals$
satisfying $f(0)=0$, equipped with the sup norm $\|f\| = \sup_{t \in [0,1]} |
f(t) |$. Let $\mathcal{H}$ be the (reproducing kernel) Hilbert space $V$ from
Example~\ref{ex:Nwp}. Treated as sets, observe that $\mathcal{H} \subset
\mathcal{B}$. The inclusion map $\iota\colon \mathcal{H} \rightarrow
\mathcal{B}$ is continuous because the norm on $\mathcal{B}$ is weaker than the
norm on $\mathcal{H}$. Furthermore, $\iota(\mathcal{H})$ is dense in
$\mathcal{B}$. The dual map $\iota^* \colon \mathcal{B}^* \rightarrow
\mathcal{H}^*$ is injective and dense; here, $\mathcal{B}^*$ is the
space of  all bounded linear functionals on $\mathcal{B}$, and $\iota^*$ takes a
functional $\mathcal{B} \mapsto \reals$ and returns a functional $\mathcal{H}
\mapsto \reals$ by composing the original functional with $\iota$. Put simply, a
linear functional on $\mathcal{B}$ is \emph{a fortiori} a linear functional on
$\mathcal{H}$ because $\mathcal{H}$ is a subset of $\mathcal{B}$. (The functions
$\iota$ and $\iota^*$ are used to account for the different topologies on the
two spaces.)

Every realisation of the noise process $N$ lies in $\mathcal{B}$. If $g \in
\mathcal{B}^*$ is a linear functional then, as in the finite dimensional case
(\S\ref{sec:difd}), $g(N)$ is a zero-mean Gaussian random variable. For
technical reasons, the geometry of $\mathcal{B}$ cannot correspond to the
statistics of $N$. Instead, it is the geometry of $\mathcal{H}$ that corresponds
to the statistics of $N$. The variance of $g(N)$ is $\| \iota^*(g)
\|_{\mathcal{H}}$. In particular, the Gaussian distribution of $N$ in
$\mathcal{B}$ is radially symmetric with respect to the norm $g \mapsto
\|\iota^*(g)\|_{\mathcal{H}}$. Believably then, the detection problem can
be solved by projecting $X$ onto the one-dimensional space
spanned by the message $m$, as in \S\ref{sec:difd}.

If there exists a $g \in \mathcal{B}^*$ such that $\iota^*(g)$ is the linear
operator $\cdot \mapsto \langle \cdot, m \rangle$ then $g(X)$ is the precise
definition of the mnemonic $\langle X,m \rangle$ in (\refeq{eq:LRmnemonic}).
This corresponds to the case when $m$ was assumed to be sufficiently smooth in
Example~\ref{ex:Nwp}. Otherwise, the extension of $\cdot \mapsto \langle \cdot,
m \rangle$ from $\mathcal{H}$ to $\mathcal{B}$ is not a \emph{bounded} linear
functional. It turns out though that there exists a subspace $\mathcal{B}'$ of
$\mathcal{B}$ such that $\langle \cdot, m \rangle$ extends to $\mathcal{B}'$ and
the probability of $N$ lying in $\mathcal{B}'$ is unity. In particular, $\langle
X, m\rangle$ can be defined to be the extension of $\langle \cdot, m\rangle$ by
continuity: if $X_k \rightarrow X$ in $\mathcal{B}$, where the $X_k$ belong to
$\mathcal{H}$, then $\langle X, m \rangle$ is defined to be $\lim_k \langle X_k,
m \rangle$. This limit exists with probability one. (Mathematically,
$\mathcal{B}^*$ can be treated as a subset of $L^2(\mathcal{B};\reals)$, where
the $L^2$-norm of $g\colon \mathcal{B} \rightarrow \reals$ is the square-root of
$E[g(N)^2]$. Then $(\iota^*)^{-1}\colon \iota^*(\mathcal{B}^*) \subset
\mathcal{H}^* \rightarrow \mathcal{B}^* \subset L^2(\mathcal{B};\reals)$ extends
by continuity to a linear isometry from $\mathcal{H}^*$ to
$L^2(\mathcal{B};\reals)$. The image of $\langle \cdot, m\rangle$ under this
extension is the true definition of $X \mapsto \langle X, m \rangle$.)

\paragraph*{Remark} Here is another perspective.
By~\cite[Lemma 1.2]{Pitcher:1960ub}, the process $N(t)$ lies
in $\mathcal{H} = L^2([0,1])$ with probability one: $\int_0^1 N^2(t)\,dt <
\infty$ almost surely. Furthermore, the same lemma proves that if the detection
problem is non-singular then $m(t)$ lies in $\mathcal{H}$. Therefore, without
loss of generality, the sample space can be taken to be $\mathcal{H}$. As
in~\cite{da2006introduction}, a Gaussian measure $N_{\mu,Q}$ can be placed on
the separable Hilbert space $\mathcal{H}$, where $\mu \in \mathcal{H}$ and
$Q\colon \mathcal{H} \rightarrow \mathcal{H}$ represent the mean and covariance
of the Gaussian measure. The signal detection problem is deciding between the
probability measures $N_{0,Q}$ and $N_{m,Q}$. The Cameron-Martin
formula~\cite[Theorem 2.8]{da2006introduction} yields the likelihood
ratio. Furthermore, the image of $\mathcal{H}$ under the square-root of $Q$,
denoted $Q^{\frac12}(\mathcal{H})$, is precisely the RKHS $V$ discussed earlier,
and goes by the name Cameron-Martin space.

% ----------------------------------------------------
%                       PART II (Amblard)
% ----------------------------------------------------

% ----------------------------------------------------
%                       CHAPTER
% ----------------------------------------------------
\newcommand{\va}{\mbox{\boldmath$a$}}
\newcommand{\vh}{\mbox{\boldmath$h$}}
\newcommand{\vk}{\mbox{\boldmath$k$}}
\newcommand{\vv}{\mbox{\boldmath$v$}}
\newcommand{\vy}{\mbox{\boldmath$y$}}
\newcommand{\vx}{\mbox{\boldmath$x$}}

\newcommand{\vA}{\mbox{\boldmath$A$}}
\newcommand{\vC}{\mbox{\boldmath$C$}}
\newcommand{\vI}{\mbox{\boldmath$I$}}
\newcommand{\vK}{\mbox{\boldmath$K$}}
\newcommand{\vN}{\mbox{\boldmath$N$}}
\newcommand{\vP}{\mbox{\boldmath$P$}}
\newcommand{\vX}{\mbox{\boldmath$X$}}
\newcommand{\vH}{\mbox{\boldmath$H$}}

\newcommand{\vY}{\mbox{\boldmath$Y$}}

\newcommand{\vones}{{\bf 1}}

\newcommand{\vzeta}{\mbox{\boldmath$\zeta$}}

\newcommand{\vmu}{\mbox{\boldmath$\mu$}}
\newcommand{\valpha}{\mbox{\boldmath$\alpha$}}
\newcommand{\vbeta}{\mbox{\boldmath$\beta$}}
\newcommand{\vpi}{\mbox{\boldmath$\pi$}}

\newcommand{\R}{\mathbb{R}}
\newcommand{\Z}{\mathbb{Z}}
\newcommand{\N}{\mathbb{N}}

\newcommand{\cov}{\mbox{Cov }}
\newcommand{\var}{\mbox{Var}}
\newcommand{\trace}{\mbox{Tr}}
\newcommand{\diag}{\mbox{Diag}}

\renewcommand{\H}{{\cal H}}

\chapter{Embeddings of Random Realisations}
\label{sec:embR}

Random variables are normally taken to be
real-valued~\cite{williams1991probability}.  Random vectors and
random processes are collections of random variables.  Random
variables can be geometrically arranged in a RKHS according to pairwise
correlations, as described in \S\ref{sec:rsp}. A random variable is
mapped to a single point in the RKHS.

A different concept is embedding the \emph{realisations} of a random
variable into a RKHS.  Let $X$ be an $\X$-valued random variable
encoding the original random variables of interest.  For example,
the random variables $Y$ and $Z$ are encoded by $X = (Y,Z)$ and $\X
= \reals^2$. Let $V \subset \reals^\X$ be a RKHS whose elements are
real-valued functions on $\X$. The realisation $x \in \X$ of $X$
is mapped to the element $K(\cdot,x)$ of $V$, where $K$ is the
kernel of $V$.

The kernel does not encode statistical information as in \S\ref{sec:rsp}.
Its purpose is ``pulling apart'' the realisations of $X$. A classic
example is embedding the realisations of $X=(Y,Z) \in \reals^2$
into an infinite-dimensional RKHS by using the Gaussian kernel
(\S\ref{sec:gk}). \emph{Linear} techniques for classification and
regression applied to the RKHS correspond to \emph{nonlinear}
techniques in the original space $\reals^2$.

If the realisations are pulled sufficiently far apart then a
probability distribution on $X$ can be represented uniquely by a
point in the RKHS. The kernel is then said to be characteristic.
Thinking of distributions as points in space can be beneficial
for filtering, hypothesis testing and density estimation.

\noindent\textbf{Remark:}
Information geometry also treats distributions as points in space.
There though, the differential geometry of the space encodes intrinsic
statistical information.  The relative locations of distributions
in the RKHS depend on the chosen kernel.

RKHS was first presented in the finite dimensional setting, and this choice is done here as well, for presenting  the interest  of embedding random variables in  a RKHS. In section \ref{basics:sec}, 
the theory of random variables with values in Hilbert space is outlined as a preparation. Next,    embeddings of random variables
in RKHS are presented and studied from the probabilistic as well as the empirical points of view.  Some generalizations and considerations needed in chapter \ref{ch:appemb} close the chapter.

\section{Finite Embeddings}
\label{sec:fe}

Definitions and concepts concerning the embedding of random
realisations in a RKHS are more easily grasped when the number of
possible outcomes is finite.

Let $X$ be an $\X$-valued random variable where $\X = \{1,2,\cdots,n\}$.
Its distribution is described by $p_1,\cdots,p_n$ where $p_x$ is
the probability of $X$ taking the value $x$.  Choose the RKHS 
to be Euclidean space $\reals^n$.  Its kernel is the identity
matrix.  The realisation $x \in \X$ is therefore embedded as the
unit vector $e_x$. Effectively, the elements $e_1,\cdots,e_n$
of the RKHS are chosen at random with respective probabilities
$p_1,p_2,\cdots,p_n$. This choice of kernel leads to
the embedded points being uniformly spaced apart from each
other: $\|e_i - e_j \| = \sqrt{2}$ for $i \neq j$.

\noindent\textbf{Remark:} This embedding is used implicitly
in~\cite{elliott1995hidden} to great advantage. The elements of
$\X=\{1,\cdots,n\}$ are the possible states of a Markov chain.
Representing the $i$th state by the vector $e_i$ converts nonlinear
functions on $\X$ into linear functions on $\reals^n$, as
will be seen presently.

When $n>1$, the distribution of $X$ cannot be recovered from its
expectation $E[X] = p_1 + 2p_2 + \cdots + n p_n$.  However, if
$\check X$ is the embedding of $X$ then $E[\check X] = p_1 e_1 +
p_2 e_2 + \cdots + p_n e_n$. The distribution of $X$ is fully
determined by the expectation of its embedding! 
\emph{Points in the RKHS can represent distributions.}

This has interesting consequences, the first being that the
distribution of $X$ can be estimated from observations $\check x_1,
\cdots, \check x_M \in \reals^n$ by approximating $E[\check X]$ by
the ``sample mean'' $\frac1M \sum_{m=1}^M \check x_m$.  Hypothesis
testing (\S\ref{sec:sigdet}) reduces to testing which side of a
hyperplane the sample mean lies: the two distributions under test
are thought of as two points in the RKHS. An observation, or the
average of multiple observations, is a third point corresponding
to the estimated distribution. Intuitively, the test is based on
determining which of the two distributions is closer to the
observation.  The following lemma validates this.  The probabilities
$p_i$ are temporarily written as $p(i)$.

\begin{lemma}
Let $p$ and $q$ be two probability
distributions on the $\X$-valued random variable $X$,
where $\X=\{1,\cdots,n\}$. Given $M$ observations
$x_1,\cdots,x_M$ of $X$, the log likelihood ratio
$\ln \left\{\prod_{m=1}^M \frac{p(x_m)}{q(x_m)} \right\}$ equals
$\langle v, \sum_{m=1}^M \check x_m \rangle$ where
$v_i = \ln \frac{p(i)}{q(i)}$.
\end{lemma}
\begin{proof}
If $x_m = i$ then $\check x_m = e_i$ and $\langle v, \check x_m \rangle
= v_i = \ln \frac{p(i)}{q(i)} = \ln \frac{p(x_m)}{q(x_m)}$.
Therefore $\ln \left\{\prod \frac{p(x_m)}{q(x_m)} \right\}
= \sum \ln \frac{p(x_m)}{q(x_m)} = \sum \langle v, \check x_m \rangle$.
\end{proof}

More advanced probability texts tend to favour working with
expectations (and conditional expectations) over probabilities (and
conditional probabilities). This generally involves expressions of
the form $E[f(X)]$ where $f\colon \X \rightarrow \reals$ is a
``test'' function; if $E[f(X)]$ is known for sufficiently many
functions $f$ then the distribution of $X$ can be inferred.

An arbitrary function $f\colon \X \rightarrow \reals$ is fully
determined by its values $f(1),\cdots,f(n)$.  Let $\check f\colon
\reals^n \rightarrow \reals$ be the unique \emph{linear} function
on $\reals^n$ satisfying $\check f(e_x) = f(x)$ for $x \in \X$.
Then
\begin{equation}
E[\check f(\check X)] = \sum_x p_x \check f(e_x) = \sum_x p_x f(x) = E[f(X)].
\end{equation}
The advantage of $E[\check f(\check X)]$ over $E[f(X)]$ is that
$E[\check f(\check X)] = \check f( E[\check X] )$ because $\check
f$ is linear. Once $E[\check X]$ is known, $E[f(X)] = \check
f(E[\check X])$ is effectively known for any $f$. By comparison,
$E[f(X)]$ generally cannot be deduced from $E[X]$.

Correlations take the form $E[f(X)\, g(X)]$. As above, this is
equivalent to $E[\check f(\check X) \, \check g(\check X)]$.
Being linear, $\tilde f(\check X)$ and $\tilde g(\check X)$ can
be rewritten using the Euclidean inner product:
there exist $v_f,v_g \in \reals^n$ such that
$\check f(\check X) = \langle \check X, v_f \rangle$
and $\check g(\check X) = \langle \check X, v_g \rangle$.
Moreover, $v_f$ and $v_g$ can be expressed as linear combinations
$\sum_i \alpha_i e_i$ and $\sum_j \beta_j e_j$
of the unit basis vectors $e_1,\cdots,e_n$. In particular,
\begin{align}
E[f(X)\, g(X)] &=
    E[ \langle \check X, \sum_i \alpha_i e_i \rangle
        \langle \check X, \sum_j \beta_j e_j \rangle ] \\
    \label{eq:abC}
    &=  \sum_{i,j} \alpha_i \beta_j E[ \langle \check X, e_i \rangle
        \langle \check X, e_j \rangle ].
\end{align}
Once the ``correlation'' matrix $C$ given by $C_{ij} = E[ \langle
\check X, e_i \rangle \langle \check X, e_j \rangle ]$ is known,
$E[f(X)\,g(X)]$ can be determined from (\ref{eq:abC}).

\section{Random elements in Hilbert spaces}
\label{basics:sec}

Prior studying the case of the embedding in a RKHS, let us consider the case of random variables which take values in a (separable) Hilbert space $ V$ \cite{BerlTA04,Fort95}. A lemma states that a function from a probability space $(\Omega,{\cal F},P)$ to $ V$ is a random variable with values in $ V$ if and only if  $x^*(X)$ is a real random variable for any linear form $x^*\in  V^*$, the dual of $ V$. Since the dual of a Hilbert space can be identified to itself, the linear form simply writes $x^*(X)=\big< x , X \big>$ where $x\in V$.  

The linear form on $ V$ defined by $\ell_X(x)=E\big[\big< x ,  X \big>\big]$ is bounded whenever $E[\|X\|]<+\infty$.
Indeed,  $\big| \ell_X(x)  \big|\leq \| x \|E\|X\| $ thanks to the Cauchy-Schwartz inequality. Thus,  Riesz representation theorem \cite{DebnM05}
shows the existence of a  unique element $m_X$ of $ V$ such that $E\big[\big< x ,  X \big>\big]=\big< x , m_X \big>$. $m_X$ is the mean element and is denoted as $E[X]$. 

Denote the space  of square integrable random elements of $ V$ as $L^2_ V(P)$ (a short notation for $L^2_ V(\Omega,{\cal F},P)$.)
It is the space of $ V$ valued random variables on $(\Omega,{\cal F},P)$ such that 
 $E\|X\|^2<+\infty$. When  equipped with 
$\big< X , Y \big>_{L^2} := E\big[\big< X ,  Y \big>_ V\big]$,  $L^2_ V(P)$ is itself a Hilbert space.

The covariance operator is a linear operator from $ V$ to $ V$ defined by 
$\Sigma_X: x \longmapsto \Sigma_X(x) := E\big[\big<x  , X -m_X \big>(X-m_X)\big]$. It is bounded whenever $X\in L^2_ V(P)$. 
To see this, suppose  for the sake of simplicity that $EX=0$. Recall that 
that the operator norm is defined as $\|\Sigma_X\| = \sup_{\|x\|\leq 1} \|  \Sigma_X(x)  \|$ and that 
$\|\Sigma_X\| = \sup_{\|x\|\leq 1,\|y\|\leq 1} |\big< y ,  \Sigma_X(x)   \big> |$.
But applying Cauchy-Schwartz inequality leads to
\begin{eqnarray*}
\big|\big< y ,  \Sigma_X(x)   \big> \big| &=& \big|E\big[\big< x ,  X\big>\big< y \big| X\big>\big]\big|\\
&\leq& E\big[\|x\|\|y\|\|X\|^2  \big]  = \|x\|\|y\|E\big[\|X\|^2  \big] 
\end{eqnarray*}
which shows that $\|\Sigma_X\| <+\infty$ whenever $X\in 
L^2_ V(P)$ .

Likewise, we can define a cross-covariance operator between two elements $X,Y$ of $L^2_ V(P)$ by the bounded linear operator from $ V$ to itself defined by $\Sigma_{YX}(x) := E[\big< x ,  (X-m_X)\big>(Y-m_Y)]$. The adjoint operator defined by 
$\big<  \Sigma_{YX}^*(y),  x\big> =\big<y,  \Sigma_{YX}(x)\big>$ is then $\Sigma_{XY}$ since by definition $\Sigma_{XY}(y)=E[\big< y , (Y-m_Y) \big> (X-m_X)]$. The two operators are completely defined by $
\big< y ,  \Sigma_{XY}(x) \big> = E[\big< x ,  X\big>\big< y , Y \big>]$ (if the mean elements are assumed to be equal to zero.) 
The cross-covariance can even be generalized to the case of two different Hilbert spaces. Consider two
random variables $X$ and $Y$ defined on a common probability space $(\Omega,{\cal F},P)$, but taking values in two different 
Hilbert spaces $ V_x$ and $ V_y$. The cross-covariance operator has the same definition, but $\Sigma_{YX}$ has a  domain of definition included in $ V_x$ and a range included in $ V_y$. 

Covariance and cross-covariance operators have furthermore the properties to be nuclear as well as Hilbert-Schmidt operators. 
A Hilbert-Schmidt operator $\Sigma$ from a Hilbert space $ V_1$ to another $ V_2$ is such that $\sum_i \|\Sigma e_i\|^2 = \sum_{ij} \big< f_j ,  \Sigma e_i\big>^2< +\infty$ where $\{e_i\},\{f_j\}$ are orthornomal bases of respectively  $ V_1$ and $ V_2$.
$\| \Sigma \|_{HS} = \sum_i \|\Sigma e_i\|^2 =\sum_{i,j}\big< f_j ,  \Sigma e_i\big>^2 $ is the Hilbert-Schmidt norm (it can be shown independent of the choice of the bases.)

 $\Sigma$ is nuclear if there exist    orthornomal bases $\{e_i\},\{f_j\}$ and a sequence $\{\lambda_i\}$ verifying $\sum_i |\lambda_i | < +\infty$
such that  $\Sigma = \sum_i \lambda_i e_i \otimes f_i $, where the tensorial product is defined as $(e_i \otimes f_i )(x)= \big< x ,  e_i \big> f_i$.  Then $\| \Sigma \|_N = \sum_i |\lambda_i | $ is the nuclear norm. It also holds $\| \Sigma \|_{HS}= \sum_j | \lambda_i|^2$. Furthermore, the three norms of operators so far introduced satisfy the  inequalities $\| \Sigma \| \leq \| \Sigma \|_{HS} \leq \| \Sigma \|_N$.

%To conclude this rapid presentation, what happens if the space is $\R^n$. Then we work with usual vectors and the usual euclidean inner product to write that $\big< y ,  \Sigma_{YX}(x) \big> = y^\top \Gamma_{YX} x =x^\top \Gamma_{XY} y$ where $\Gamma_{..}$ is the covariance matrix! Likewise, the mean element is the mean vector!
%
%
%

\section{Random elements in reproducing kernel Hilbert spaces}

The theory of random elements in Hilbert space is  specialized in the sequel to the case where $V$ is a reproducing kernel Hilbert space. Specifically,
the embedding of an $\X$ valued  random variable  is developed. 

Consider the embedding of an $\X$ valued random variable $X$ into a
RKHS $V \subset \reals^\X$ defined by its kernel $K\colon \X \times
\X \rightarrow \reals$. This mean that each realisation $x \in X$ is mapped to
the element $K(\cdot,x) \in V$. This is equivalent to defining a random variable 
 $\check X = K(\cdot,X)$ with values in the RKHS $V$.

The expected value of $\check X = K(\cdot,X)$ 
was defined coordinate-wise
in \S\ref{sec:fe}. This generalises immediately to
defining $E[\check X]$ by
\begin{equation}
\label{eq:EcX}
\langle E[\check X], K(\cdot,z) \rangle =
E[ \langle \check X, K(\cdot, z) \rangle ],\qquad z \in \X,
\end{equation}
where the left-hand side is the $z$th coordinate of $E[\check X]$
and the right-hand side is the expected value of the $z$th
coordinate of $\check X$.
Evaluating $E[\check X]$ is straightforward in principle: if $E[X]
= \int x\,d\mu(x)$ then $E[\langle \check X, K(\cdot,z) \rangle ]
= \int \langle K(\cdot,x), K(\cdot,z) \rangle\,d\mu(x) = \int K(z,x)
\,d\mu(x)$.\\

%\begin{example}
%Let $X \sim N(a,c^2)$ be a Gaussian random variable
%with mean $a$ and variance $c^2$.
%With respect to the Gaussian kernel
%$K(x,y) = \exp\{-\frac12(x-y)^2\}$, the expected value of
%the embedding $\check X = K(\cdot,X)$ of $X$ is
%the function $z \mapsto \int_{-\infty}^\infty e^{-\frac12(z-x)^2}
%\frac1{\sqrt{2\pi}c} e^{-\frac12(\frac{x-a}c)^2}\,dx$.
%This simplifies to $z \mapsto \frac1{\sqrt{1+c^2}}
%\exp\{-\frac12\frac{(z-a)^2}{1+c^2}\}$, an
%unnormalised Gaussian centred at $a$ with ``variance''
%$\sqrt{1+c^2}$. Note $a$ and $c$ are
%deducible from $E[\check X]$.
%
%Given observations $x_1,\cdots,x_M$ of $X$, the sample mean
%of their embeddings is
%\begin{align*}
%\frac1M \sum_{m=1}^M \check x_m &= \frac1M \sum_{m=1}^M K(\cdot,x_m) \\
%    &= \frac1M \sum_{m=1}^M \exp\{-\tfrac12(\cdot - x_m)^2\}.
%\end{align*}
%Unlike in \S\ref{sec:fe}, the sample mean does not take the
%same form as the true mean; estimating $a$ and $c$ from the
%sample mean would entail fitting $z \mapsto \frac1{\sqrt{1+c^2}}
%\exp\{-\frac12\frac{(z-a)^2}{1+c^2}\}$ to
%$z \mapsto \frac1M \sum_{m=1}^M \exp\{-\tfrac12(z - x_m)^2\}$
%and would not be preferable to estimating $a$ and $c$ directly
%from $x_1,\cdots,x_M$ in the traditional way. 
%\end{example}

\noindent\textbf{Remark:}
For $E[\check X]$ to be well-defined, $x \mapsto \langle K(\cdot,x),
K(\cdot,z) \rangle = K(z,x)$ must be measurable for every $z \in \X$.
By symmetry of the kernel, this is the same as $K(\cdot,x)$ being
measurable for every $x \in \X$. By~\cite[Theorem 90]{Berlinet:2012uv},
this is equivalent to every element of the RKHS being measurable.\\

While (\ref{eq:EcX}) defines $E[\check X]$ coordinate-wise, a
subtlety is whether $E[\check X] \in \reals^\X$ belongs to the RKHS
$V$. It is usually desirable to require $E[\|\check X\|]$ to be
finite, and this suffices for $E[\check X]$ to be an element of the
RKHS. Rather than prove this directly, an equivalent definition of
$E[\check X]$ is given, using the theory outlined  in the preceding section.

For $v \in V$, $| \langle \check X, v \rangle | \leq \| \check X
\| \|v \|$ by the Cauchy-Schwartz inequality. Therefore, $E[|\langle
\check X, v \rangle |] \leq E[\| \check X\|] \|v\|$. In particular,
if $E[\| \check X\|] < \infty$ then $E[|\langle \check X, v \rangle|]
< \infty$ and $E[\langle \check X, v \rangle]$ is well-defined.
Moreover, $v \mapsto E[\langle \check X, v \rangle]$ is a bounded
linear function. The Riesz representation theorem implies
the existence of an element $m_X \in V$ such that this
linear map is given by $v \mapsto \langle m_X, v \rangle$.
Then $E[\check X]$ is defined to be $m_X$.

Henceforth, this definition of $E[\check X]$ will be adopted,
so that $E[\check X]$ is an element of the RKHS by definition.
The notations $m_X$ and $E[\check X]$ will be used interchangeably
and will be called the \emph{mean element} of $X$.\\

\noindent\textbf{Remark:}
Condition (\ref{eq:EcX}) only required $\langle E[\check X], v \rangle
= E[ \langle \check X, v \rangle ]$ to hold for $v = K(\cdot,z)$,
$z \in \X$. However, the $K(\cdot,z)$ are dense in $V$, hence
the two definitions agree provided $E[\|\check X\|] < \infty$.\\

Finally, let $f\in V$. The mean element completely determine $E[f(X)]$. This is easy to show 
because $E[f(X)]= E[\big< K(.,X) ,  f \big>]=E[\big< \check X ,  f \big>] = \langle m_X,  f \rangle$.
Therefore, knowing $m_x$ allows to evaluate the mean of any transformation of $X$, provided the transformation belongs to $V$.\\

It is convenient to complement $m_X$ with a covariance operator
$\Sigma_X$ capturing $E[(f(X)-E[f(X)])\,(g(X)-E[g(X)])]$ for
$f,g \in V$. Observe
\begin{align*}
\MoveEqLeft
E[(f(X)-E[f(X)])\,(g(X)-E[g(X)])] \\
    &= E[\langle f, \check X-m_X \rangle \langle g, \check X-m_X \rangle ] \\
    &= E[\langle f, \langle g, \check X-m_X \rangle (\check X-m_X) \rangle ] \\
    &= \langle f, E[\langle g, \check X-m_X \rangle (\check X-m_X) ] \rangle \\
    &= \langle f, \Sigma_X(g) \rangle
\end{align*}
where the linear operator $\Sigma_X\colon V \rightarrow V$ is given by
\begin{equation}
\Sigma_X(g) = E[\langle g, \check X-m_X \rangle (\check X-m_X) ].
\end{equation}
Provided $E[\|\check X\|^2] < \infty$, $\Sigma_X$ is well-defined and
its operator norm is finite: $\| \Sigma_X \| = \sup_{\|g\|=1} \|\Sigma_X(g)\|
= \sup_{\|f\|=\|g\|=1} \langle f, \Sigma_X(g) \rangle$.\\

\noindent\textbf{Remark:}
If $E[\|\check X\|^2] = E[K(X,X)] < \infty$ then $E[\|\check X\|]
< \infty$, a consequence of Jensen's inequality~\cite[Section
6.7]{williams1991probability}.  In particular, $E[\|\check X\|^2] <
\infty$ ensures both $m_X$ and $\Sigma_X$ exist.
See also~\cite[\S4.5]{Berlinet:2012uv}.

\section{Universal and characteristic kernels}

%Choice of the kernel important...
%
%
%
% Spacing the possible realisations
%uniformly apart in $V$ is achieved with the kernel given by $K(x,x)=1$
%and $K(y,x) = 0$ for $y \neq x$, as verified by (\ref{eq:sqdist}).
%However, this is undesirable because the topology is destroyed and
%$V$ is too large (that is, $V$ is not separable).  If $\X$ is
%Euclidean space, the Gaussian kernel (\S\ref{sec:gk}) and other
%bounded continuous kernels ensure $V$ is not too large and $K(\cdot,x_i)
%\rightarrow K(\cdot,x)$ if $x_i \rightarrow x$.  (By~\cite[Section
%1.5]{Berlinet:2012uv}, if $\X$ is a separable metric space, if $K$
%is bounded and $K(\cdot,x)$ continuous, then $V \subset \reals^\X$
%is a separable RKHS of continuous functions.)

An aim of embedding points of a space into a RKHS is to reveal features that are difficult to study or see in the initial space. 
Therefore, the RKHS has to be sufficiently rich to be useful in a given application. 
For example, the fact that the mean element reveals $E[f(X)]$ simply as $\big<  m_x, f \big>$ is an interesting property 
essentially if this can be applied for a wide variety of functions $f$. It is thus desirable that $V\subset \R^\X$ is a sufficiently rich function space.
The richness of the RKHS is obviously provided by some properties of the kernel, since a kernel 
gives rise to a unique RKHS. 

Two important notions occur: the notion of universality, defined initially by Steinwart \cite{Stei01} and the notion of characteristic 
kernel.

Universality  is linked to the denseness of the RKHS into a target space of functions and is therefore linked to the ability of the functions of the RKHS to approximate functions in the target space. Since it depends on the kernel, on the initial space and on the target space, there exist different definitions of universality.  We will need in the sequel kernel universal in the sense that their reproducing kernel Hilbert space is dense in the space of  continuous functions. In some applications, universality refers to denseness in $L^p$ spaces. This is particularly important when dealing with embedding of square integrable random variables. Thus the following definition is considered.
\begin{definition}
Let $\X$ be a locally compact Hausdorff space (such as Euclidean
space).  Denote by $C_0(\X)$ the class of real-valued continuous
functions on $\X$ vanishing at infinity.  Let $K\colon \X \times
\X \rightarrow \reals$ be a bounded kernel for which $K(\cdot,x)
\in C_0(\X)$ for all $x \in \X$.  Then $K$, or its corresponding
RKHS $V$, is \emph{universal} if $V$ is dense in $C_0(\X)$ with
respect to the uniform norm.  (Actually, universality has several
different definitions depending on the class of functions of most
interest~\cite{SripFL11}.)
\end{definition}

The notion of characteristic kernel is useful when embedding probability
measures and is linked to the ability to discriminate two original measures in
the RKHS. Let $K:\X \times \X \longrightarrow \R$ a kernel on $\X$. Let ${\cal
P}$ the set of probability measures on $\X$ equipped with its Borel $\sigma$
algebra. Then $K$  (and thus $V$) is characteristic if and only if the map
from ${\cal P}$ to $V$ defined by $P \mapsto E_P[K(., X)]$ is injective, where
$X$ is any random variable distributed under $P$.

There exist links between the two notions which are studied in
\cite{SripFL11}. For this paper and applications in signal processing, it
suffices to know that some well-known and often used kernels are universal and
characteristic. Importantly, it is shown in  \cite{SripFL11}, Proposition 5,
that for radial kernels on $\R^d$, kernels are universal (in many diffferent
sense)  if and only if they are characteristic. Furthermore the proposition
gives other necessary and sufficient condition for universality. For example,
strict positive-definiteness of the radial kernel insures it is
characteristic.

Thus, important kernels such as the Gaussian kernel, $K(x,y)=\exp(-\sigma \|
x-y \|^2)$, the inverse multiquadrics $K(x,y)= (c +  \| x-y \|^2)^{-\beta},
\beta>d/2)$ are characteristic and universal.

\section{Empirical estimates of the mean elements and covariance operators}

The mean element $m_X$ can be estimated by the sample mean
\begin{equation}
\hat m_X^N = \frac1N \sum_{i=1}^N K(\cdot,x_i)
\end{equation}
where $x_1,\cdots,x_N$ are independent realisations of $X$.  The
law of large numbers implies $\hat m_X^N$ converges to $m_X$ almost
surely~\cite{HoffmannJorgensen:1976fd}. It also converges
in quadratic mean:
\begin{align*}
E[\|\hat m_X^N - m_X\|^2] &=
    \frac1{N^2} \, E\left[\left\langle \sum_{i=1}^N K(\cdot,X_i) - m_X,
        \sum_{j=1}^N K(\cdot,X_j) - m_X \right\rangle \right] \\
    &= \frac1{N^2} \sum_{i,j=1}^N E\left[K(X_j,X_i)
        - m_X(\check X_j) - m_X(\check X_i) + \|m_X\|^2 \right] \\
    &= \frac1{N^2} \sum_{i,j=1}^N E\left[K(X_j,X_i)
        - \|m_X\|^2 \right].
\end{align*}
If $i \neq j$ then $E[K(X_j,X_i)] = \|m_X\|^2$ by independence:
\begin{align*}
E[K(X_j,X_i)] &= E[\langle K(\cdot,X_i), K(\cdot,X_j) \rangle ] \\
    &= \langle E[K(\cdot,X_k)], E[K(\cdot,X_j)] \rangle \\
    &= \langle m_X, m_X \rangle.
\end{align*}
Therefore,
\[
E[\|\hat m_X^N - m_X\|^2] = \frac1N \left( E[K(X,X)] - \|m_X\|^2 \right)
    \rightarrow 0
\]
as $N \rightarrow \infty$, proving convergence in quadratic mean.

A central limit theorem also exists. For arbitrary $f \in V$,
$\langle f, \hat m_X^N \rangle = \frac1N \sum_{i=1}^N f(x_i)$ is
the sample mean of real-valued independent random variables $f(x_i)$ and
therefore $\sqrt{N} \langle f, \hat m_X^N - m_X \rangle$ converges to
a Gaussian random variable having zero mean and the same variance as $f(X)$.
This implies $\sqrt{N}(\hat m_X^N - m_X)$ converges weakly on $V$ to
a Gaussian distribution~\cite[Theorem 108, full proof and statement]{Berlinet:2012uv}.

Estimating the covariance of $\check X$ in $V$ can be
reduced to estimating the mean of $(\check X - m_X) \otimes (\check
X - m_X)$ in $V \otimes V$, as now explained.

Given $f,g \in V$, define $f \otimes g$ to be the bilinear continuous
functional $V \times V \rightarrow \reals$ sending $(u,v)$ to
$\langle f,u \rangle \langle g,v \rangle$. That is,
\begin{equation}
(f \otimes g)(u,v) = \langle f,u \rangle \langle g,v \rangle.
\end{equation}
The space generated by linear combinations of such functionals
is the ordinary tensor product of $V$ and $V$. Completing this
space results in the Hilbert space tensor product $V \otimes V$.
The completion is with respect to the inner product defined by
\begin{equation}
\langle f_1 \otimes g_1, f_2 \otimes g_2 \rangle =
    \langle f_1, f_2 \rangle \langle g_1, g_2 \rangle.
\end{equation}
If $V$ is a RKHS then $V \otimes V$ is also a RKHS.

\begin{theorem}
Let $V_1$ and $V_2$ be two RKHSs with kernels $K_1$ and $K_2$.  Then
the Hilbert space tensor product $V = V_1 \otimes V_2$ is a RKHS
with kernel $K((x_1,x_2),(y_1,y_2)) = K_1(x_1,y_1) K_2(x_2,y_2)$.
\end{theorem}
\begin{proof}
This is Theorem 13 of~\cite{Berlinet:2012uv}.
\end{proof}

Define $\Sigma_{XX} = E[ (\check X - m_X) \otimes (\check X - m_X) ]$.
It encodes the same information as $\Sigma_X$ because
$\Sigma_{XX}(f,g) = E[ \langle f, \check X - m_X \rangle
\langle g, \check X - m_X \rangle ]$. Being a mean, it can be estimated by
\begin{equation}
\hat m_{XX}^N = \frac1N \sum_{i=1}^N (K(\cdot,x_i) - \hat m_X^N)
\otimes (K(\cdot,x_i) - \hat m_X^N).
\end{equation}

\section{Generalisations and further considerations}

A covariance operator between two random variables $X$ and
$Y$ defined on a common probability space can be defined. The variables
are embedded in RKHSs $V_x$ and $V_y$ using kernels $K_x$ and $K_y$,
and the covariance operator $\Sigma_{YX}$ is the linear operator
from  $V_x$ to $V_y$ defined as $\Sigma_{YX} f=E[ \langle f, \check
X - m_X \rangle (\check Y - m_Y ) ] $. It also has the tensor
representation   $\Sigma_{YX} = E[ (\check X - m_X) \otimes (\check
Y - m_Y) ]$, is a mean in the tensor product space and can thus be
estimated by
\begin{equation}
\hat m_{YX}^N = \frac1N \sum_{i=1}^N (K_x(\cdot,x_i) - \hat m_X^N)
\otimes (K_y(\cdot,x_i) - \hat m_Y^N).
\end{equation}

% TODO - define nuclear operators
% TODO - funny symbol in equation below
The covariance operators are nuclear operators and  Hilbert-Schmidt
operators (see section \ref{basics:sec}.)  This fact allows
to construct measures of independence and of conditional independence
in an elegant and efficient way, as developed in 
\S\ref{testindep:ssec}.  As nuclear operators, they admit a nuclear
decomposition as
\begin{eqnarray*}
\Sigma = \sum_i \lambda_i e_i \otimes f_i
\end{eqnarray*}
where $\{e_i\},\{f_j\}$ are orthonormal bases of the spaces of
interest.

If the operator is the covariance operator of a random variable
embedded in the space, it is a positive operator. The $\lambda_i$
are the eigenvalues and are positive or zero, and the $\{e_i\}$ are
the eigenfunctions of the operator.  In this case, the Hilbert-Schmidt
theorem  \cite{DebnM05} states that any element of the space admits the decomposition
\begin{eqnarray*}
x= \sum_i \lambda_i \big< x ,  e_i\big> e_i   +  x'
\end{eqnarray*}
where $x'  \in {\cal N}(\Sigma)$ is in the null space of the
covariance operator. The range of the operator  ${\cal R}(\Sigma)$
is spanned by the eigen vectors $\{ e_i \}$.  Thus the range of
$\Sigma$ is the subset of $V$ of those functions $f\in V$ not in
the null space of $\Sigma$ verifying $\sum_i \lambda_i^2 \big< f
, e_i \big>^2 < +\infty$. Restricting the domain of $\Sigma$ to
the space of function of $V$ that can be written $\sum_i \big< f
,  e_i\big>/ \lambda_i$ for some $f$ in the range of $\Sigma$,
we define a bijective restriction, and the inverse is unambiguously
defined as $\Sigma^{-1}f = \sum_i \big< f ,  e_i\big>/ \lambda_i$
for $f\in {\cal R}(\Sigma)$. The operator is then said invertible on its range.
 Note that since $\Sigma $ is positive
the $\lambda_i$ considered are strictly positive, $\lambda_i=0$
characterizing members of the null space.

Hilbert-Schmidt operators are compact. If a compact operator is
invertible, then necessarily its inverse is unbounded, otherwise
$A A^{-1}=I$ would be compact. But $I$ is not compact in infinite dimension.
A compact operator transforms a bounded set into a precompact set
(every sequence contains a convergent subsequence). But  the unit
ball is not precompact in infinite dimension. One can construct a
sequence in it from which no subsequence is  convergent \cite{DebnM05}.

The problem of unboundedness of the inverse of  Hilbert-Schmidt
operators is at the root of  ill-posed inverse problems, since
unboundedness implies non continuity. Thus, two arbitrary close
measurements in the output space may be created by two inputs
separated by a large distance in the input space. This is at the
% TODO - section reference below does not exist
heart of regularization theory.  In \S\ref{ch:appemb}, inverting covariance
operators  is needed in almost all developments in estimation or
detection. In theoretical development, we will assume the covariance
operators are invertible, at least on their range (see above).
However though, the inverse may be unbounded, and  
regularizing the inversion is needed.  Practically,  Tikhonov approach will systematically
be called for   regularization \cite{TikhA77}.
Typically, the inverse of $\Sigma$ will be replaced by $(\Sigma +
\lambda I)^{-1}$, where $I$ is the identity operator, or even by
$(\Sigma + \lambda I)^{-2} \Sigma$. Parameter $\lambda$ is the
regularization parameter, which can be chosen in some problems by
either inforcing  solutions (which depend on the inverse) to satisfy
some constraints, or by studying convergence of solutions as the
number of data grows to infinity.

% ----------------------------------------------------
%                       CHAPTER
% ----------------------------------------------------

\chapter{Applications of Embeddings}
\label{ch:appemb}

In this chapter, the use of embeddings in RKHS for signal processing applications is illustrated. 
 Many applications exist in signal processing and machine learning. We made here arbitrary choices 
that cover topics in signal processing mainly. The aim is not only to cover some topics, but also to provide some practical 
developments which will allow the reader to implement some algorithms. 

Since practical developments deal with finite amount of data,
only finite dimensional subspaces of possibly infinite dimensional RKHS are used empirically. This is explained  in a first paragraph and used to develop matrix representations for the  mean elements and the    operators presented in the previous chapter (such as covariance operators.)
The following paragraphs illustrate applications of embeddings. 
We first discuss embeddings for application of statistical testing in signal processing:
Comments on the use of the so-called deflection criterion for signal detection are made;  The design of independence and conditional independence measures are then presented. Filtering is discussed next. A first approach elaborate on the embedding of Bayes rule into a RKHS, while a second approach directly deals with the embeddings of the realisations of random signals, and how they are used for optimal filtering.

\section{Matrix representations of mean elements and covariance operators}
\label{matrixrep:ssec}

Practically, mean elements and covariance operators are used by applying them to functions in the RKHS.  Furthermore, the operators estimated are usually of finite rank. In inference problems, 
and when dealing with a finite number of observations $\{x_i\}_{i=1,\ldots,N}$ the representer theorem \cite{KimeW70,SchoS02}  states that the optimizers of some empirical risk function are to be searched for in the subspace $W_x$ of $ V_x$ generated by $\{K_x(.,x_i)\}_{i=1,\ldots,N}$ (from now on, a kernel is indexed to stress different variables and different kernels and spaces jointly.)  This is a direct consequence of the reproducing property. Any empirical risk to be minimized with respect to functions $f\in V_x$ is evaluated at the points $x_i$. Let
$f(.)=f_{\|}(.) +f_\perp(.)$, where $ f_{||}(.)\in W_x$ and $f_\perp(.) \in W_x^\perp$. Then $f(x_i) =\big< f , K_x(.,x_i) \big>=f_{\|}(x_i)$ and $f_\perp(x_i)=0$. The useful subspace of $ V_x$ is $W_X$, and
 practically, the empirical mean element or  operators are applied to functions in the form
\begin{eqnarray}
f(.) = \sum_{i=1}^N \alpha_i K_x(., x_i)
\end{eqnarray}

Consider first the action of the mean element  $m=\big< f , \widehat{m}^N_X  \big>$. The reproducing property leads to 
\begin{eqnarray}
m&=&\big< \sum_{i=1}^N \alpha_i K_x(., x_i) ,  \frac{1}{N} \sum_{i=1}^{N}K_x(., x_i)  \big>  \nonumber \\ 
%&=&  \frac{1}{N} \sum_{i,j=1}^{N}  \alpha_i \big<K_x(., x_i)  , K_x(., x_j) \big>  \\
&=&\frac{1}{N} \sum_{i,j=1}^{N}  \alpha_i  K_x(x_i, x_j)  \quad = \quad \vones_N^\top \vK_x \valpha 
\end{eqnarray}
where the Gram matrix $\vK_x$ has entries $\vK_{x,ij} = K_x(x_i, x_j) $, 
and where we introduced the vectors $\valpha=(\alpha_1,\ldots,\alpha_N)^\top$  and $\vones_N = \vones/N= (1/N,\ldots,1/N)^\top $. %to finally write
%\begin{eqnarray}
%m=\vones_N^\top \vK_x \valpha = \valpha ^\top \vK_x \vones_N 
%\end{eqnarray}
This is precisely the inner product of two elements in the finite dimensional RKHS with kernel $\vK_x$ and internal representations
$\vones_N$ and $\valpha$.
Furthermore, this formula leads to $\big< K(.,x_i) ,  \widehat{m}^N_X \big> =\delta_i^\top \vK_x \vones_N$
where $\delta_i$ is a vector of zeros except a 1 at the $i$th position. 

Hence, the mean can be calculated without embedding explicitly the data into the RKHS, but just by using the kernel evaluated at the data points.

To get the analytic form for the application of   the covariance operator to functions of $V_X$,    consider first the interpretation of the covariance as the mean element in the tensor product $V_x \otimes V_y$, and evaluate
\begin{eqnarray}
\lefteqn{\big< f(.) ,  \widehat{m}^N_{XY}(.,v)\big>_{V_x}}  \nonumber \\ 
&=&\big< \sum_{i=1}^N \alpha_i K_x(., x_i) ,  \frac{1}{N} \sum_{i=1}^{N}(K_x(., x_i)-\widehat{m}_X) (K_y(v, y_i)-\widehat{m}_Y(v)) \big>_{V_x} \nonumber \\
 %&=& \frac{1}{N} \sum_{i,j=1}^{N} \alpha_i \big< K_x(., x_i) , (K_x(., x_j)-\widehat{m}_X)   \big>_{V_x}(K_y(v, y_j)-\widehat{m}_Y(v))\\
 %&=&  \frac{1}{N} \sum_{i,j=1}^{N} \alpha_i \Big(\big< K_x(., x_i) , (K_x(., x_j) \big> -\big< K_x(., x_i) , \widehat{m}_X)   \big>\Big)(K_y(v, y_j)-\widehat{m}_Y(v))\\
  % &=&  \frac{1}{N} \sum_{i,j=1}^{N} \alpha_i (K_{x,ij} - \delta_i^\top \vK_x \vones_N)(K_y(v, y_j)-\widehat{m}_Y(v))\\
  &=&   \frac{1}{N} \sum_{j=1}^{N}  \big( (\valpha^\top\vK_x)_j -  \valpha^\top\vK_x \vones_N \big)(K_y(v, y_j)-\widehat{m}_Y(v))
\end{eqnarray}
Applying his result to a function $g\in V_y$ allows to obtain
\begin{eqnarray}
\lefteqn{\big< g ,  \widehat{ \Sigma}_{YX} f \big> =  \cov[g(Y) , f(X)] } \nonumber  \\
&=& \frac{1}{N} \sum_{j=1}^{N}  \big( (\valpha^\top\vK_x)_j -  \valpha^\top\vK_x \vones_N \big) \big<  \sum_{i=1}^N \beta_i K_y(v, y_i)  ,  K_y(v, y_j)-\widehat{m}_Y(v)\big>   \nonumber \\
%&=& \frac{1}{N} \sum_{j=1}^{N}  \big( (\valpha^\top\vK_x)_j -  \valpha^\top\vK_x \vones_N \big)
 %\sum_{i=1}^{N} \beta_i\big( K_{y,ij} - \delta_i^\top \vK_y \vones_N \big)\\
%&=& \frac{1}{N} \sum_{j=1}^{N}  \big( (\valpha^\top\vK_x)_j -  \valpha^\top\vK_x \vones_N \big)
% \big((\vbeta^\top \vK_{y})_j - \vbeta^\top \vK_y \vones_N \big)\\
&=&\frac{1}{N} \valpha^\top\vK_x (\vI - \frac{1}{N} \vones \vones^\top ) \vK_y \vbeta
\end{eqnarray}
The matrix $\vC_N=\vI - \frac{1}{N}\vones \vones^\top$ is the so-called centering matrix. If the mean elements  
$\widehat{m}_X$ and $\widehat{m}_Y$ are known to be identically equal to zero,  then the centering matrix does not appear in the calculations.  Note that $\vC_N$ is idempotent, $\vC_N^2=\vC_N$.

To directly study how the empirical covariance operator acts,   the alternative definition of the covariance operator as a linear operator is called for. Let $f(.)= \sum_{i=1}^N \alpha_i K_x(., x_i)$ in $ V_x$ and 
$g(.)= \sum_{i=1}^N \beta_i K_y(., y_i)$ in $ V_y$ such that $g=\widehat{\Sigma}_{YX} f$.

Assuming $\widehat{m}_X=0$ and $\widehat{m}_Y=0$
for the sake of simplicity, the empirical covariance operator can  be written as 
\begin{eqnarray}
\widehat{\Sigma}_{YX}f =\frac{1}{N}\sum_i \big< f(.) ,  K_x(.,x_i)\big>K_y(v, y_i)
\end{eqnarray}
Then  evaluating $g=\widehat{\Sigma}_{YX} f$ at the data points gives
\begin{eqnarray}
g(y_k) &=&\sum_j \beta_j  K_y(y_k,y_j)  
      \quad = \quad  (\vK_y \vbeta)_k   \nonumber \\
              &=&  \frac{1}{N}\sum_i \big< f(.) ,  K_x(.,x_i)\big>K_y(y_k, y_i) \nonumber  \\
    %          &=& \frac{1}{N}\sum_i f(x_i) K_y(y_k, y_i) \\
              &=& \frac{1}{N}\sum_{i,j} \alpha_j K_x(x_i,x_j) K_y(y_k, y_i) \!\quad =\! \quad \frac{1}{N} (\vK_y \vK_x \valpha)_k
\end{eqnarray}
Thus $\vbeta=N^{-1} \vK_x \valpha$ and $\big< g ,  \widehat{ \Sigma}_{YX} f \big>  = N^{-1}\vbeta^\top \vK_y \vK_x \valpha$ is recovered.

Finally, the application of the inverse of a covariance operator is important to study.
As discussed earlier,  the effect of the regularised version of the operator 
 $\widehat{ \Sigma}_{r,XX}= \widehat{ \Sigma}_{XX}+ \lambda \vI$ is studied, $\vI$ being the identity (it also denotes the identity matrix in finite dimension.)
Let $f=\widehat{\Sigma}_{r,XX}^{-1}g$, or $g=\widehat{\Sigma}_{r,XX}f$ where $f$ and $g$ are in the RKHS.
Using the decomposition  $f=f_{\|}+f_\perp$  recalled earlier,
 $f_{\|}(.)=\sum_i \alpha_i K_x(.,x_i)$,  $ f(x_k)=f_{\|}(x_k)$ and $f_\perp(x_k)=0$. Thus
\begin{eqnarray}
g(.) &=& \frac{1}{N}\sum_i f(x_i) K_x(., x_i) + \lambda f(.) \nonumber  \\
     &=&   \frac{1}{N}\sum_{i,j} (\vK \valpha)_i  K_x(., x_i) + \lambda\sum_i \alpha_i K_x(.,x_i)   + \lambda f_\perp(. ) 
\end{eqnarray}
Then, evaluating $g(.) $ at all $x_i$, 
\begin{eqnarray}
\vK_x \vbeta = \frac{1}{N} (\vK_x +N\lambda I) \vK_x \valpha
\end{eqnarray}
or solving, the action of the regularized inverse is obtained as
\begin{eqnarray}
\valpha  = N  (\vK_x +N\lambda I) ^{-1} \vbeta
\end{eqnarray}

\section{Signal detection and the deflection criterion}

Signal detection is usually modeled as a binary testing problem: Based on the observation of a signal, a detector has to decide which hypothesis among $H_0$ or $H_1$ is true. The detector is in general a functional of the observation,  denoted here as a filter. 
In \cite{PiciD88,PiciD90}, Picinbono\&Duvaut developed the theory of Volterra filters for signal detection. 
Recall that Volterra filters are polynomial filters. If $x(n)$ is an input signal, the output of a $M$th order Volterra filter reads
\begin{eqnarray}
y(n)= h_0+ \sum_{i=1}^M \sum_{j_1,\ldots, j_i}h_i(j_1,\ldots, j_i) x(n-j_1) \times \ldots \times x(n-j_i)
\end{eqnarray}
where the functions $h_i$ satisfy some hypothesis ensuring the existence of the output $y(n)$. The first term $i=1$ is nothing but a linear filter; the term $i=2$ is called a quadratic filter, and so on.  If the range of summation of the $j_i$'s is finite for all $i$, the filter is said to be of finite memory.

In signal detection theory, if the detection problem is set up as a binary hypothesis testing problem, different approaches exist to design an efficient detector.
For example, in the Neyman-Pearson approach, the detector that maximizes the probability of detection subject to  a maximal given probability of false alarm is seeked for.  Recall for instance that the optimal detector  in this approach
of a known signal in Gaussian noise is the matched filter, which is a linear filter, and thus a first order Volterra filter.

\noindent {\bf Deflection.} A simpler approach relies on the so-called deflection criterion. This criterion does not require the full modeling of the probability laws
under each hypothesis, as is the case for Neyman-Pearson approach.
The deflection is a measure that quantifies a contrast (or distance) between the two hypotheses for a particular detector. The greater the deflection the easier the detection, because the greater the separation between the two hypotheses. 
 Let an observation
$x$ be either distributed according to $P_0$ under hypothesis $H_0$ or according to $P_1$ under hypothesis 
$H_1$.  Let $y(x)$ be a test designed to decide whether $H_0$ or $H_1$ is true. The deflection is defined
as 
\begin{eqnarray}
d(y) = \frac{(E_1[ y ] - E_0[y])^2}{\var_0[y]}
\end{eqnarray}
where the subscript $0$ or $1$ corresponds to the distribution under which averages are evaluated. As mentionned above, the deflection quantifies the ability of the test to separate the two hypotheses.

In general the structure of the detector is imposed, and the best constrained structure which maximises the deflection is seeked for.
If a linear structure is  chosen, such as  $y(x)=h^\top x$ in finite dimension,  the matched filter is recovered. Picinbono\&Duvaut studied the optimal detector according to the deflection when it is constrained to be a Volterra filter of the observation. They particularly develop the geometry of the filter, recoursing to the Hilbert space underlying Volterra filters of finite orders. An interesting fact outlooked in \cite{PiciD88,PiciD90} is that this Hilbert space is the finite dimensional reproducing kernel Hilbert space generated by the kernel $K(x,y)= (1+ x^\top y)^M$.  This is easily seen using a very simple example for $M=2$. Consider the map defined by 
\begin{eqnarray}
\Phi&:& \R^2 \longrightarrow  V \nonumber \\
& &  \vx \longmapsto \big(1 ,\sqrt{2} x_1, \sqrt{2} x_2,     x_1^2,  \sqrt{2} x_1   x_2 , x_2^2 \big)^\top
\end{eqnarray}
which embeds $\R^2$ into a subspace $  V  $ of $\R^6$. $ V$ is a reproducing kernel Hilbert space whose kernel is $K(x,y)= (1+ \vx^\top \vy)^2$.
Indeed, a direct calculation shows that $\Phi(\vx)^\top \Phi(\vy) = (1 + \vx^\top \vy)^2$.  Now, consider
$y(n)= \vH^\top \Phi\big( (x(n), x(n-1) \big)$ where $\vH \in \R^6$. $ y(n)$ is then the output of a linear-quadratic  Volterra filter,
with $h_0=H_1, h_1(1)=H_2/\sqrt{2},  h_1(2)=H_3/\sqrt{2}, h_2(1,1)=H_4,h_2(1,2)=H_5/\sqrt{2}, h_2(2,2)=H_6$, and all other parameters set to zeros. This simple example generalizes to any $M$ and any finite memory.

The aim here is to  analyse the deflection for filters living in arbitrary RKHS. 

\noindent {\bf Detecting in a RKHS.} Consider $ V$ to be the RKHS associated with kernel $K$. We assume the kernel $K$ to be characteristic
 ({\it i.e.} the mapping sending a probability to the mean element in $ V$ is injective.)
Let $\mu_i$ be the embedding of $P_i$, or $\mu_i(.)=E_i[K(., x)]$. 
Let $\Sigma_0:  V \rightarrow  V$ be the covariance operator of a random variable distributed under $P_0$. 
The detector is seeked for as a function $f$ in the RKHS that maximizes the deflection. Thus $y=f(x)=\big< f(.) , K(.,x) \big>$. Then 
$E_i[y] =E_i[f(x)]= \big< \mu_i ,  f \big>$ (definition of the mean element.) Furthermore, the definition of the 
covariance operator gives $\var_0[y]=\big< f , \Sigma_0f \big>$. The deflection for $y$ thus reads
\begin{eqnarray}
d(y) = \frac{\big< \mu_1-\mu_0 ,  f\big>^2}{\big< f , \Sigma_0f \big>}
\end{eqnarray}
For ease of discussion,  the covariance operator is  assumed invertible.
Since $\Sigma_0$ is positive definite, so is its inverse, and a new inner product in $ V$ can be defined  by
$\big< f ,  g\big>_0 = \big< f ,  \Sigma_0^{-1}g\big>$. The deflection of $y$ then writes
\begin{eqnarray}
d(y) = \frac{\big< \mu_1-\mu_0 , \Sigma_0 f\big>^2_0}{\big< f , \Sigma_0f \big>}
\end{eqnarray}
and Schwartz inequality offers an easy means to maximize the deflection. Indeed, the following holds
\begin{eqnarray}
d(y) &\leq& \frac{\big< \mu_1-\mu_0 , \mu_1-\mu_0\big>_0 \big< \Sigma_0 f , \Sigma_0 f\big>_0}{\big< f , \Sigma_0f \big>}
\nonumber \\
%&=&\big< \mu_1-\mu_0 , \mu_1-\mu_0\big>_0\\
&=&\big< \mu_1-\mu_0 , \Sigma_0^{-1}\big( \mu_1-\mu_0\big)\big>
\end{eqnarray}
and the maximum is attained when vectors $\mu_1-\mu_0$ and $\Sigma_0 f$ are proportional.
The constant of proportionality is not important, since this amounts to scale function $f$ and does not change the optimal 
deflection. 
The constant of proportionality is thus chosen to be 1. The optimal function hence satisfies $\mu_1-\mu_0=\Sigma_0 f$, or $f=\Sigma_0^{-1}\big( \mu_1-\mu_0\big)$.

The assumption of the invertibility of the covariance operator is not fundamental in the derivation. If it is not invertible,  the 
derivative of the deflection (in some functional sense such as the Gateaux derivative) is used to show that the maximum is obtained when $\mu_1-\mu_0$ and $\Sigma_0 f$ are proportional.

If no structure is imposed to the filter, and if  a  Neyman-Pearson approach is taken to solve the detection problem, the 
best strategy is to compare the likelihood ratio to a threshold, chosen in order to satisfy a constraint on the error of the first kind.
 A link between the optimal detector in the deflection sense and the likelihood ratio can be established. 
Precisely let $r(x)=l(x)-1=p_1(x)/p_0(x) -1$. In the following,  $r(x)$ is assumed square integrable under hypothesis $H_0$.

First note the following. Any $f$ for which $f(x)\in L^2$ satisfies $E_0[f(x)r(x)]=E_1[f(x)]-E_0[f(x)]$. Thus any $f\in  V$ such that 
$f(x)\in L^2$ satisfies $E_0[f(x)r(x)] = \big< f , \mu_1-\mu_0 \big>$. In particular, $E_0[K(u,x)r(x)] = \big< K(u,.) , \mu_1-\mu_0 \big> =\big( \mu_1-\mu_0\big)(u)$.
If we seek for the best estimator of the likelihood ratio in $ V$ minimizing the mean square error  under $H_0$, we must center it and then estimate $r(x)$, since $E_0[p_1(x)/p_0(x)]=1$.  Then the optimal estimator will lead to an error orthogonal (in the $L^2$ sense) to $ V$, and thus the following equation has to be solved
\begin{eqnarray}
E_0[g(x) (r(x) -\widehat{r}(x))]=0, \quad \forall g \in  V
\end{eqnarray}
But  using the reproducing property this last expectation can be written as
\begin{eqnarray}
E_0[g(x) (r(x) -\widehat{r}(x))] &=&E_0\Big[ \big< g ,  K(.,x)\big> \big( r(x) -\widehat{r}(x)  \big)\Big] \\
&=& \Big< g \Big|E_0\big[K(.,x)r(x) \Big] - E_0\big[K(.,x)\big< \widehat{r} ,  K(.,x)\big> \big] \Big> \\
&=& \big< g ,  \mu_1-\mu_0- \Sigma_0 \widehat{r} \big> 
\end{eqnarray}
Since the last result is equal to 0 for any $g\in  V$, the solution necessarily satisfies the equation $\mu_1-\mu_0- \Sigma_0 \widehat{r} =0$, hence proving that $\widehat{r} $ is the detector that maximizes the deflection.

The optimal element in the RKHS which maximises the deflection is therefore also the closest in the mean square sense  to the likelihood ratio. 
%Let us finally mention that the analysis presented above is very close to what is called kernel Fisher discriminant  analysis in the %machine learning literature. 

\section{Testing for independence}
\label{testindep:ssec} 

Recent works in machine learning and/or signal processing use RKHS mainly for the possibility offered to unfold complicated data in a larger space. 
In classification for example, data nonlinearly separable in the physical space may become linearly separable in a RKHS. Testing independence by embedding data into a RKHS relies in some way on the same idea.

The argument is to use covariance (in the RKHS) to assess independence. It is well known that no correlation between two variables does not imply independence between these variables.
However, an intuitive idea is that no correlation between any nonlinear transformation of two variables may reveal independence. This simple idea was at the root of the celebrated Jutten-Herault algorithm, the first device to
perform blind source separation. In fact it was studied as early as 1959 by R{\'e}nyi \cite{Reny59}. He showed $X$ and $Y$, two variables defined on some common probability space, are independent if and only if 
the so-called maximal correlation $\sup_{f,g} \cov[f(X),g(Y)]$ is equal to zero, where $f,g$ are continuous bounded functions. 

Recently, this result was revisited   by Bach, Gretton  and co-workers through the lense of RKHS \cite{BachJ02,GretHSBS05}. The maximal correlation as used by R{\'e}nyi is too complicated to be practically evaluated, because the space over which the supremum has to be calculated is far too big. The idea is then to look for the supremum in a space in which, firstly the maximum correlation can  be more easily calculated, secondly, in which  R{\'e}nyi's result remains valid. 

Gretton showed in \cite{GretHSBS05} that  $X$ and $Y$ are independent if and only if 
$\sup_{f,g} \cov[f(X),g(Y)]$ is equal to zero, where $f,g$ are living in a RKHS (precisely its unit ball) generated by a universal kernel. Recall that universality is understood here as the denseness of the RKHS into the space of bounded continuous functions. The link between Gretton's result and R{\'e}nyi's is then intuitive, since under universality, any continuous bounded function may be approximated as closely as needed by a function in the RKHS.  
The second requirement above is thus satisfied. The first one is also verified, and this is the magical part of Gretton's approach. The maximal correlation evaluated in the unit ball of the universal RKHS
\begin{eqnarray}
\sup_{f\in {\cal U}_x,g \in {\cal U}_y}  \cov  \big[f(X),g(Y) \big]
\end{eqnarray}
where ${\cal U}=\big\{ f\in  V \backslash \|f\|=1\big\}$, is nothing but the operator norm of the covariance operator $\Sigma_{XY}$. Indeed, by definition, 
\begin{eqnarray}
\| \Sigma_{XY}\| &=& \sup_{g \in {\cal U}_y}  \| \Sigma_{XY} g \|_{ V_x} \quad =\quad \sup_{f\in {\cal U}_x,g \in {\cal U}_y}  \big|\left< f  ,   \Sigma_{XY} g \right>\big| \nonumber  \\
&=& \sup_{f\in {\cal U}_x,g \in {\cal U}_y}  \big|\cov  \big[f(X),g(Y) \big]\big|
\end{eqnarray}
As shown by Bach \cite{BachJ02} and Gretton \cite{GretHSBS05} this quantity can be efficiently evaluated from a finite amount of data. Let $(x_i,y_i)_{i=1,\ldots, N}$  be $N$ independent and identically distributed copies of $X$ and $Y$, then  $\left<  f \left|\right. \Sigma_{XY} g  \right>$ approximated by $\left<  f \left|\right. \widehat{\Sigma}^N_{XY} g  \right>$ is given by 
\begin{eqnarray}
\left<  f \left|\right. \widehat{\Sigma}^N_{XY} g  \right> = \frac{1}{N} \valpha^\top \vK_x \vC_N \vK_y \vbeta
\end{eqnarray}
where the $\alpha$s (respectively $\beta$s) are the coefficients of the expansion of $f$  (respectively $g$) in the subspace of $ V_x$  (respectively $ V_y$) spanned by $K_x(.,x_i), i=1,\ldots ,N $ (respectively $K_y(.,y_i), i=1,\ldots ,N $). The norm of $f$ is given by $\valpha^\top \vK_x \valpha$ and the norm of $g$ by $\vbeta^\top \vK_y \vbeta$.
Therefore the maximal correlation is approximated by  
\begin{eqnarray}
\sup_{\valpha^\top \vK_x \valpha=1,\vbeta^\top \vK_x \vbeta=1 }  \frac{1}{N} \valpha^\top \vK_x \vC_N \vK_y \vbeta=  \frac{1}{N} \big\|\vK_x^{1/2} \vC_N \vK_y^{1/2} \big\|_2
\end{eqnarray}
where $\| A\|_2=\sqrt{\lambda_M(A^\top A )}$ is the usual spectral norm of matrix $A$, that is the square root of its maximal singular value.

The Gram matrices appearing in the estimation of the maximal correlation are of size $N\times N$. Therefore, the calculation of the maximal correlation 
can be very time consuming for large data sets. A simple idea allows to end up with a measure easier to evaluate. It relies on Hilbert-Schmidt norms.
It is known that for any Hilbert-Schmidt operator, the operator norm is lower or equal than the Hilbert-Schmidt norm. Thus Gretton's result remains
true if the operator norm is replaced with he Hilbert-Schmidt norm: $X$ and $Y$ are independent if and only if $\| \Sigma_{XY} \|_{HS}=0$. 

Furthermore, when dealing with the data set $(x_i,y_i), i=1,\ldots ,N$,  the Hilbert-Schmidt norm of the covariance operator may be approximated by the 
 the Hilbert-Schmidt norm of the empirical estimates of the operator. It is then not difficult to show that 
 \begin{eqnarray}
\|  \widehat{\Sigma}^N_{XY} \|^2_{HS} = \frac{1}{N^2} \trace \big( \vC_N \vK_x \vC_N \vK_y )
\end{eqnarray}
For independent and identically distributed data, this estimator satisfies a central limit theorem. It is asymptotically unbiased, and its variance can be explicitly written.
An unbiased version also exists  which in the spirit of $k$-statistics  eliminates the $1/N$ bias \cite{SongSGBB12}. Both versions satisfy some concentration inequalities and a central limit theorem (at the usual rate) \cite{GretHSBS05,SongSGBB12}.

This measure has been called HSIC for Hilbert-Schmidt Independence Criterion. It can obviously be used for testing independence between two samples,
but has also been used for feature selection \cite{SongSGBB12}. Its evaluation requires $O(N^2)$ operations. This complexity can be lowered by using approximation 
to Gram matrices such as the incomplete Cholesky factorization \cite{FineS01} or other types of approximation \cite{SchoS02}. As an alternative, we have designed a recursive implementation of
HSIC which is exact, of course still requiring $O(N^2)$ operations, but which only manipulates vectors. This algorithm can be used on very long data sets in reasonable time. However, for huge data sets
it is still impractical. But its recursive structure allows an efficient approximation leading to the possibility of calculating HSIC on-line or whatever the size of the data set. 
We just give the forms of the second algorithm, the derivation being developed in \cite{Ambl13,AmblM13}. In this procedure, a dictionary is
built recursively, the elements of which are used to evaluate HSIC on-line. New data are included in the dictionary if they are  sufficiently incoherent with the members of the dictionary. This procedure was proposed in a filtering context in \cite{RichBH09} as a simplification of the approximate linear dependence (ALD) criterion proposed in \cite{EngeMM04}. Coherence is measured in the tensor product $ V_x \otimes  V_y$. Thanks to the reproducing property,
the coherence between two data $(x_\alpha,y_\alpha)$ and $(x_n,y_n)$ is evaluated  as $| K_x(x_n,x_\alpha)K_y(y_n,y_\alpha)| $ (assuming the kernels are normalized).
The dictionary ${\cal D}^\mu_n$ contains the index of the data retained up to time $n$, initializing it with ${\cal D}^\mu_1=\{1\}$. It is updated according to 
\begin{eqnarray}
{\cal D}^\mu_n = \left\{
\begin{array}{ll}
{\cal D}^\mu_{n-1} \cup \{n\}  & \mbox{ if }  \sup_{\alpha \in {\cal D}^\mu_{n-1}  } \big| K_x(x_n,x_\alpha)K_y(y_n,y_\alpha) \big| \leq \mu \\
{\cal D}^\mu_{n-1} & \mbox{ otherwise}
\end{array}
\right.
\label{coherence:eq}
\end{eqnarray}
Parameter $\mu$ is in $(0,1]$. If $\mu=1$ all the new data are aggregated to the dictionary and the following algorithm exactly delivers HSIC. If $\mu<1$, the new data are added to the dictionary 
if it is sufficiently incoherent with all the members of the dictionary.  

To describe the algorithm, some more notations are needed.
Let  $\kappa^n_{x}$ be the norm $K_x(x_{n},x_{n})$. Let $\vpi_{n}$ be  a $|{\cal D}_{n}|$  dimensional vector  whose entries 
 $\pi_{n}(\alpha)$  count the number of times the element $\alpha$ of the dictionary has been  chosen by the rule (\ref{coherence:eq}).   $\vpi_{n}$ is  initialised by $\pi_1(1)=1$. Let  $ \vk_{x}^n$ contain the $K_x(x_n,x_\alpha), \forall \alpha\in {\cal D}_{n-1}^\mu$. Vectors $\vv^n$ appearing in the algorithm below  are initialised as $\vv^0=1$. Finally, $\circ $ denotes the Hadamard product, {\it i.e. } the entrywise product for vectors or matrices.   Equipped with all this, the algorithm is the following:

\noindent\framebox{ 
\parbox[t]{.98\linewidth} { 
\underline{Sparse HSIC~:}
\begin{eqnarray}
\widehat{H}_n^{\mu}\!\!\!\!\!  &=&\!\!\!\!\! \|M_{yx}^n \|^2 + \|    m_x^n\|^2\|    m_y^n\|^2 - 2c^n_{yx} \\
 \| M_{yx}^n \|^2  \!\!\!\!\! &=&  \!\!\!\!\!  \frac{(n-1)^2}{n^2} \| M_{yx}^{n-1} \|^2 +\frac{2}{n^2}\vpi^\top_{n-1} \vk_{x}^n\circ\vk_{y}^n +\frac{\kappa^n_{x}\kappa^n_{y}}{n^2}\\
  \|    m_x^n\|^2  \!\!\!\!\!  &=&  \!\!\!\!\!  \frac{(n-1)^2}{n^2} \|    m_x^{n-1}\|^2+ \frac{2}{n^2}\vpi^\top_{n-1} \vk_{x}^n
  +\frac{\kappa^n_{x}}{n^2}  \\
  c^n_{yx}  \!\!\!\!\!  &=& \! \!\!\!\!  \frac{1}{n^3} \vpi^\top_n \vv_x^n \circ \vv^n_y   \mbox{ where:}
  \end{eqnarray}
{\bf 1.} Si ${\cal D}_n = {\cal D}_{n-1} \cup \{ n\}  $
\begin{eqnarray}
\vv_{x}^{n} =
\left( \begin{array}{c}
\vv_{x }^{n-1} +\vk_{x}^n \\
\vpi^\top_{n-1} \vk_{x}^n+\kappa^n_{x}
\end{array}\right)  \quad  \mbox{ and }\quad 
\vpi_{n}  = \left( \begin{array}{c}  \vpi_{n-1} \\ 1 \end{array}\right) 
\end{eqnarray}
{\bf 2.} Si ${\cal D}_n = {\cal D}_{n-1} \quad :  \quad a= \arg \max_{\alpha \in {\cal D}_{n-1} }  \big| K_x(x_n,x_\alpha)K_y(y_n,y_\alpha) \big|$, 
\begin{eqnarray}
\vv_{x}^{n} = \vv_{x }^{n-1} +\vk_{x}^n  \quad  \mbox{ and }\quad 
\vpi_{n}  =   \vpi_{n-1} + \delta_{a \alpha}
\end{eqnarray}  }}

Sparse HSIC is illustrated on the following example, taken from \cite{FukuGSS07}. Consider the couple of independent variables $(X,Y)$ where 
$X$ is uniformly distributed on $[-a,a]$ and $Y$ is uniformly distibuted on $[-c,-b]\cup [b,c]$,   $a,b,c$ being real positive constants. We choose $a,b,c$ to ensure that $X$ and $Y$ have the same variance.
Let $Z_\theta$ the random vector obtained by rotating vector $(X,Y)$ by an angle $\theta$. The components of $Z_\theta$ are uncorrelated whatever $\theta$ but are independent if and only if $\theta= \pi/2 \times \Z$.  $N$ i.i.d. samples of  $Z_\theta$ are simulated.  For   $\theta$ varying in $[0,\pi/2]$   the final value of sparse HSIC $\widehat{H}_N^{\mu}$ is plotted   in figure (\ref{HSICtheta:fig}). The right plot displays $\widehat{H}_N^{\mu}$ for the five values $\mu=0.8, 0.85, 0.9, 0.95, 1$. The kernel used is the Gaussian kernel $\exp(-\|\vx-\vy\|^2)$. As can be seen, the distorsion due to sparsity is very small. Interestingly, the size of the dictionary (left plot) is very small as soon as $\mu<0.95$ in this example, thus implying a dramatic decrease in computation time and memory load. Moreover,  the form of the algorithm can be simply turned into an adaptive estimator by changing decreasing step-sizes into constant step sizes. This allows to track changes into the dependence structure
of two variables \cite{Ambl13}.
\begin{figure}[htbp]
\centering
\includegraphics{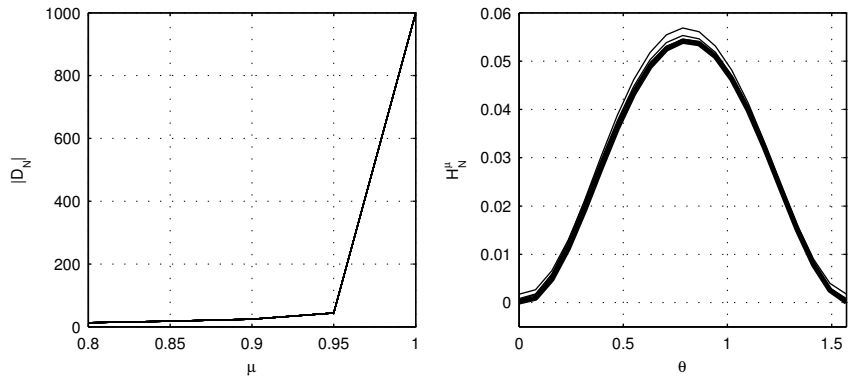}
\caption{Right plot:  HSIC as a function of $\theta$ in the example in the text. The thick line corresponds to $\mu=1$ or equivalently to a non sparse evaluation of HSIC. The other lines correspond to $\mu=0.8, 0.85, 0.9, 0.95$. Left plot:  Size  of the dictionary achieved after $N$ iterations as a function of $mu$. }
\label{HSICtheta:fig}
\end{figure}
The size of the dictionary is probably linked to the speed of decrease of the spectrum of the Gram matrix. In general, when the Gram matrix can be 
nicely approximated by a low rank matrix, we observed that the size of the dictionary obtained is small. A general study of the approximation of the Gram matrix by the Gram matrix issued from the coherence dictionary remains to be done. Note however that some information are given for the related ALD criterion  in \cite{SunGS12} (see also \cite{Bach13} for related materials on low rank approximations in a regression context).
We will meet again the coherence-based sparsification procedure in the application of RKHS to on-line  nonlinear filtering.

\noindent {\bf Maximum Mean Discrepancy (MMD).} Another approach to testing independence using embeddings into RKHS relies on 
the maximum mean discrepancy measures or MMD \cite{GretBRSS12}. MMD measures a disparity between two probability measures 
$P$ and $Q$ . Let  $X$ and $Y$ be two  random variables taking values in a space $\X$ and respectively distributed according to $P$ and $Q$.
Let ${\cal F}$ be a function space on ${\cal X}$. MMD is defined as
\begin{eqnarray}
MMD(P,Q; {\cal F}) =\sup_{f\in {\cal F}} \Big( E_P[f(X)] -E_Q[f(Y)] \Big)
\end{eqnarray}
If the function space is rich enough, a null MMD implies equality between $P$ and $Q$. As previously, it can be shown that if ${\cal F}$ is restricted to be the unit ball of a RKHS associated to a universal kernel, this result is true, $MMD(P,Q; {\cal F})=0 \Leftrightarrow P=Q$. 
This can be used to test for independence by measuring the maximum mean discrepancy between a joint probability and the product of its marginals. Furthermore, and like HSIC, there is a nice way of estimating MMD in a RKHS. In fact, thanks to Schwartz inequality, MMD can be expressed as
\begin{eqnarray}
MMD(P,Q;  V)^2 &=&\sup_{f\in  V, \|f\|\leq 1} \Big( E_P[f(X)] -E_Q[f(Y)] \Big)^2 \nonumber\\
&=&\sup_{f\in  V, \|f\|\leq 1} \Big( \big< \mu_P-\mu_Q ,   f \big> \Big)^2 \nonumber\\
&=& \big\|\mu_P-\mu_Q \big\|^2_ V
\end{eqnarray}
where $\mu_P$, $\mu_Q$ are the mean elements of respectivelyy $P$ and $Q$ in $ V$. When data $(x_i,y_i), i=1,\ldots, N$ are 
to be tested for independence, empirical estimators of MMD are very easy to develop and implement. Furthermore, asymptotic results for their efficiency exist that allows a complete development of independence testing \cite{GretBRSS12}.

\section{Conditional independence measures}

Remarkably, the measure of independence presented in the preceding section has an extension which allows to quantify conditional independence.
Conditional independence is a fundamental concept in different problems such as graphical modeling or dimension reduction. For example, graphical Markov properties   of Gaussian random vectors are revealed by conditional independence relations between these  components \cite{Laur96,Whit89}. Since measuring independence was found to be elegantly and efficiently done by embedding measures into RKHS, it was natural to work on the extension to conditional independence. It turns out that conditional independence can also be assessed in RKHS. \\

\noindent {\bf Some recalls \cite{Laur96,Whit89}.}
Let $X,Y,Z$ be three real  random vectors of arbitrary finite dimensions, and $\hat{X}(Z)$ and $\hat{Y}(Z)$ the best linear MMSE (minimum mean square error) estimates of $X$ and $Y$ based on $Z$. It is well-known that these are given by $\hat{X}(Z) =\Sigma_{XZ}\Sigma_{ZZ}^{-1}Z$ and $\hat{Y}(Z) =\Sigma_{YZ}\Sigma_{ZZ}^{-1}Z$,  where $\Sigma_{AB}:=\cov \left[A,B \right] $
stands for the covariance matrix of vectors $A$ and $B$.
The errors $X-\hat{X}(Z)$ and $Y-\hat{Y}(Z)$ are orthogonal to the linear subspace  generated by $Z$, and this can be used to show the well-known relations
\begin{eqnarray}
\Sigma_{XX|Z}&:=& \cov \left[X-\hat{X}(Z),X-\hat{X}(Z) \right]   \nonumber \\
&=& \Sigma_{XX} - \Sigma_{XZ} \Sigma_{ZZ}^{-1} \Sigma_{ZX} \\
\Sigma_{XY|Z}&:=& \cov \left[X-\hat{X}(Z),Y-\hat{Y}(Z) \right]  \nonumber\\
&=& \Sigma_{XY} - \Sigma_{XZ} \Sigma_{ZZ}^{-1} \Sigma_{ZY}
\end{eqnarray}
$\Sigma_{XX|Z}$ is the covariance of the error in estimating $X$ linearly from $Z$. It is also called the partial covariance and it is equal to the conditional covariance in the Gaussian case. The second term measures the correlation remaining between $X$ and $Y$ once the effect of their possibly common observed cause $Z$ has been linearly  removed from them. $\Sigma_{XY|Z}$ is called the partial cross-covariance matrix and is equal to the conditional cross-covariance in the Gaussian case ({\it i.e.} $X$, $Y$ and $Z$ are jointly Gaussian.) 

Therefore, in the Gaussian case, conditional independence can be assessed using linear prediction and the partial cross-covariance matrix. This has led to extensive development in the field of graphical modeling. \\

\noindent {\bf Using kernels.}
The approach above can be extended to assess conditional independence 
for nonGaussian variables by using embeddings in RKHS. 
The extension relies on the notion of conditional cross-covariance operators, a natural extension of the covariance operators. Having in mind that cross-covariance operators suffices to assess independence (as cross-covariance does in the finite dimensional Gaussian case), the idea is consider 
\begin{eqnarray}
\Sigma_{XY|Z} := \Sigma_{XY} - \Sigma_{XZ} \Sigma_{ZZ}^{-1} \Sigma_{ZY} 
\end{eqnarray}
as a potential candidate to assess conditional independence the operator.
The first remark concerns the existence of this operator. 

$\Sigma_{ZZ}$ is an operator from $V_z$ to  $V_z$. Let ${\cal N}( \Sigma_{ZZ})$ and ${\cal R}(\Sigma_{ZZ})$ be respectively its null space and its range. The operator is supposed to be  invertible on its range and the inverse is abusively denoted as $\Sigma_{ZZ}^{-1}$.  The inverse exits in full generality if and only if 
${\cal N}( \Sigma_{ZZ})=\{0\}$ and ${\cal R}(\Sigma_{ZZ})=V_z$, corresponding to injectivity and surjectivity. Thus in the sequel, when dealing with ensemble operators, covariance operator will be supposed invertible. 
To avoid working with inverses,  normalized covariance operators $V_{XY}$ should be considered. They are defined using $\Sigma_{XY}= \Sigma_{XX}^{1/2} V_{XY} \Sigma_{YY}^{1/2}$ \cite{Bake73}.
The conditional covariance operator then reads
\begin{eqnarray}
\Sigma_{XY|Z} := \Sigma_{XY} - \Sigma_{XX}^{1/2} V_{XZ} V_{ZY}\Sigma_{YY}^{1/2}
\end{eqnarray}
and the normalized version is given by 
\begin{eqnarray}
V_{XY|Z} := V_{XY} -  V_{XZ} V_{ZY}
\end{eqnarray}
This last definition is the only theoretically well grounded, since the $V$ operators are shown to exist in \cite{Bake73} under some assumptions, but without relying  on the invertibility of the covariance operators. However for practical purposes, we will use the other form, knowing that the existence of the inverse is subject to caution. 

Several theorems show the meaning of the conditional covariance operators, and how we can assess conditional independence with them. They are all mainly due to  Fukumizu,  Bach and  Jordan   \cite{FukuBJ04,FukuBJ09}. 
The first result links conditional expectation to covariance and cross-covariance operators. 
\begin{theorem}
\label{indepcondeqm:th}
For all $g \in V_y$,
\begin{eqnarray}
\big< g ,  \Sigma_{YY|X} g \big> = \inf_{f\in V_x} E
\Big[  \big( (g(Y)-E[g(Y)])  -  (f(X)-E[f(X)]\big)^2\Big]
\end{eqnarray}
If furthermore the direct sum $V_x+\R$ is dense in $L^2(P_X)$, then 
\begin{eqnarray}
\big< g ,  \Sigma_{YY|X} g \big> = E_X\left[\var [ g(Y) \big| X]    \right]
\end{eqnarray}
\end{theorem}
The density assumption means than any  random variable of $L^2(P_X)$   can be approximated as closely as desired by a function of $V_x$ plus a real. Adding the real is necessary since very often, constants do not belong to the RKHS under study (Remind that  $L^2(P_X)$ is  the space of square integrable functions with respect to $P_X$, or otherwise stated, the space of functions of  $X$ with finite expected squared norm $E[\|f(X)\|^2]<+\infty$.) The result of the theorem is an extension of what was recalled above, but stated in RKHS. The operator $\Sigma_{YY|X}$ measures the power of the error made in approximating a function of a random variable embedded in a RKHS by a function of another random variable embedded in its own RKHS. The second result generalizes the Gaussian case since under the assumption of density the operator evaluates a conditional variance. An informal proof is given, needing hypothesis not present in the statement. A full proof may be found in \cite[Prop. 2 and 3]{FukuBJ09}.
\begin{proof}
Let ${\cal E}_g(f) = E\Big[  \big( (g(Y)-E[g(Y)])  -  (f(X)-E[f(X)]\big)^2\Big]$. Then $f_0$ provides the infimum if 
${\cal E}_g(f_0 +f )-{\cal E}_g(f_0 ) \geq 0$ for all $f\in V_x$. But we have
\begin{eqnarray}
{\cal E}_g(f_0 +f )-{\cal E}_g(f_0 ) &=& \big< \Sigma_{XX}f , f \big>+2\big< \Sigma_{XX}f_0 -\Sigma_{XY} g ,  f\big>
\end{eqnarray}
Obviously, $\Sigma_{XX}f_0 -\Sigma_{XY} g =0$ satisfies the condition. It is also necessary. Indeed, suppose 
$\Sigma_{XX}f_0 -\Sigma_{XY} g \not = 0$. $\Sigma_{XX}$ is auto-ajoint and thus only has positive or null eigen values. Thus $\Sigma_{XX}f=-f$ has no solution and the null space of $\Sigma_{XX}+I$ is reduced to 0. Thus $\Sigma_{XX}+I$ is invertible. Therefore there is a non zero $f$ such that $\Sigma_{XX}f+f=-2(\Sigma_{XX}f_0 -\Sigma_{XY} g)$, and this $f$ satisfies
${\cal E}_g(f_0 +f )-{\cal E}_g(f_0 ) =-\big< f , f \big><0$, giving a contradiction. Thus, this gives the result. Note we use the fact that $\Sigma_{XX}$ is invertible, at least on its range. The fact that $\big< g ,  \Sigma_{YY|X} g \big> = E_X\left[\var [ g(Y) \big| X]    \right]$ is shown hereafter as a particular case of conditional crosscovariance  operator.
\end{proof}

Since the conditional operator is linked to optimal estimation (in the mean square sense) of a function $g(Y)$ from a transformation of $X$, 
 $\Sigma_{XX} E[ g(. )  \big| X] = \Sigma_{XY} g(.)$ should be a solution. However, this requires that the conditional expectation $ E[ g(. )  \big| X]  $ lies in $ V_x$, a fact that is absolutely not guaranteed. 
If it is supposed so, the statement and results are more direct. 
In that case,  for any $g\in V_y$,
\begin{eqnarray}
\Sigma_{XX} E[ g(. )\big| X] = \Sigma_{XY} g(.)
\end{eqnarray}
 this provides a means of calculating the conditional mean in a RKHS if the covariance is invertible. 

The following set of relations highlights the effect of the conditional covariance operators on function.
\begin{eqnarray}
\big< f , \Sigma_{XY|Z}g \big>&=& \cov[ f(X),g(Y)] -\big< f ,  \Sigma_{XZ} \Sigma_{ZZ}^{-1} \Sigma_{ZY} g\big> \nonumber\\
&=& \cov[ f(X),g(Y)] -\big< \Sigma_{ZX} f ,    \Sigma_{ZZ}^{-1} \Sigma_{ZY} g\big> \nonumber\\
&=& \cov[ f(X),g(Y)] -\big< \Sigma_{ZX} f ,    E[g(.) | Z] \big> \nonumber\\
&=&\cov[ f(X),g(Y)] - \big< \Sigma_{ZZ} E[f(.) | Z] ,    E[g(.) | Z] \big> \nonumber\\
&=&\cov[ f(X),g(Y)] - \cov_Z\big[ E[f(.) | Z] ,   E[g(.) | Z] \big] \nonumber\\
&=& E[ f(X) g(Y)]  - E_Z\big[ E[f(.) | Z]     E[g(.) | Z] \big]\nonumber\\
%&=&E_Z\Big[  E[ f(X) g(Y)| Z]  -  E[f(.) | Z]     E[g(.) | Z] \Big] \nonumber\\
&=&E_Z \Big[  \cov[ f(X) , g(Y)| Z]  \Big]
\end{eqnarray}
The next question concerns whether conditional independence can be measured using the conditional covariance operator or not? The previous  result  and the first one in the following theorem show  that a zero conditional covariance operator  is not equivalent to conditional independence, but equivalent to a weaker form. The second result in the theorem below shows how to slightly modify the covariance operator to obtain the equivalence. This theorem is also from Fukumizu and his colleagues \cite{FukuBJ04}.
We suppose in the following that all the kernels used  are characteristic, and that the conditional mean involved belongs to the proper RKHS. 
  \begin{theorem}
  \label{condindep:th}
  Let $X,Y,Z$ be three random vectors embedded in corresponding RKHS.
  Then we have
  \begin{enumerate}
\item $\Sigma_{XY|Z}=0  \Longleftrightarrow  P_{XY} = E_Z[ P_{X|Z}\otimes P_{Y|Z} ]$
\item $\Sigma_{(XZ) Y | Z}=0  \Longleftrightarrow X \perp Y | Z$.
\end{enumerate}
\end{theorem}

  \begin{proof}
  
  {\bf First assertion.}
  We have seen that 
  \begin{eqnarray}
\big< f , \Sigma_{XY|Z}g \big> &=& E[ f(X) g(Y)]  - E_Z\big[ E[f(.) | Z]     E[g(.) | Z] \big]
\end{eqnarray}
which can be written as
\begin{eqnarray}
\lefteqn{\big< f , \Sigma_{XY|Z}g \big> =}\\
% \int f(x)g(y) P_{XY} (dx,dy)-\int P_Z(dz) \int f(x)g(y) P_{X|Z}(dx,z)P_{Y|Z}(dy,z)\\
&& \int f(x)g(y) \Big(P_{XY} (dx,dy)-\int P_Z(dz)  P_{X|Z}(dx,z)P_{Y|Z}(dy,z)\Big)\nonumber
\end{eqnarray}
Thus obviously, if for all $A$ and $B$ in the adequate sigma algebra
\begin{eqnarray}
P_{XY} (A,B) = \int P_Z(dz)  P_{X|Z}(A,z)P_{Y|Z}(B,z)
\end{eqnarray}
we have $\big< f , \Sigma_{XY|Z}g \big> = 0$ for all $f$ and $g$ leading necessarily to $\Sigma_{XY|Z}=0$. 
Now if the covariance operator is zero then we have for all $f$ and $g$ $E_{P_{XY}}[f(X)g(Y)]=
E_Q\big[ f(X)g(Y)\big]$ where $Q= E_Z[ P_{X|Z}\otimes P_{Y|Z} ]$. Working in the tensorial product $V_x\otimes V_y$ where we have assumed $K_x K_y$ as a characteristic kernel allows to conclude that $Q=
  P_{XY}$. 
  
  {\bf Second assertion.} Let $A,B,C$ be elements of the sigma algebra related to  $X,Y$ and $Z$ respectively. Let 
 $ {\bf 1}_{A}$ the characteristic function of set $A$. 
Then we have
\begin{eqnarray}
\lefteqn{P_{XZY}(A,C,B) - E_Z[ P_{XZ|Z}(A,C) P_{Y|Z}(B)]}\nonumber\\
&=&E[ {\bf 1}_{A\times C}(X,Z) {\bf 1}_{B}(Y)]  - E_Z\big[ E[{\bf 1}_{A\times C}(X,Z)| Z]     E[{\bf 1}_{B}(Y) | Z] \big]\nonumber\\
%&= &E_Z\big[ E[ {\bf 1}_{A\times C}(X,Z) {\bf 1}_{B}(Y) | Z] \big]- E_Z\big[{\bf 1}_{C}(X) E[{\bf 1}_{A}(X)| Z]     E[{\bf 1}_{B}(Y) | Z] \big]\\
%&=& E_Z\big[{\bf 1}_{C}(Z) E[ {\bf 1}_{A}(X) {\bf 1}_{B}(Y)| Z] \big]- E_Z\big[{\bf 1}_{C}(Z) E[{\bf 1}_{A}(X)| Z]     E[{\bf 1}_{B}(Y) | Z] \big]\\
&=& E_Z\Big[{\bf 1}_{C}(Z) \Big( E[ {\bf 1}_{A}(X) {\bf 1}_{B}(Y)| Z] - E[{\bf 1}_{A}(X)| Z]     E[{\bf 1}_{B}(Y) | Z]\Big) \Big]\nonumber\\
&=&\int_C P_Z(dz) \Big(  P_{X,Y|Z}(A,B,z) - P_{X|Z}(A,z)P_{X|Z}(B,z) \Big)
\end{eqnarray} 
If $\Sigma_{(XZ) Y | Z}=0$ then the first assertion implies
  $P_{XZY}=E_Z[ P_{XZ|Z}\otimes P_{Y|Z}]$ and the previous integral is equal to zero for any $C$, 
  which in turn implies that $P_{X,Y|Z}(A,B,z) - P_{X|Z}(A,z)P_{X|Z}(B,z)$ almost everywhere ($P_Z$) for any $A,B$. But this is precisely the defintion of conditional independence. The converse is evident.  
    \end{proof}
  
  The conclusion of this theorem is that the variables $X$ and $Y$ has to be extended using $Z$ {\it prior }
  conditioning.  Assessing conditional  independence relies on  the conditional covariance operators (extended as above.) However, as done for independence testing, a measure is needed. The Hilbert-Schmidt norm 
  $\|    \Sigma_{(XZ) (YZ) | Z}\|^2$ is used for that purpose. \\

\noindent {\bf Estimation.}  The estimators of the conditional measures have representations in terms of Gram matrices. In the following, the indication of the RKHS in the inner product  is suppressed for the sake of readability. For $N$ identically distributed observations $(x_i,y_i,z_i)$  the application of the empirical covariance estimator
to a function is given by 
\begin{eqnarray}
\widehat{\Sigma}_{XY} f = \frac{1}{N}\sum_j \tilde{K}_x(., x_j)\big<  \tilde{K}_y(., y_j),  f\big>
\end{eqnarray}
where the tildas mean  the kernel are centered. 
The effect of the regularized  inverse of this operator on a function $\sum_i \beta_i K_x(.,x_i)$ is to produce 
the function $\sum_i \alpha_i K_x(.,x_i)$ with $\valpha  = N  (\widetilde{\vK}_x +N\lambda I) ^{-1} \vbeta:=N\widetilde{\vK}_{r,x}^{-1}\vbeta$.
%
%Now we now that applying its inverse to the direct operator, we should obtain the identity,
%thus 
%\begin{eqnarray}
%\widehat{\Sigma}^{-1}_{r,XX} \widehat{\Sigma}_{r,XX}  f =  f
%\end{eqnarray}
%Since we work in the subspace generated by  the $N$ $K_x(.,x_i)$ we should have the two following relations
%\begin{eqnarray}
%\widehat{\Sigma}^{-1}_{r,XX} \widehat{\Sigma}_{r,XX} K_x(.,x_n) &=&  K_x(.,x_n) \\
%\widehat{\Sigma}^{-1}_{r,XX} K_x(.,x_n)&=& \sum_k (\vH_x)_{nk} K_x(.,x_k)
%\end{eqnarray}
%Therefore, 
% \begin{eqnarray}
% \tilde{K}_x(.,x_n) &=& \widehat{\Sigma}^{-1}_{r,XX}  \frac{1}{N}\sum_k \tilde{K}_x(., x_k)\big<  \tilde{K}_x(., x_k),  K_x(.,x_n)\big>\\
% &=& \frac{1}{N}\sum_{k,l} (\widetilde{\vK}_{r,x})_{kn}  (\vH_x)_{lk} K_x(.,x_l)
%\end{eqnarray}
%This is obtain if $ \sum_{k }(\widetilde{\vK}_{r,x})_{kn}  (\vH_x)_{lk} =N \delta_{nl}$.
%Therefore, $\vH_x=N\widetilde{\vK}_{r,x}^{-1}$, and we have 
%\begin{eqnarray}
%\widehat{\Sigma}^{-1}_{r,XX} K_x(.,x_n)&=& N\sum_k (\widetilde{\vK}_{r,x}^{-1})_{nk} K_x(.,x_k)
%\end{eqnarray}
%We can now obtain the representation of the empirical measures in terms of Gram matrices. 
The inner product   $\big< f , \Sigma_{XY|Z} g \big>$ for $f(.)=\sum_i \alpha_i \tilde{K}_x(.,x_i)$ and $g(.)=\sum_i \beta_i  \tilde{K}_y(.,y_i)$ is evaluated for $N$ triple $x_i,y_i,z_i$ identically distributed. The result for the covariance operator is known from the first section
\begin{eqnarray}
\big<   f , \widehat{\Sigma}_{XY} g\big> =\frac{1}{N} \vbeta^\top \widetilde{\vK}_y \widetilde{\vK}_x \valpha
\end{eqnarray}
Then the remaining term  in $ \widehat{\Sigma}_{XY|Z} = \widehat{\Sigma}_{XY} -\widehat{\Sigma}_{XZ}\widehat{\Sigma}_{ZZ}^{-1}\widehat{\Sigma}_{ZY} $ leads to 
\begin{eqnarray}
\lefteqn{ \big<   f , \widehat{\Sigma}_{XZ}\widehat{\Sigma}_{ZZ}^{-1}\widehat{\Sigma}_{ZY} g\big> =\sum_{i,j} \alpha_i\beta_j\big<    \tilde{K}_x(.,x_i) , \widehat{\Sigma}_{XZ}\widehat{\Sigma}_{ZZ}^{-1}\widehat{\Sigma}_{ZY}  \tilde{K}_y(.,y_j)\big>}  \nonumber \\
&=&\frac{1}{N^2}\sum_{i,j,k,l} \alpha_i\beta_j
(\widetilde{\vK}_{y})_{kj} (\widetilde{\vK}_{x})_{li}
\big<    \tilde{K}_z(.,z_l) , \widehat{\Sigma}_{ZZ}^{-1}  \tilde{K}_z(.,z_k)\big> \nonumber\\
%&=&\frac{1}{N}\sum_{i,j,k,l,m} \alpha_i\beta_j (\widetilde{\vK}_{y})_{kj} (\widetilde{\vK}_{x})_{li}
%(\widetilde{\vK}_{r,z}^{-1})_{km} (\widetilde{\vK}_{z})_{ml} \nonumber\\
&=&\frac{1}{N}\vbeta^\top   \widetilde{\vK}_{y} \widetilde{\vK}_{r,z}^{-1} \widetilde{\vK}_{z}\widetilde{\vK}_{x} \valpha
\end{eqnarray}
Thus the final result is then 
\begin{eqnarray}
\big<   f , \widehat{\Sigma}_{XY|Z} g\big> =\frac{1}{N} \vbeta^\top\Big( \widetilde{\vK}_y \widetilde{\vK}_x 
- \widetilde{\vK}_{y} \widetilde{\vK}_{r,z}^{-1} \widetilde{\vK}_{z}\widetilde{\vK}_{x} \Big)\valpha
\end{eqnarray}

\noindent {\bf Hilbert-Schmidt norms.}
Practically, a measure using the cross-covariance is preferable. Like for independence testing, a nice measure is provided by measuring the norm of the operator. For the Hilbert-Schmidt norm, we have
\begin{eqnarray}
\lefteqn{\Big\|  \widehat{\Sigma}_{XY|Z}    \Big\|_{HS}^2 =  \sum_i \big< \widehat{\Sigma}_{XY|Z}\varphi_i ,  \widehat{\Sigma}_{XY|Z}\varphi_i \big>}   \\
&=& \Big\|  \widehat{\Sigma}_{XY}  \Big\|_{HS}^2  + \Big\|   \widehat{\Sigma_{XZ}}  \widehat{\Sigma_{ZZ}^{-1}}  \widehat{\Sigma_{ZY}}   \Big\|_{HS}^2 
-2\sum_i \big< \widehat{\Sigma}_{XY} \varphi_i,   \widehat{\Sigma_{XZ}}  \widehat{\Sigma_{ZZ}^{-1}}  \widehat{\Sigma_{ZY}}\varphi_i\big>
 \nonumber \end{eqnarray}
where we recall that $\{\varphi_i \}_{i\in \N}$ is an orthonormal basis of $ V_y$. The double product term is denoted as $P$. It reads
\begin{eqnarray}
P &:=&\sum_i \big< \widehat{\Sigma}_{XY} \varphi_i,   \widehat{\Sigma_{XZ}}  \widehat{\Sigma_{ZZ}^{-1}}  \widehat{\Sigma_{ZY}}\varphi_i\big> \nonumber\\
&=&\frac{1}{N}\sum_{i,k}\big< \tilde{K}_y(.,y_k) ,  \varphi_i\big>
\big< \tilde{K}_x(.,x_k) , \widehat{\Sigma_{XZ}}  \widehat{\Sigma_{ZZ}^{-1}}  \widehat{\Sigma_{ZY}}\varphi_i\big> \nonumber\\
&=& \frac{1}{N^2}\sum_{i,k,l}\big< \tilde{K}_y(.,y_k) ,  \varphi_i\big>\big< \tilde{K}_y(.,y_l) ,  \varphi_i\big>
\big< \tilde{K}_x(.,x_k)  , \widehat{\Sigma_{XZ}}  \widehat{\Sigma_{ZZ}^{-1}} \tilde{z}_y(.,z_l)\big> \nonumber \\
%&=& \frac{1}{N}\sum_{k,l,m}(\widetilde{\vK}_{y})_{kl} (\widetilde{\vK}_{r,z}^{-1})_{lm}
%\big< \tilde{K}_x(.,x_k)  , \widehat{\Sigma_{XZ}}  k_z(.,z_m)\big> \\
%&=& \frac{1}{N^2}\sum_{k,l,m,n}(\widetilde{\vK}_{y})_{kl} (\widetilde{\vK}_{r,z}^{-1})_{lm}
%\big< \tilde{K}_x(.,x_k)  , \tilde{K}_x(.,x_n)  \big> \big< \tilde{K}_z(.,z_n)  , \tilde{K}_z(.,z_m)  \big> \\
&=& \frac{1}{N^2}\sum_{k,l,m,n}(\widetilde{\vK}_{y})_{kl} (\widetilde{\vK}_{r,z}^{-1})_{lm}
(\widetilde{\vK}_{x})_{kn}(\widetilde{\vK}_{z})_{mn} \nonumber\\
%&=& \frac{1}{N^2}\sum_{k,m}(\widetilde{\vK}_{y}  \widetilde{\vK}_{r,z}^{-1})_{km}(\widetilde{\vK}_{z}
%\widetilde{\vK}_{x})_{mk} \nonumber\\
&=& \frac{1}{N^2} \trace\Big(   \widetilde{\vK}_{y} \widetilde{\vK}_{r,z}^{-1} \widetilde{\vK}_{z}
\widetilde{\vK}_{x} \big)
\end{eqnarray}
Carrying the same calculation for the last term allows to obtain 
\begin{eqnarray}
\Big\|  \widehat{\Sigma}_{XY|Z}  \Big\|_{HS}^2 &=& \frac{1}{N^2}
\trace\Big( \widetilde{\vK}_{x}\widetilde{\vK}_{y}   -
2 \widetilde{\vK}_{y} \widetilde{\vK}_{r,z}^{-1} \widetilde{\vK}_{z }\widetilde{\vK}_{x} \\ & +& 
 \widetilde{\vK}_{y}\widetilde{\vK}_{z } \widetilde{\vK}_{r,z}^{-1}  \widetilde{\vK}_{x} \widetilde{\vK}_{r,z}^{-1} \widetilde{\vK}_{z }\Big)    
\end{eqnarray}

If  the normalized version $V_{XY|Z}=V_{XY}- V_{XZ}V_{ZY}$ is used,  the estimator of the Hilbert-Schmidt norm $\big\| V_{XY|Z} \big\|^2_{HS}$ evaluated using the empirical estimate $\widehat{V}_{XY|Z} =\widehat{\Sigma}_{r,XX}^{-1/2}\widehat{\Sigma}_{XY|Z}    \widehat{\Sigma}_{r,YY}^{-1/2}$ is given by 
\begin{eqnarray}
\Big\|    \widehat{V}_{XY|Z}  \Big\|_{HS}^2 \!\!\!&=& \!\!\!
\trace\Big(\vN_{x}\vN_{y}   -
2 \vN_{y} \vN_{z }\vN_{x} + 
 \vN_{y}\vN_{z } \vN_{x} \ \vN_{z }\Big)    
\end{eqnarray}
where $\vN_u$ is the normalized centered Gram matrix for variable $u=x,y$ or $z$, and reads 
$\vN_u=\widetilde{\vK}_{u}  \widetilde{\vK}_{r,u}^{-1} $.
The proof follows the same line as before. This estimator has been shown to converge to $\Big\|  {\Sigma}_{XY|Z}  \Big\|_{HS}^2 $ in probability in \cite{FukuGSS07}. To obtain the result, the regularization parameter $\lambda$ must of course depend on $N$ and goes to zero at an appropriate rate. Refer to \cite{FukuBJ04,FukuBJ09,FukuGSS07} for results concerning the consistency of all the estimates seen so far. 

Following theorem \ref{condindep:th}, the use of these measures to assess conditional independence is not enough.
The extension of the random variables including $Z$ must be considered. The measure to be tested is thus $\Big\|  \widehat{V}_{(X Z)Y|Z}  \Big\|_{HS}^2$.

\noindent {\bf A simple illustration.} The simplest toy example to illustrate conditional HSIC is to test that three random variables
$X,Y,Z$ constitute a Markov chain, {\it e.g. } $X$ and $Z$ are independent conditionally to $Y$. Consider the simple generative model
\begin{eqnarray*}
X&=&U_1\\
Y&=& a (X^2-1) +U_2\\
Z &=& Y + U_3
\end{eqnarray*}
where $U_1,U_2,U_3$ are three independent zero mean, unit variance  Gaussian random variables, and $a\in \R$ is a coupling parameter. $N=512$   independent and identically distributed samples $(x_i,y_i,z_i)$ were generated and used to estimate $\Big\|  \widehat{V}_{(X Y) Z|Y}  \Big\|_{HS}^2$ using the equations above.

 Practically with finite length data, 
the distribution of the measure is not known under both hypothesis  $H_0$ :  Markov and $H_1$: non Markov. To simulate  the null hypothesis $X-Y-Z$ is a Markov chain, equivalently, $X,Z$ are independent conditionally to $Y$ or $\Big\|  \widehat{V}_{(X Y) Z|Y}  \Big\|_{HS}^2=0$ , we use random permutations of the realizations of one variable, say $X$. This is usually done for independence test, since randomly permuting preserve empirical distributions. However here, some
care must be taken because the distributions that must be preserved under permutations are the conditional distributions. Thus, to create the permuted data,  the range of the conditioning variable $Y$ is partitioned into $L$ domains 
$Y_1,\ldots,Y_L$ such that each domain contains the same number of observations. 
%The data $\{x_j {\mbox{ such that } y_j \in Y_l \}$ are then randomly permuted, and this for all $l$. s
The permuted observations $\tilde x_i$ are obtained {\it per } domain, {\it i.e. }, $\forall l=1,\ldots, L$, $\tilde x_j=x_{\sigma(j)}$ for those $j$s such that $y_j \in Y_l$, where $\sigma(.)$ is a random permutation of these $j$s.
For this toy problem, 100 permuted data sets were created. This allows to find the threshold corresponding to 
a $5\%$ false alarm probability (level of the test). 

Figure \ref{test-markov:fig} displays the result of the simulation for this simple example. We test the three possibilities of having the  Markov property among $X,Y,Z$. The coupling parameter $a$ is varied from 0 to 1. The plot displays the conditional HSIC measure, as well as the threshold that ensures at most  5$\%$ of false alarms (evaluated as mentioned above
with $L=8$ domains.) In the left plot,
the squared norm of $\widehat{V}_{(X Z) Y|Z} $ is plotted to test 
the Markov property  for $X-Z-Y$. In the right plot, the squared norm of $\widehat{V}_{(X Y) Z|Y} $ is displayed to test if $X-Y-Z$ is a Markov chain. Finally, in the middle plot, the squared norm of $\widehat{V}_{(Y X) Z|X} $ is plotted to test 
the Markov property  for $Y-X-Z$. In the left plot for $a=0$, $X$ and $Y$ are independent and $X-Z-Y$ is a particular  
Markov chain. However, as soon as $a>0$, $X-Z-Y$ is not a Markov chain, and this is correctly inferred by the measure for $a$ as small as 0.1. The fact that $Y-X-Z$ is not a Markov chain is clearly assessed also, as illustrated in the middle plot. Finally, the Markov case $X-Y-Z$ is also correctly assessed in the right plot, since the squared norm of 
$\widehat{V}_{(X Y) Z|Y} $ is always below the threshold insuring a $5\%$ level test.
\begin{figure}[t]
\centering
\includegraphics{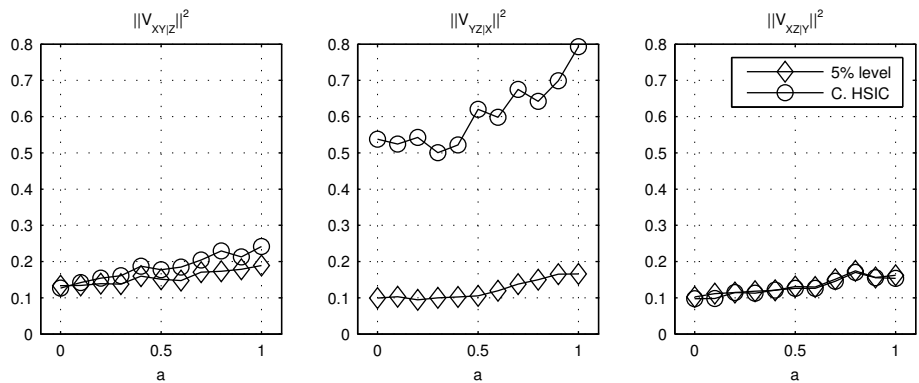}
\caption{Hilbert-Schmidt norms of conditional covariance operator to assess conditional independence. Diamond represent the threshold insuring a 5$\%$ level test.
Left plot:  Testing for the chain $X\rightarrow Z\rightarrow Y$. Middle plot:  Testing for the chain $Y\rightarrow X\rightarrow Z$.  Right plot: Testing for the chain $X\rightarrow Y\rightarrow Z$. For the two leftmost plots, the measures
is above the threshold meaning that the alternative $H_1$ is chosen by the test, except for $a=0$ in the case 
$X\rightarrow Z\rightarrow Y$. These results are perfectly consistent with the simple model used. In the right plot, the measure is always below the diamond curves, meaning that $H_1$ is rejected, and that $X-Y-$ is Markov, in agreement with the model.}
\label{test-markov:fig}
\end{figure}

\section{Kernel Bayesian filtering}

The idea developed mainly by Song, Fukumizu and Gretton is to transfer into reproducing kernel Hilbert spaces the manipulation of probability measures used in inference, namely the sum rule, the chain rule and their application in Bayesian inference \cite{FukuSG13,SongFG13}.  Since conditional distribution may be embedded into RKHS, it seems natural to seek for the generalization of Bayesian inference in RKHS.

The starting point is   theorem \ref{indepcondeqm:th} which states that the conditional covariance operator $\Sigma_{YY|X}:  V_y \rightarrow  V_y$
is deeply linked to optimal estimation of random variables of $ V_y$ from transformation of random variables in $ V_x$. Precisely,
 $\Sigma_{XX}  E[g(Y) | X=.] = \Sigma_{XY} g$ provided that the function $E[g(Y) | X=.]$ belongs to $ V_x$, a fact that is not guaranteed for any kernel, but a fact  assumed to be true in the sequel (see \cite{FukuBJ04} for a sufficient condition.) Recall that the conditional covariance operator is defined as 
\begin{eqnarray}
\Sigma_{YY|X} = \Sigma_{YY} - \Sigma_{YX} \Sigma_{XX}^{-1} \Sigma_{XY}
\end{eqnarray}
a formula which assumes the invertibility of the covariance operator in $ V_x$, and which exactly match the conditional covariance formula for Gaussian finite dimensional vectors.

The expression of the conditional mean allows to obtain an explicit form for the conditional kernel mean, that is, for the embedding of the conditional distribution. 
Let $\mu_X=E[K_x(.,X)]$ and $\mu_Y=E[K_y(.,Y)]$ be the embeddings of $X$ and $Y$ respectively in $ V_x$ and $ V_y$. The 
two embeddings are linked by
\begin{eqnarray}
\mu_Y= \Sigma_{YX}\Sigma_{XX}^{-1} \mu_X
\label{muydef:eq}
\end{eqnarray}
To show this relation,  the conditional expectation of $g(Y)$ given $X$ which satisfies
$
E[g(Y) | X=.]= (\Sigma_{XX}^{-1} \Sigma_{XY} g) (. ) 
$
is used in the following set of equations, valid for all $g\in  V_y$,
\begin{eqnarray}
\big<  \Sigma_{YX} \Sigma_{XX}^{-1} \mu_X ,   g \big>_{ V_y} 
&=& \big<  \mu_X ,  \Sigma_{XX}^{-1}  \Sigma_{XY} g \big>_{ V_x} \nonumber \\
&=&\big<  \mu_X ,  E[g(Y) | X] \big>_{ V_x} \nonumber \\
&=& E_X[E[g(Y) | X] ]  \nonumber \\
&=& E[g(Y) ] \nonumber  \\
&=& \big< \mu_Y ,  g \big>_{ V_y}
\end{eqnarray}

In particular, since $\mu_Y=E_Y[K_y(.,Y)]= E_XE_{Y|X}[K_y(.,Y)]=E_X[\mu_{Y|X}]$, setting  $P(dX)=\delta_{x}(dX)$ in (\ref{muydef:eq}) allows to obtain the conditional kernel mean as 
\begin{eqnarray}
\mu_{Y|x}(.)=E_{Y|X=x}[K_y(., Y)]& = & \Sigma_{YX}\Sigma_{XX}^{-1} K_x(.,x) \nonumber\\
&:=&\Sigma_{Y|X}  K_x(.,x)
\end{eqnarray}
Note that the embedding $\mu_{Y|x}(.)$ is a function which belongs to $V_y$.

In these formula, the couple $X,Y$ is distributed according to the joint probability $P(X,Y)$. The covariance operator $\Sigma_{YX}$ is defined as 
the expected value of the tensor product $K_y(.,Y)\otimes K_x(.,X)$ over the joint probability. Recall that it has three interpretations. The first  considers the covariance  as a tensor, it is to say 
a bilinear functional over the product $V_y \times V_x$ (precisely their duals). The second well-know fact  is based on the very definition of the tensor product $(K_y(.,Y)\otimes K_x(.,X))(g,f) = \left< g(.)  \left|\right. K_y(.,Y)  \right>\left< f(.)  \left|\right. K_x(.,X)  \right>$, which allows to write $\Sigma_{YX}f = E[\left< f(.)  \left|\right. K_x(.,X)  \right>K_y(.,Y)]$ and to consider
$\Sigma_{YX}$ as a linear operator from $V_x$ to $V_y$. The third interpretation considers $\Sigma_{YX}$ as the embedding of the joint probability into the tensor product space
$V_y\otimes V_x$. Since under this interpretation the covariance $\Sigma_{XX}$ can be seen as the embedding of $P(X)$ into the tensor product $V_x \otimes V_x$, this point of view allows to consider $\Sigma_{YX}\Sigma_{XX}^{-1} $ as  a representation of 
\begin{eqnarray}
P(Y|X)=\frac{P(X,Y)}{P(X)} =\frac{P(X|Y)  \int P(dX,Y) }{ \int P(X,dY)}
\end{eqnarray}
Writing the conditional kernel embedding as  $\mu_{Y|x}(.)= \Sigma_{YX}\Sigma_{XX}^{-1} K_x(.,x)$ is at the root of the embedding of Bayesian inference into RKHS. It can be seen as the embedding of Bayes law when the likelihood is $P(X|Y)$ and the {\it prior } probability is $P(Y)=\int P(dX,Y)$. However,
in Bayesian inference, if the likelihood is generally given, the {\it prior } is not given and in general not equal to the marginal of a given joint probability distribution. 

Thus for Bayesian inference,   the previous formula for the conditional embedding can be used, but for 
\begin{eqnarray}
P(Y|X) =\frac{Q(X,Y)  }{ \int Q(X,dY)} 
=  \frac{P(X|Y)\pi(Y) }{ \int Q(X,dY)}
\end{eqnarray}
The joint probability to be considered is no longer $P(X,Y)=P(X|Y)P(Y)$ but 
instead $Q(X,Y)=P(X|Y)\pi(Y)$. The embedding of the {\it a posteriori} probability is then given by 
\begin{eqnarray}
\mu^\pi_{Y|x}  = E^\pi[K_y(., Y)| x] =  \Sigma^\pi_{YX}\Sigma_{XX}^{\pi-1} K_x(.,x)
\end{eqnarray}
where  the superscript $\pi$ reminds that the {\it a priori} probability is $\pi$ instead of $\int P(dX,Y) $.

It is possible to relate $\Sigma^\pi_{YX}$  (or its adjoint $ \Sigma^\pi_{XY} $) and $\Sigma_{XX}^{\pi-1}$ to embeddings evaluated
on the joint probability $P$, using the following set of equations
\begin{eqnarray}
\Sigma^\pi_{XY} &=& E_Q[ K_x(.,X) \otimes K_y(.,Y) ]\nonumber \\
&=& E_\pi\left[  E[   K_x(.,X) \otimes K_y(.,Y) \big|  Y] \right] \nonumber\\
&=& E_\pi\left[E[   K_x(.,X) \big|Y] \otimes K_y(.,Y) \right]\nonumber \\
&=& E_\pi\left[\mu_{X|Y} \otimes K_y(.,Y) \right] \nonumber\\
&=&E_\pi\left[\Sigma_{X|Y} K_y(.,Y) \otimes K_y(.,Y) \right] \nonumber\\
&=& \Sigma_{X|Y} \Sigma^\pi_{YY}
\end{eqnarray}
The second line in this equation can also be interpreted as the average of the embedding of $P(Y,X |   Y)$: this interpretation offers an alternative expression as 
\begin{eqnarray}
\Sigma^\pi_{XY} &=& E_\pi\left[  E[   K_x(.,X) \otimes K_y(.,Y) \big|  Y] \right]\nonumber \\
&=& E_\pi\left[\mu_{XY|Y} \right]\nonumber \\
&=&E_\pi\left[\Sigma_{XY|Y} K_y(.,Y) \right] \nonumber\\
&=& \Sigma_{XY|Y} \mu^\pi_{Y}
\end{eqnarray}
Likewise, the covariance operator reads
\begin{eqnarray}
\Sigma_{XX}^\pi &= &\Sigma_{XX|Y} \mu_{Y}^\pi
\end{eqnarray}

\noindent {\bf Interpretations.} The operators $\Sigma^\pi_{XY}$ and $\Sigma^\pi_{XX}$ have simple interpretations when considered as embeddings.
$\Sigma^\pi_{XX}$ corresponds to the embedding of the law $Q(X)= \int P(X|Y) \pi (dY)$ into the tensorial product $ V_x\otimes  V_x$. 
$\Sigma^\pi_{XY}$ is the embedding into $ V_x\otimes  V_y$ of $Q(X,Y)=P(X|Y)\pi(Y)$.
Thus $\Sigma^\pi_{XX}$ can be seen as the embedding of the sum rule, and is thus called {\bf kernel sum rule}, whereas 
$\Sigma^\pi_{XY}$ is the embedding of the chain rule, and is thus called {\bf kernel chain rule}.
Obviously, Bayesian manipulation are a succession of applications of these rules. \\

To sum up, the  embedding of the {\it a posteriori} probability reads
\begin{eqnarray}
\mu_{Y|x} \quad =  \quad \Sigma^\pi_{YX}\Sigma_{XX}^{\pi-1} K_x(.,x)  
\quad= \quad(\Sigma^\pi_{XY})^\top \Sigma_{XX}^{\pi-1} K_x(.,x)  
\label{bayesrule:eq}
\end{eqnarray}
where
\begin{eqnarray}
\mbox{\bf Chain rule} &:&
\left\{
\begin{array}{ccccc}
\Sigma^\pi_{XY} &=&\!\!  \Sigma_{X|Y} \Sigma^\pi_{YY} &=&  \Sigma_{XY} \Sigma^{-1}_{YY} \Sigma^\pi_{YY}  \\
\mbox{or} &=& \Sigma_{XY|Y} \mu^\pi_{Y}  &=& \Sigma_{(XY)Y}\Sigma^{-1}_{YY}  \mu^\pi_{Y}
\end{array}\right.\label{chainrule:eq}  \\
%\Sigma^\pi_{XY} &=& \Sigma_{X|Y} \Sigma^\pi_{YY} 
%\quad = \quad  \Sigma_{XY} \Sigma^{-1}_{YY} \Sigma^\pi_{YY} \nonumber  \\
%\mbox{or} &=& \Sigma_{XY|Y} \mu^\pi_{Y}  \quad = \quad  \Sigma_{(XY)Y}\Sigma^{-1}_{YY}  \mu^\pi_{Y}  \\
\mbox{\bf Sum rule}&:&\!\! \Sigma_{XX}^\pi  = \Sigma_{XX|Y} \mu_{Y}^\pi\quad = \quad\Sigma_{(XX)Y}\Sigma^{-1}_{YY}   \mu_{Y}^\pi
\label{sumrule:eq}
\end{eqnarray}

\noindent {\bf Estimators.} Estimators  for 
\begin{eqnarray}
\mu_{Y|x}&=&  \Sigma^\pi_{YX}\Sigma_{XX}^{\pi-1} K_x(.,x)  \\
\Sigma^\pi_{XY} &=& \Sigma_{(XY)Y}\Sigma^{-1}_{YY}  \mu^\pi_{Y}  \\
\Sigma_{XX}^\pi &=& \Sigma_{(XX)Y}\Sigma^{-1}_{YY}   \mu_{Y}^\pi
\end{eqnarray}
are obtained using empirical estimators of the different covariance operators. 
The last two operators are seen as linear operator from $V_y$ into respectively $ V_x\otimes  V_y$
and $ V_x\otimes  V_x$. 
Let us find an estimator for $\Sigma^\pi_{XY}$. The other will be obtained immediately by replacing $\Sigma_{(XY)Y}$ with $\Sigma_{(XX)Y}$.

The estimators are based on the observation of $N$ i.i.d. samples $(x_i,y_i)$ of the couple $(X,Y)$. We denote by $\vK_x$ and $\vK_y$ the Gram matrices evaluated on this sample. 
Furthermore, since   information about the {\it prior } is needed,   a number $N_\pi$ of i.i.d. 
samples $Y^\pi_i$ from the {\it prior } $\pi$ are assumed to be observed. This seems a strange assumption, but in many situations, these samples are at hand. For example, in recursive nonlinear filtering, the {\it posterior } probability at a particular  time step serves as {\it prior } probability for the following time step. This will  be detailed   in the next section. 
The estimator for the function $ \mu^\pi_{Y}  $  is written as
\begin{eqnarray}
\mu^\pi_{Y}(.)= \sum_{i=1}^{N_\pi} \gamma_i K_y(., Y^\pi_i)
\end{eqnarray}
Let $\vmu^\pi_Y$ the vector containing the $ \mu^\pi_{Y}(y_k)$.
Using the results of \S\ref{matrixrep:ssec},  we know that
$\Sigma^{-1}_{YY}   \mu_{Y}^\pi(.)=\sum_i \beta_i K_y(.,y_i)$ where
\begin{eqnarray}
\vmu^\pi_Y = \frac{1}{N}(\vK_y +N\lambda I)\vK_y \vbeta
\end{eqnarray}
%
%
%
%
%To evaluate an estimate of $\Sigma^\pi_{XY} $, first we find an estimate of $\Sigma^{-1}_{YY}  \mu^\pi_{Y}$.
%It is a function in $ V_y$ and we look for it in the form $\sum_i \beta_i K_y(., y_i)  + \perp$. Further we use the 
%regularized version of the inverse.
%
%Then we have
%\begin{eqnarray}
% \mu^\pi_{Y}(.) &=& (\Sigma_{YY}+\lambda  I) (\sum_i \beta_i K_y(., y_i)  + \perp) \\
% &=& \sum_{ij} \beta_i \vK_{y,ij} K_y(.,y_j) + \lambda \sum_i \beta_i K_y(., y_i)   + \lambda \perp\\
%  \mu^\pi_{Y}(y_k) &=& \sum_{ij} \beta_i \vK_{y,ij} \vK_{y,k,j} + \lambda \sum_i \beta_i \vK_{y,ki} 
%\end{eqnarray}
%Let $\vmu^\pi_Y$ the vector containing the $ \mu^\pi_{Y}(y_k)$. We then have
%$\vmu^\pi_Y = (\vK_y +\lambda I) \vK \vbeta$. 
Applying $\Sigma_{(XY)Y}$ to $\Sigma^{-1}_{YY}  \mu^\pi_{Y} $ leads to
\begin{eqnarray}
\Sigma^\pi_{XY} &=&\frac{1}{N} \sum_{ij}    \beta_i \vK_{y,ij} K_x(., x_j)\otimes K_y(.,y_j) \nonumber  \\
        &=& \sum_j \mu_j  K_x(., x_j)\otimes K_y(.,y_j)  \mbox{ where }\\
    \vmu &=&   (\vK_y +\lambda I)^{-1}  \vmu^\pi_Y
\end{eqnarray}
Likewise
\begin{eqnarray}
\Sigma^\pi_{XX}         &=& \sum_j \mu_j  K_x(., x_j)\otimes K_y(.,x_j)  \mbox{ where }\\
    \vmu &=&   (\vK_y +\lambda I)^{-1}  \vmu^\pi_Y
\end{eqnarray}
 To get an estimate for $\mu_{Y|x}$ note that $\Sigma^\pi_{YX}=\Sigma_{XY}^{\pi,\top} =  \sum_j \mu_j  K(., y_j)\otimes K(.,x_j) $.
 Since $\Sigma_{XX}^\pi$ is not insured to be positive definite,   the regularization $(\Sigma^2+\varepsilon I)^{-1} \Sigma$ of  the inverse is used.
  Doing as above, searching for $\mu_{Y|x}(.)=\sum_j \zeta_j(x) K_y(.,y_j)$,   the vector
  \begin{eqnarray}
\vzeta (x)&= &\Lambda \left(  (\vK_x \Lambda )^2 + \varepsilon I\right)^{-1} \vK_x \Lambda \vk_X(x) \nonumber \\
&=& \Lambda\vK_x \left(  (\vK_x \Lambda )^2 + \varepsilon I\right)^{-1}  \Lambda \vk_X(x) 
\end{eqnarray}
is finally obtained, where  $\vk_X(x)=(K_x(x_1,x),\ldots, K_x(x_N,x))^\top$ and $\Lambda = \diag(\vmu)$ is a diagonal matrix, the diagonal  elements of which are the entries of $\vmu$. \\
Note that the matrix representation presented here has been shown to converge to the true embedding when the number of data goes to infinity, and when the regularization parameters goes to zero at correct speed (see \cite{FukuSG13}.)

Some application of kernel Bayes rule were presented in \cite{FukuSG13,SongFG13}, among which a kernel belief propagation for inference on graphs, Bayes inference problems with unknown likelihood (an alternative solution to Approximate Bayesian Calculation), and to filtering.
In the sequel,  kernel Bayes filtering is developed and applied to the prediction of time series.

%\subsection{Another stuff on estimators}
%
%
%More generally, the function
%\begin{eqnarray}
%g(.) = \Sigma_{XY}\Sigma_{YY}^{-1} f(.)
%\end{eqnarray}
%is estimated by  
%\begin{eqnarray}
%\hat{g}(.) &=& \sum_i \vmu_i K_x(., x_i)  \quad \mbox{where}\\
%\vmu &=& (\vK_y +\lambda \vI)^{-1} \vf \quad \mbox{with } \vf = (f(y_1),\ldots,f(y_N))^\top 
%\end{eqnarray}
%or 
%
%

\noindent {\bf Application in filtering.}
The problem of filtering is to estimate an hidden state $x_k$ from past observations $y_{1:k}:=(y_1,\ldots,y_k)$. Assuming the state is Markovian and the observation conditionally white, the solution
of the problem is given by the well-known recursion for the {\it a posteriori } probability
\begin{eqnarray}
p(x_k|y_{1:k} )  = \frac{p(y_k|x_k)p(x_k|y_{1:k-1}) }{\int p(y_k|x_k)p(x_k|y_{1:k-1}) dx_{k} }
\end{eqnarray}
which is nothing but  Bayes law where the {\it prior} probability  is $p(x_k|y_{1:k-1})$. Therefore,  kernel Bayes rules
can realize this recursion in a RKHS.
Let $m_{z,k|l}$ be the embedding of $p(z_k|y_{1:l} )$ in  $ V_z$ where $z$ is either $x$ or $y$, and $ V_z$ is the RKHS associated to kernel $K_z(., .)$.

Embedding the previous recursion in a RKHS amounts to applying kernel Bayes rule (\ref{bayesrule:eq})
with {\it prior} probability $p(x_k|y_{1:k-1})$ and likelihood $p(y_k|x_k)$, to obtain the embedding 
$m_{x,k|k}$ of the {\it posterior} probability.

Firstly,  the  link between the embedding  $m_{x,k|k-1}$ of the {\it prior} probability,  and  $m_{x,k-1|k-1}$ is  obtained by applying (\ref{muydef:eq}), or
\begin{eqnarray}
m_{x,k|k-1}  = \Sigma_{x_kx_{k-1}} \Sigma^{-1}_{x_{k-1}x_{k-1}}m_{x,k-1|k-1}
\label{maj:eq}
\end{eqnarray}

Then the kernel sum rule for $p(y_k|  y_{1:k-1}) = \int p(y_k|x_k)p(x_k|y_{1:k-1}) dx_{k} $ and the 
kernel chain rule for $ p(y_k|x_k)p(x_k|y_{1:k-1})$ have to be used. The application of (\ref{chainrule:eq}) and (\ref{sumrule:eq})  will respectively 
give the operators $c_{x,k|k} $ and $c_{yy,k|k-1}$ needed for kernel Bayes rule (\ref{bayesrule:eq}), 
\begin{eqnarray}
m_{x,k|k} = c_{x,k|k} c_{yy,k|k-1}^{-1} K_y(.,y_k)
\label{bayesfiltering:eq}
\end{eqnarray}
The operator  $c_{yy,k|k-1} $    satisfies, according to  the sum rule  (\ref{sumrule:eq}) 
\begin{eqnarray}
c_{yy,k|k-1}  = \Sigma_{(y_ky_k) x_k} \Sigma^{-1}_{x_kx_k} m_{x,k|k-1}
\label{opesum:eq}
\end{eqnarray}
whereas the operator $c_{x,k|k}$ is provided by the chain rule (\ref{chainrule:eq}), or 
\begin{eqnarray}
c_{x,k|k}=\Sigma_{y_kx_k}\Sigma^{-1}_{x_kx_k}m_{x,k|k-1}
\label{opechain:eq}
\end{eqnarray}
These last four equations provide the embedding of the optimal filtering solution into the RKHS $ V_x$.

To obtain a matrix representation for all these rules,  $N+1$ samples of the
couple $(x_k,y_k)$ are supposed to be observed. At time $k-1$  the kernel conditional mean is given by
\begin{eqnarray}
m_{x,k-1|k-1}(.)  = \sum_{i=1}^N \alpha^{k-1}_i K_x(., x_i)  =\vk_X(.) \valpha^{k-1} 
\end{eqnarray}
and therefore, the matrix representation of (\ref{maj:eq}) is given by 
\begin{eqnarray}
m_{x,k|k-1}(.)  &=& \vk_{X+}(.) (\vK_x +\lambda I)^{-1} \vK_x \valpha^{k-1}
\end{eqnarray}
where  $ \vk_{X+}(.) = (K_x(.,x_2), \ldots, K_x(.,x_{N+1}) )$ and $\vK_x$ is the Gram matrix built on $x_1,\ldots,x_N$. 

Then, the operator $c_{x,k|k}$ given by equation (\ref{opechain:eq}) has the representation
\begin{eqnarray}
c_{x,k|k}&=& \vk_Y(.) (\vK_x +\lambda I)^{-1}  \big(  m_{x,k|k-1} (x_1) , \ldots m_{x,k|k-1} (x_N)  \big)^\top \nonumber\\
&=&  \vk_Y(.) (\vK_x +\lambda I)^{-1}  \vK_{X X+} (\vK_x +\lambda I)^{-1} \vK_x \valpha^{k-1} \nonumber \\
&=&\sum_{i=1}^N \vmu^k_i K_y(., y_i) \quad \mbox{ where } \\
\vmu^k&=&(\vK_x +\lambda I)^{-1}  \vK_{X X+} (\vK_x +\lambda I)^{-1} \vK_x \valpha^{k-1}
\end{eqnarray}
Likewise the operator $c_{yy,k|k-1}$ in (\ref{opesum:eq}) has the representation
\begin{eqnarray}
c_{yy,k|k-1}  = \sum_i \vmu^k_i K_y(., y_i)\otimes K_y(., y_i)
\end{eqnarray}
Finally,   $m_{x,k|k}(.)  = \vk_X(.) \valpha^{k-1} $  where parameter $\valpha^k$ reads
 \begin{eqnarray}
\valpha^k&= & \Lambda^k \vK_y \left(  (\vK_y \Lambda^k )^2 + \varepsilon I\right)^{-1}  \Lambda^k \vk_Y(y_k) 
\end{eqnarray}
where   $\Lambda^k = \diag(\vmu^k)$.

To synthesize: the kernel conditional mean is given by
\begin{eqnarray}
m_{x,k|k}(.)  = \sum_{i=1}^N \alpha^{k}_i K_x(., x_i)  =\vk_X(.) \valpha^{k} 
\end{eqnarray}
where the vector $\valpha^k$ satisfies the recursion
\begin{eqnarray}
\vmu^k&=&(\vK_x +\lambda I)^{-1}  \vK_{X X+} (\vK_x +\lambda I)^{-1} \vK_x \valpha^{k-1}\\
\Lambda^k &=& \diag(\vmu^k)\\
\valpha^k&= & \Lambda^k \vK_y \left(  (\vK_y \Lambda^k )^2 + \varepsilon I\right)^{-1}  \Lambda^k \vk_Y(y_k) 
\end{eqnarray}
The matrix $(\vK_x +\lambda I)^{-1}  \vK_{X X+} (\vK_x +\lambda I)^{-1} \vK_x $ can obviously be pre-computed.
With the notation taken, the first useful time for a real estimation is $k=N+2$, since the first $N+1$ dates are used for learning.
To initialize, $\widehat{\pi}(x_{N+1})=E[K_x(.,x_{N+1})]=\Sigma_{xy}\Sigma^{-1}_{yy}K(.,y_{N+1})$  can be used, and thus $\valpha^{N+1}=(\vK_y+\lambda I)^{-1}\vk_Y(y_{N+1})$.

The outcome of the algorithm is an estimation of the embedding of the {\it a posteriori }  measure. If  an estimate
of $E[f(x_k)|y_1,\ldots, y_k]$ where $f\in  V_x$ is seeked for,    the definition $E[f(x_k)|y_1,\ldots, y_k]=\big< f ,  m_{k|k}\big>$ is applied.
However, if $f$ does not belong to the RKHS,  this can not be applied.
A possibility is to find the pre-image $x^k$ whose image $K_x(.,x^k)$ is the closest to the embedding of the {\it posterior} probability.
For radial kernel $K_x(.,x)=f(\|.-x\|^2)$ this can be solved efficiently if closeness is measured using the RKHS norm \cite{SchoS02}. Indeed,
searching for $\min_x \|  K_x(.,x) - \sum_i K_x(.,x_i)\alpha^t_i\|$ leads to the fixed point condition
$x=\sum_i x_i f'(\|x-x_i\|^2) \alpha^t_i / \sum_i f'(\|x-x_i\|^2) \alpha^t_i $. A solution can be obtained sequentially as
\begin{eqnarray}
x^t_n = \frac{\sum_i x_i f'(\|x^t_{n-1}-x_i\|^2) \alpha^t_i }{\sum_i f'(\|x^t_{n-1}-x_i\|^2) \alpha^t_i }
\end{eqnarray}
No convergence guarantees exist for this procedure, and {\it ad-hoc  } stategies are usually called for, such as running the algorithms with several different initial conditions, \ldots \\

\noindent{\bf Illustration in prediction.} An example of kernel Bayes filtering on a prediction problem is presented. In the example taken from  \cite{RichBH09}, a signal $z_n$ is generated using the following nonlinear autoregressive model
\begin{eqnarray}
%x_n=(0.8-0.5e^{-x_{n-1}^2})x_{n-1} -(0.3+0.9e^{-x_{n-1}^2})x_{n-2} +0.1 \sin(\pi x_{n-1}) + \varepsilon_n
z_n=\frac{(8-5e^{-z_{n-1}^2})z_{n-1}}{10} -\frac{(3+9e^{-z_{n-1}^2})z_{n-2} }{10}+\frac{\sin(\pi z_{n-1})}{10} + \varepsilon_n
\end{eqnarray}
where $\varepsilon_n$ is chosen from an i.i.d. sequence of zero mean Gaussian random variables of standard deviation  set to 0.1 in the illustration. A plot of the phase space $(z_{n-1},z_{n}) $ is depicted in the left plot in  figure \ref{kbrillustr:fig}. The learning set is composed of the first 512 samples of the signal $z_n$. The data set is composed  by the state $x_n=(z_n,z_{n+1})^\top$ and the observation is simply $y_n=z_n$. The state $x_n$ is estimated and 
$\widehat{z_{n+1}}$ is defined as the second coordinate of the estimated state. The parameters chosen for 
the simulation are not optimized at all. The kernels are Gaussian kernels with variance parameters set to 0.5. 
The regularisation parameters were set to $10^{-4}$.
\begin{figure}[htbp]
\centering
\includegraphics{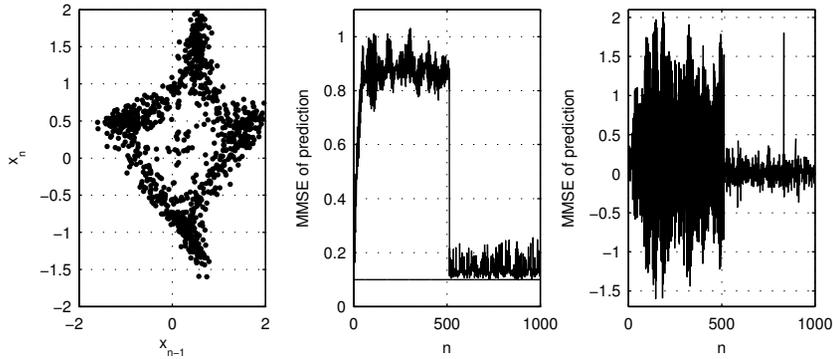}
\caption{Left plot: phase diagram of the time series used to illustrate kernel Bayes filtering. Middle plot: Square root of the mean square error of the predictor obtained bby averagin gover 200 snapshots. The first 512 sample represent the power of the signal to be estimated since this interval correspond to learning, and no estimationis performed: the estimator is set to zero during this period of time.  The horizontal line marks the amplitude 0.1 which corresponds to the standard deviation of the dynamical noise, and is therefore the optimal power to reach. Right plot: a particular snaphot. The outlier corresponds to a bad estimate by the pre-image algortihm. }
\label{kbrillustr:fig}
\end{figure}

As shown in the right plot of figure \ref{kbrillustr:fig}, the error $e_{n+1|n}=
z_{n+1}-\widehat{z_{n+1}}$ is almost equal to to the dynamical noise $\varepsilon_{n+1}$. 
The standard deviation of $e_{n+1|n}$ is estimated by averaging the above procedure over 100 snaphots; its value after convergence is  0.13, close to 0.1, the standard deviation of the dynamical noise. The square root of the mean square error of prediction is depicted in the middle plot. As recalled earlier learning is performed on the first 512 samples and estimation begins at time 514.  Therefore, the first 512 samples presented are respectively the signal $z_{n+1}$ in the right plot, and its power in the middle plot. 

Finally, the presence of an outlier at time around 850 in the right plot is interesting. It comes from a bad estimation of the pre-image. The algorithm to go back from the RKHS to the input space is very important to produce a good estimation. In the implementation used for this illustration, the convergence of the pre-image algorithm is controlled. If divergence occurs, or if no convergence (oscillations) occurs,  the algorithm is re-initialized and run again. The initial condition is randomly chosen. The outlier in the figure corresponds to a point which is just below the threshold used to decide of divergence. 

This example  illustrates the effectiveness of kernel Bayes rules here applied in prediction. Note that as developped,
kernel Bayes filtering is not  an on-line procedure {\it per se}, since learning is done on a fixed set of data. Thus the approach here should be compared with gaussian processes regression or kernel regression. All these methods will have the same level of quality. On-line learning with kernel Bayes rule remains to be developped, especially for cases in which the {\it a priori} and/or the observation are nonstationary. For these cases, on-line filtering approaches have been developed in a non Bayesian framework, as evoked in the following.

\section{On-line filtering using RKHS}

Filtering may be seen is an estimation problem. Before discussing on-line filtering in RKHS it is worth presenting some consideration of estimation in RKHS.

\noindent {\bf On estimation in  RKHS.} Consider for example the following problem which is at the root of filtering. Let $y$ be a real square integrable random variable,  and consider using this random variable to estimate  another random variable, say $X$ with values in $\X$. Minimizing the mean square error 
$E[(Y-h(X))^2]$ is  often used. With the only constraint that $h$ is measurable, the optimal function $h$ is given by $h(X)=E[Y|X]$, the conditional mean. 

Usually the conditional mean is very difficult to evaluate so that the function $h$ is searched for in some more constrained classes of functions. If $h$ is restricted to be a linear functional, the solution is Wiener filter. For example, if $\X=\R^n$,  the vector $\vh$  which minimizes $E[(Y-\vh^\top X)^2]$ is seeked for. The solution satisfies the well known equation $E[YX]=E[XX^\top] \vh$. Here, the problem is formalized and solved when functions $h$ is searched for in  a RKHS whose kernel is a kernel on $\X$. This corresponds to the important situation where the optimal filter $h$ is as a  linear functional of the  transformed observation (by the kernel.) 

Thus let $K$ be a kernel on $\X$ and $\cal H$ its associated reproducing kernel Hilbert space of functions from $\X$ to $\R$. 
The best estimator satisfies
\begin{eqnarray}
f_0= \arg \min_{f\in \cal H} E\left[ \left(Y - f(X) \right)^2  \right]
\end{eqnarray}
Let $R(f) =E\left[ \left( Y - f(X) \right)^2  \right]$. To find the minimun, a necessary condition is to set to zero 
the Gateaux derivative in every possible direction  \cite{DebnM05}, {\it i.e.} since $R$ is a functional, to solve
\begin{eqnarray}
\frac{dR(f+\varepsilon \varphi)}{d\varepsilon}\Big|_{\varepsilon=0} =0 , \quad \forall \varphi \in  V
\end{eqnarray}
This is equivalent to 
\begin{eqnarray}
 \Sigma_{YX}   \varphi &=&\big<  \varphi  ,   \Sigma_{XX} f\big>   , \quad \forall \varphi \in  V   
\end{eqnarray}
where 
\begin{eqnarray}
 \Sigma_{XX} :   V &\longrightarrow &  V \nonumber \\
                f  &\longmapsto& \Sigma_{XX}f:= E[ \big< f ,  K(., X) \big> K(., X) ]  \\
                \Sigma_{YX} :   V &\longrightarrow & {\R} \nonumber \\
                f  &\longmapsto& \Sigma_{YX}f:= E[ \big< f ,  K(., X) \big> Y ]   
\end{eqnarray}
are the covariance and the correlation operators.  The correlation operator is in this particular case a bounded linear functional and thanks to the Riesz representation theorem can be written as 
$\Sigma_{YX}f:=\big< f ,  E[YK(., X)] \big> $. Then the optimal function is found if we solve
\begin{eqnarray}
 \big<\varphi   , E[YK(., X)]  \big>   =\big<  \varphi  ,   \Sigma_{XX} f\big>  \quad \forall \varphi \in  V
\end{eqnarray}
Since this is valid for any $\varphi$ the solution is given by the function that solves 
$ \Sigma_{XX} f =E[YK(., X)]$.  Practically of course, the covariance and correlation operators have to be estimated from a finite set of data. Let $(y_i,x_i)$ these data. From the representer theorem we know that the functions are to be searched for in the subspace of $\cal H$ generated by the $K(.,x_i)$. Then the preceding equation has an empirical counterpart which reads 
\begin{eqnarray}
 \big<\sum_i \alpha_i K(., x_i) , \sum_j y_j K(., x_j)]  \big> &  =&\valpha^T K_X K_X \vbeta, \quad   \forall \valpha \\
 \valpha^T K_X \vy &  =&\valpha^T K_X K_X \beta  , \quad \forall \valpha
 \end{eqnarray}
where $K_X$ is the Gram matrix, and $f(.) := \sum_j \beta_j K(., x_j)$. This is  the same solution as the regression case in machine learning, as developed in the following paragraph.\\

\noindent {\bf Filtering in a RKHS.}
Filtering in a RKHS is a particular problem of regression. The data $(y_i,x_i)$ considered above can be seen as observations
that can be explained using a regression $Y=f(X)$ where $f$ is searched for in a RKHS. From a signal processing perspective, 
$y_i$ may represent the value of signal $y$ at time $i$, whereas vector $x_i$ contains past samples of another signal. $y=f(X)$ is thus considered as a black box model, and the aim of the optimal estimation is to identify $f(.)$.

Thus, let $y_n$ and $z_n$ be  two real valued random signals, and suppose we look for a  function   which minimizes the power of the error $y_n-f(x_n)$, where $x_n=(z_n, z_{n-1},\ldots,z_{n-d+1})^\top \in \R^d$. $\R^d$ is embedded into a RKHS using a kernel $K$ on $\R^d$, and the optimizing function is seeked for in the RKHS.  The optimization is realized on a set of $N$ observed data $(y_i,x_i)$.
Furthermore, since most of the interesting RKHS are of either high or infinite dimension, overfitting is likely to occur,
and some regularization must be included. The norm in the RKHS is used to constrain the function.
Thus, the optimization problem can be written as
\begin{eqnarray}
f_0(.)= \arg \min_{f\in  V} \sum_i \Big| y_i-f(x_i)\Big|^2 +\lambda \Big\|    f \Big\|^2_ V
\end{eqnarray}
The representer theorem \cite{KimeW70,SchoS02} states that the optimizing function $f_0$ belongs to the the subspace of
$ V$ generated by $\{K(., x_i)\}, i=1,\ldots,N$. If $f_0=\sum_i \alpha_{0,i} K(.,x_i)$, the  preceding program
is equivalent to 
\begin{eqnarray}
\valpha_0= \arg \min_{\valpha \in \R^N}  (\vy_i-\vK_x \valpha)^\top(\vy_i-\vK_x \valpha) +\lambda \valpha^\top \vK_x \valpha
\end{eqnarray}
the solution of which is given by 
\begin{eqnarray}
\valpha_0 = (\vK_x + \lambda \vI)^{-1}\vy
\end{eqnarray}
Then when a new datum $x_k$ is observed, $k>N$, the corresponding estimate $y_k$ writes
\begin{eqnarray}
y_k = \sum_{i=1}^N \alpha_{0,i} K(x_k, x_i) = \vK_x(x_k)^\top ( \vK_x + \lambda \vI)^{-1}\vy
\end{eqnarray}
where $ \vK_x(x_k)^\top = \big( K(x_k, x_1), \ldots,K(x_k, x_N) \big)$. This regularized solution is also known as the ridge regression. 
The result may appear strange since in the linear case it reads $y_k = x_k^\top  \vX ( \vX \vX^\top +\lambda \vI)^{-1} \vy$
where $\vX$ is the design matrix $(x_1,\ldots,x_N)^\top$. But the well known trick $ \vX ( \vX \vX^\top +\lambda \vI)^{-1}=
( \vX^\top \vX +\lambda \vI)^{-1}\vX^T$ allows to recover the usual Wiener filtering form
$y_k = x_k^\top  ( \vX^\top \vX +\lambda \vI)^{-1}  \vX^\top\vy$. \\

\noindent {\bf On-line filtering.}
As an application  the problem of on-line filtering is considered. The aim here is to use the previous filter in situations where data arrives in streams or on-line, and can not be processed in batch mode. Thus a recursive structure is needed to refresh the filter when a new datum is acquired, so as to produce a new estimate without the need of re-calculating. This problem has seen some important developments in the last decade, with the 
presentation of the kernel Recursive Least Square (kRLS) algorithm \cite{EngeMM04}, followed by many works such as those reviewed recently in  \cite{SlavBT14}.

The biggest difficulty for on-line kernel algorithms lies in the fact that they inherently manipulates Gram matrices whose sizes grow linearly with the size of the data. Thus, on-line kernel algorithms necessarily include a sparsification procedure which allow to approximate  the Gram matrices needed as well as possible. For a recent thorough review of on-line kernel algorithms for signal processing problems, we refer to \cite{SlavBT14}. Here,
the kRLS in which the sparsification is the Approximate Linear Dependence criterion is presented. The kLMS which uses the coherence criterion as sparisfication criterion is proposed as a simplified alternative. We already used the coherence criterion in \S\ref{testindep:ssec}.

The sparsification procedure is based on a recursive building of a dictionary using the ALD. Initializing the dictionary ${\cal D}_1={1}$ with the first datum acquired $x_1$, the rule to create the dictionary is 
\begin{eqnarray}
{\cal D}_n \! = \! \left\{
\begin{array}{ll}
 \!  \! {\cal D}_{n-1} \cup \{n\}  & \!  \!  \!\!\!\mbox{ if } K(.,x_n) \mbox{ is not ALD of }   \{ K(., x_\alpha) \}_{\alpha \in {\cal D}_{n-1}  }\\
 \!  \! {\cal D}_{n-1} &  \!  \!  \!\!\!\mbox{ otherwise}
\end{array}
\right.
\end{eqnarray}
Approximate linear dependence is measured according to the mean square error in the prediction of 
$K(.,x_n)$ from the members $ K(., x_\alpha)$ of the dictionary at time $n-1$. If this error is greater than a user specified threshold,
the ALD is rejected, and $K(.,x_n)$ carries enough new information to be incorporated into the dictionary. 

Thus at each time $n$, the approximation is $\hat{k}(.,x_n)= \sum_{i=1}^{d_{n-1}}a_{n,i} K(.,x_{\alpha_i})$, where $d_n=|{\cal D}_n|$, and
where the coefficients are obtained as in the regression framework presented above, that is 
\begin{eqnarray}
\va_n &=& \widehat{\vK}_{n-1}^{-1} \vk_{n-1}(x_n)
\end{eqnarray}
and the minimum mean square error reached is given by  
\begin{eqnarray}
e_n &=& K(x_n,x_n) -  \vk_{n-1}(x_n)^\top \va_n \nonumber \\
&=&K(x_n,x_n) -  \vk_{n-1}(x_n)^\top\widehat{\vK}_{n-1}^{-1} \vk_{n-1}(x_n)
\end{eqnarray}
In these equations, $(\widehat{\vK}_{n-1})_{ij}=K(x_{\alpha_i},x_{\alpha_j})$ is the Gram matrix evaluated on the dictionary at time $n-1$ and
 $\vk_{n-1}(x_n)^\top=\big(K(x_{n},x_{\alpha_1}), \ldots,K(x_{n},x_{\alpha_{d_{n-1}}})\big)$.
 The test for ALD consists in comparing $e_n$ to a given value $e_0$. If the ALD is accepted or $e_n<e_0$, the dictionary is not modified
 and ${\cal D}_n={\cal D}_{n-1}$.
 If the ALD is rejected,  ${\cal D}_n={\cal D}_{n-1}\cup \{n\}$ and obviously $\hat{k}(.,x_n)=K(.,x_n)$.

Interestingly, this allows to obtain an approximation of the Gram matrix calculated over all the observation measured up to time $n$.
Indeed, since for all $n$, $\hat{k}(.,x_n)= \sum_{i=1}^{d_{n-1}}a_{n,i} K(.,x_{\alpha_i})$, the $(l, m)$ entry of the full Gram matrix reads
approximately
\begin{eqnarray}
K(x_l,x_m) &\approx&\big<  \hat{k}(.,x_l),  \hat{k}(.,x_m)\big>\\
&=& \sum_{i,j=1}^{d_{m-1}}a_{l,i}a_{m,j} K(x_{\alpha_i},x_{\alpha_j})
\end{eqnarray}
This supposes without loss of generality that $l\leq m$, and implicitely set $a_{l,i}=0$ as soon as $i>d_{l-1}$.
This allows to store the coefficient $a_{l,i}$ into a matrix  $\vA_l$ of  appropriate dimension and to write
\begin{eqnarray}
\vK_{n} \approx \vA_n \widehat{\vK}_{n} \vA_n^\top
\end{eqnarray}
The matrix $\vA_n$ is updated as $(\vA_{n-1}^\top \va_n^\top)^\top$ if ALD is accepted, and as 
\begin{eqnarray}
\left(\begin{array}{c|c}\vA_{n-1}& 0 \\\hline 0 & 1\end{array}\right)
\label{upAnon:eq}
\end{eqnarray}
if  ALD is rejected.

For  the problem of on-line regression, the following problem has to be solved
\begin{eqnarray}
\valpha_n= \arg \min_{\valpha }  (\vy_n-\vK_n \valpha)^\top(\vy_n-\vK_n \valpha) 
\end{eqnarray}
a program replaced by 
\begin{eqnarray}
\valpha_n= \arg \min_{\valpha }  (\vy_n- \vA_n \widehat{\vK}_{n} \valpha)^\top(\vy_n- \vA_n \widehat{\vK}_{n} \valpha) 
\end{eqnarray}
where  the substitution $\valpha \leftrightarrow \vA_n^\top \valpha $ has been made.
The optimal parameter $\valpha_n$ at time $n$ is thus given by the pseudo-inverse of $\vA_n \widehat{\vK}_{n}$ applied to $\vy_n$, or
after elementary manipulations
\begin{eqnarray}
\valpha_n=  \widehat{\vK}_{n}^{-1} (\vA_n^\top \vA_n)^{-1} \vA_n^\top \vy_n
\end{eqnarray}

This form allows an easy transformation into a recursive form. Basically, at time $n$, if ALD is accepted and the dictionary does not change ${\cal D}_n={\cal D}_{n-1} $, then $ \widehat{\vK}_n= \widehat{\vK}_{n-1}$, 
$\vA_n$ is updated as $(\vA_{n-1}^\top \va_n^\top)^\top$ and $\vP_n=(\vA_n^\top \vA_n)^{-1} $ is updated as usual as
\begin{eqnarray}
\vP_n=\vP_{n-1}-\frac{\vP_{n-1} \va_n\va_n^\top\vP_{n-1}}{1+\va_n^\top\vP_{n-1}\va_n}
\label{upPoui:eq}
\end{eqnarray}

However, if at time $n$ ALD is rejected, the dictionary is increased  ${\cal D}_n={\cal D}_{n-1} \cup \{n\}$, 
$\vK_n^{-1}$ is updated as 
\begin{eqnarray}
 \widehat{\vK}_n^{-1} &= &\left(\begin{array}{c|c} \widehat{\vK}_{n-1} & \vk_{n-1}(x_n) \\\hline \vk_{n-1}(x_n)^\top & K(x_n,x_n)\end{array}\right)^{-1} \nonumber \\
 &=&\frac{1}{e_n}
\left(\begin{array}{c|c} e_n  \widehat{\vK}_{n-1}^{-1}+\va_n \va_n^\top  & -\va_n \\\hline - \va_n^\top& 1 \end{array}\right)
\label{upKnon:eq}
\end{eqnarray}
$\vA_{n-1}$ updated according to (\ref{upAnon:eq}), and  $\vP_n$ follows 
\begin{eqnarray}
\vP_n= \left(\begin{array}{c|c}\vP_{n-1}& 0 \\\hline 0 & 1\end{array}\right)
\label{upPnon:eq}
\end{eqnarray}

Using some algebra then leads to the following kRLS algorithm. Initialize 
${\cal D}_1=\{1\}, \valpha_1=y_1/K(x_1,x_1)$ and then for $n\geq2$
\begin{enumerate}
\item $\va_n=\widehat{\vK}_{n-1}^{-1} \vk_{n-1}(x_n) $ and $e_n = K(x_n,x_n) -  \vk_{n-1}(x_n)^\top \va_n $
\item {\bf if} $e_n\leq e_0$, ${\cal D}_n={\cal D}_{n-1}$, update $\vP_n$ according to (\ref{upPoui:eq}),  $ \widehat{\vK}_n^{-1}=\widehat{\vK}_{n-1}^{-1}$ , $\vA_n=(\vA_{n-1}^\top \va_n^\top)^\top$
and set 
\begin{eqnarray}
\valpha_n = \valpha_{n-1}+ \frac{\widehat{\vK}_{n-1}^{-1}\vP_{n-1}\va_n}{1+\va_n^\top\vP_{n-1}\va_n}\big(y_n- \vk_{n-1}(x_n)^\top\valpha_{n-1}\big) 
\end{eqnarray}

{\bf else} ${\cal D}_n={\cal D}_{n-1}\cup \{n\}$, update $\vP_n$ according to (\ref{upPnon:eq}), $ \widehat{\vK}_n^{-1}$ according to (\ref{upKnon:eq}),  $\vA_{n}$  according to (\ref{upAnon:eq}) and set
\begin{eqnarray}
\valpha_n = \left\{
\begin{array}{l}
\valpha_{n-1}- \frac{\va_n}{e_n}\big(y_n- \vk_{n-1}(x_n)^\top\valpha_{n-1}\big) \\
\frac{1}{e_n}\big(y_n- \vk_{n-1}(x_n)^\top\valpha_{n-1}\big) 
\end{array}
\right.
\end{eqnarray}

\end{enumerate}

The complexity of the kRLS is dominated by inverting the matrices. If a lower complexity is required, simpler approaches may be used, such as the kernel Least Mean Square (kLMS) algorithm. Its expression is just written down, referring to \cite{LiuPP08,SlavBT14} for more details and variants. It uses a dictionary as well. However, since the ALD requires the propagation of the inverse of $\widehat{\vK_n}$, the criterion can not be used here because a motivation for the kLMS is to reduce complexity. Thus another criterion must be used. The simpler one up to today is the coherence criterion proposed in this context in \cite{RichBH09}, and that we detailed previously. If the new datum is not coherent enough with the dictionary it is included. Coherence is measured in the RKHS and therefore simply uses the kernel evaluated at the new datum and at the members of the dictionary. 
 The normalized kLMS  recursively and adaptively computes the coefficient vector according to the following steps:
 
 Initialize 
${\cal D}_1=\{1\}, \valpha_1=y_1/K(x_1,x_1)$ and then for $n\geq2$
\begin{enumerate}
\item $e_n =\max_{\alpha\in {\cal D}_{n-1} }    \big| K(x_n, x_\alpha)    \big| $

\item {\bf if} $e_n\geq e_0$, ${\cal D}_n={\cal D}_{n-1}$,  set 
\begin{eqnarray}
\valpha_n = \valpha_{n-1}+ \frac{\lambda \big(y_n- \vk_{n-1}(x_n)^\top\valpha_{n-1}\big) 
}{\varepsilon + \vk_{n-1}(x_n)^\top \vk_{n-1}(x_n) }  \vk_{n-1}(x_n)
\end{eqnarray}

{\bf else} ${\cal D}_n={\cal D}_{n-1}\cup \{n\}$, set
\begin{eqnarray}
\valpha_n = \left(
\begin{array}{l}
\!\! \! \valpha_{n-1}\!\!\!   \\\!\!\!  0\!\! \! 
\end{array}
\right)
+ \frac{\lambda \big(y_n- \vk_{n-1}(x_n)^\top\valpha_{n-1}\big) 
}{\varepsilon + \vk_{n}(x_n)^\top    \vk_{n}(x_n)} \left(
\begin{array}{l}
\!\!\! \vk_{n-1}(x_n) \!\! \! \\  \!\! \! K(x_n,x_n)\!\! \! 
\end{array}
\right)
\end{eqnarray}
where recall that $\vk_{n}(x_n)^\top=(K(x_{n},x_{\alpha_1}), \ldots,K(x_{n},x_{\alpha_{d_{n}}})$
\end{enumerate}
$\lambda $ is the step size parameter that controls as usual the variance-bias trade-off. A study of convergence of this algorithm as been proposed recently in the Gaussian case for Gaussian kernels \cite{ParrBRT12}.\\

\noindent {\bf An illustration.} The two algorithms are illustrated on the example developped for illustrating kernel Bayes filtering. In figure \ref{kRLSkLMSillustr:fig}  the left plot depicts the square root of the mean square error as a function of time. The kernel used is the Gaussian 
$\exp(-\| x-y\|^2/\sigma^2)$.
Parameters chosen are $\sigma^2=1/3.73$, $\lambda=0.09$, $\varepsilon=0.03$. These values were used in \cite{RichBH09} and are taken here for the sake of simplicity. 
 For the left plot, $e_0$ has been set to 0.55 for 
the kRLS and the kLMS as well. This choice was made since it ensures the dictionaries are of the same size on average for the two algorithms (27 for this value of $e_0$). The convergence curves were obtained by averaging 500 snapshots.
\begin{figure}[htbp]
\centering
\includegraphics{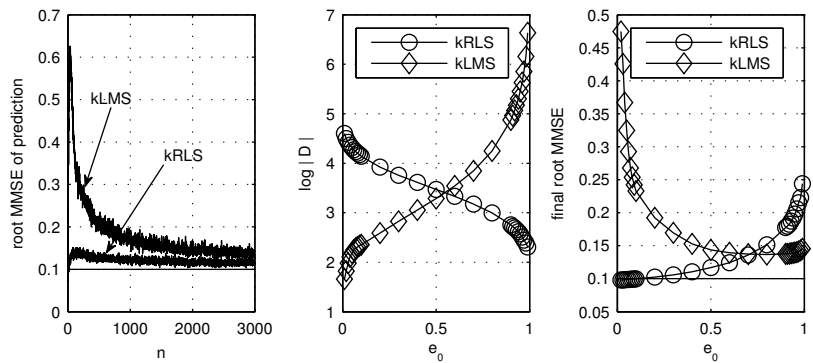}
\caption{Left plot:  Square root of the mean square error obtained by the kLMS and kRLS for the prediction problem studied in the text. The horizontal line marks the amplitude 0.1 which corresponds to the standard deviation of the dynamical noise, and is therefore the optimal power to reach. Middle plot:  Size of the dictionary at convergence as a function of parameter $e_0$.  Right plot: Square root of the mean square error obtained by the kLMS and kRLS at convergence as a function of parameter $e_0$.  }
\label{kRLSkLMSillustr:fig}
\end{figure}
 The middle plot displays the logarithm of the size of the dictionary after convergence as a function of $e_0$. The same parameter is used here but has a different meaning for the two algorithms. For the kLMS, the larger the parameter, the easier the dictionary is increased. For the kRLS, this goes the other way around. $e_0$ is varied from 0.01 to 1 in a non uniform way. The right plot shows the asymptotic mean square error (evaluated by averaging the last 500 samples of the convergence curves, these curves being obtained by averaging 500 snapshots). Interestingly with these two plots is the fact that the loss in MMSE is small for a high compression. For example, in the kRLS, the optimal performance of 0.1 is nearly obtained as soon as $e_0=0.1$ , for which the size of the dictionary is only 55!
 Likewise,  the best performance of the kLMS (about 0.13) is obtained for $e_0=0.7$ for which the size of the dictionary is 47.  This simple example show the effectiveness of the algorithms coupled with simple yet powerful sparsification techniques. These sparsification procedures open  up the use of kernel methods to very large data sets as well as on-line efficient algorithms.

\backmatter

\bibliographystyle{abbrv}
\bibliography{rkhs,rkhs2}

\begin{thebibliography}{10}

\bibitem{Ambl13}
P.~O. Amblard.
\newblock Versions r{\'e}cursives et adaptatives d'une mesure
  d'ind{\'e}pendance.
\newblock In {\em GRETSI 2013, Brest, France}, 2013.

\bibitem{AmblM13}
P.~O. Amblard and J.~H. Manton.
\newblock On data sparsification and a recursive algorithm for estimating a
  kernel-based measure of independence.
\newblock In {\em proc. ICASSP, Vancouver, Canada}, 2013.

\bibitem{Aronszajn:1950tq}
N.~Aronszajn.
\newblock {Theory of Reproducing Kernels}.
\newblock {\em Transactions of the American Mathematical Society}, 68:337--404,
  1950.

\bibitem{Bach13}
F.~R. Bach.
\newblock Sharp analysis of low-rank kernel matrix approximations.
\newblock {\em JMLR: Workshop and Conference Proceedings}, 30:1--25, 2013.

\bibitem{BachJ02}
F.~R. Bach and M.~I. Jordan.
\newblock Kernel independent component analysis.
\newblock {\em Journal of Machine Learning Research}, 3(1--48), 2002.

\bibitem{Bake73}
C.~R. Baker.
\newblock Joint measures and cross-covariance operators.
\newblock {\em Trans. Am. Math. Soc.}, 186:273--289, 1973.

\bibitem{Barton:1990ir}
R.~J. Barton and H.~V. Poor.
\newblock {An RKHS approach to robust $L^2$ estimation and signal detection}.
\newblock {\em Information Theory, IEEE Transactions on}, 36(3):485--501, 1990.

\bibitem{bergman1970kernel}
S.~Bergman.
\newblock {\em The kernel function and conformal mapping}, volume~5.
\newblock American Mathematical Soc., 1970.

\bibitem{BerlTA04}
A.~Berlinet and C.~Thomas-Agnan.
\newblock {\em Reproducing kernel {Hilbert} spaces in probability and
  statistics}.
\newblock Kluwer Academic Publishers, 2004.

\bibitem{Berlinet:2012uv}
A.~Berlinet and C.~Thomas-Agnan.
\newblock {\em {Reproducing Kernel Hilbert Spaces in Probability and
  Statistics}}.
\newblock Springer, Dec. 2012.

\bibitem{Brown:2004gy}
R.~F. Brown.
\newblock {\em {A topological introduction to nonlinear analysis}}.
\newblock Birkh\"auser Boston, Inc., Boston, MA, second edition, 2004.

\bibitem{Curtain:1995hv}
R.~F. Curtain and H.~Zwart.
\newblock {\em {An introduction to infinite-dimensional linear systems
  theory}}, volume~21 of {\em Texts in Applied Mathematics}.
\newblock Springer-Verlag, New York, 1995.

\bibitem{da2006introduction}
G.~Da~Prato.
\newblock {\em An Introduction to Infinite-Dimensional Analysis}.
\newblock Universitext. Springer, 2006.

\bibitem{DebnM05}
L.~Debnath and P.~Mikusi{\'n}ski.
\newblock {\em Introduction to Hilbert spaces with applications, 3rd ed.}
\newblock Elsevier Academic Press, 2005.

\bibitem{Driscoll:1973un}
M.~F. Driscoll.
\newblock {The reproducing kernel Hilbert space structure of the sample paths
  of a Gaussian process}.
\newblock {\em Z. Wahrscheinlichkeitstheorie und Verw. Gebiete}, 26:309--316,
  1973.

\bibitem{Duttweiler:1973bq}
D.~L. Duttweiler and T.~Kailath.
\newblock {RKHS approach to detection and estimation problems. IV. Non-Gaussian
  detection}.
\newblock {\em Information Theory, IEEE Transactions on}, IT-19(1):19--28,
  1973.

\bibitem{Duttweiler:1973io}
D.~L. Duttweiler and T.~Kailath.
\newblock {RKHS approach to detection and estimation problems. V. Parameter
  estimation}.
\newblock {\em Information Theory, IEEE Transactions on}, IT-19(1):29--36,
  1973.

\bibitem{elliott1995hidden}
R.~Elliott, L.~Aggoun, and J.~Moore.
\newblock {\em Hidden Markov Models: Estimation and Control}.
\newblock Applications of mathematics. Springer, 1995.

\bibitem{EngeMM04}
Y.~Engel, S.~Mannor, and R.~Meir.
\newblock The kernel recursive least-squares algorithm.
\newblock {\em IEEE Trans. on Signal Processing}, 52(8):2275--2284, 2004.

\bibitem{FineS01}
S.~Fine and K.~Scheinberg.
\newblock Efficient svm training using low-rank kernel representations.
\newblock {\em Journal of Machine Learning Research}, 2:243--264, 2001.

\bibitem{Fort95}
R.~Fortet.
\newblock {\em Vecteurs, fonctions et distributions al{\'e}atoires dans les
  espaces de Hilbert (Random Vectors, functions and distributions in Hilbert
  spaces)}.
\newblock (in French), Herm{\`e}s, 1995.

\bibitem{FukuBJ04}
K.~Fukumizu, F.~R. Bach, and M.~I. Jordan.
\newblock Dimensionality reduction for supervised learning with reproducing
  kernel hilbert spaces.
\newblock {\em Journal of Machine Learning Research}, 4:73--99, 2004.

\bibitem{FukuBJ09}
K.~Fukumizu, F.~R. Bach, and M.~I. Jordan.
\newblock Kernel dimension reduction in regression.
\newblock {\em The Annals of Statistics}, 37(4):1871--1905, 2009.

\bibitem{FukuGSS07}
K.~Fukumizu, A.~Gretton, X.~Sun, and B.~Scholkopf.
\newblock Kernel measures of conditional dependence.
\newblock In {\em NIPS}, 2007.

\bibitem{FukuSG13}
K.~Fukumizu, L.~Song, and A.~Gretton.
\newblock Kernel bayes'rules: Bayesian inference with positive definite
  kernels.
\newblock {\em Journal of Machine Learning Research}, 14:3753--3783, 2013.

\bibitem{GretBRSS12}
A.~Gretton, K.~M. Borgwardt, M.~J. Rasch, B.~Sch{\"o}lkopf, and A.~Smola.
\newblock A kernel two sample test.
\newblock {\em Journal of Machine Learning Research}, 13:723--773, 2012.

\bibitem{GretHSBS05}
A.~Gretton, R.~Herbrich, A.~Smola, O.~Bousquet, and B.~Sch{\"o}lkopf.
\newblock Kernel methods for measuring independence.
\newblock {\em Journal of Machine Learning Research}, 6:2075--2129, 2005.

\bibitem{han2007frames}
D.~Han.
\newblock {\em Frames for Undergraduates}.
\newblock Student mathematical library. American Mathematical Society, 2007.

\bibitem{Hida:1967ux}
T.~Hida and N.~Ikeda.
\newblock {Analysis on Hilbert space with reproducing kernel arising from
  multiple Wiener integral}.
\newblock In {\em Proc. Fifth Berkeley Sympos. Math. Statist. and Probability
  (Berkeley, Calif., 1965/66). Vol. II: Contributions to Probability Theory,
  Part 1}, pages 117--143. Univ. California Press, Berkeley, Calif., 1967.

\bibitem{Hille:1972wi}
E.~Hille.
\newblock {Introduction to general theory of reproducing kernels}.
\newblock {\em The Rocky Mountain Journal of Mathematics}, 2(3):321--368, 1972.

\bibitem{HoffmannJorgensen:1976fd}
J.~Hoffmann-J{\o}rgensen and G.~Pisier.
\newblock {The Law of Large Numbers and the Central Limit Theorem in Banach
  Spaces}.
\newblock {\em The Annals of Probability}, 4(4):587--599, Aug. 1976.

\bibitem{Kailath:1971hk}
T.~Kailath.
\newblock {RKHS approach to detection and estimation problems. I. Deterministic
  signals in Gaussian noise}.
\newblock {\em Information Theory, IEEE Transactions on}, IT-17(5):530--549,
  1971.

\bibitem{Kailath:1972jo}
T.~Kailath and D.~Duttweiler.
\newblock {An RKHS approach to detection and estimation problems. III.
  Generalized innovations representations and a likelihood-ratio formula}.
\newblock {\em Information Theory, IEEE Transactions on}, IT-18(6):730--745,
  1972.

\bibitem{Kailath:1972bv}
T.~Kailath, R.~Geesey, and H.~Weinert.
\newblock {Some relations among RKHS norms, Fredholm equations, and innovations
  representations}.
\newblock {\em Information Theory, IEEE Transactions on}, 18(3):341--348, May
  1972.

\bibitem{Kailath:1975bj}
T.~Kailath and H.~L. Weinert.
\newblock {An RKHS approach to detection and estimation problems. II. Gaussian
  signal detection}.
\newblock {\em Information Theory, IEEE Transactions on}, IT-21(1):15--23,
  1975.

\bibitem{Kallianpur:1970wu}
G.~Kallianpur.
\newblock {The role of reproducing kernel Hilbert spaces in the study of
  Gaussian processes}.
\newblock In {\em Advances in Probability and Related Topics, Vol. 2}, pages
  49--83. Dekker, New York, 1970.

\bibitem{KimeW70}
G.~S. Kimeldorf and G.~Wahba.
\newblock A correspondence between bayesian estimation on stochastic processes
  and smoothing by splines.
\newblock {\em The Annals of Mathematical Statistics}, 41(2):495---502., 1970.

\bibitem{Larkin:1970do}
F.~M. Larkin.
\newblock {Optimal Approximation in Hilbert Spaces with Reproducing Kernel
  Functions}.
\newblock {\em Mathematics of Computation}, 24(112):911--921, Oct. 1970.

\bibitem{Laur96}
S.~Lauritzen.
\newblock {\em Graphical models}.
\newblock Oxford University Press, 1996.

\bibitem{LiuPP08}
W.~Liu, P.~Pokharel, and J.~Principe.
\newblock The kernel least mean squares algorithm.
\newblock {\em IEEE Trans. on Signal Processing}, 56(2):543--554, 2008.

\bibitem{Luenberger:1969wv}
D.~G. Luenberger.
\newblock {\em {Optimization by vector space methods}}.
\newblock John Wiley {\&} Sons Inc., New York, 1969.

\bibitem{Lukic:2001un}
M.~Lukic and J.~Beder.
\newblock {Stochastic processes with sample paths in reproducing kernel Hilbert
  spaces}.
\newblock {\em Transactions of the American Mathematical Society},
  353(10):3945--3969, 2001.

\bibitem{Manton:2013kk}
J.~H. Manton.
\newblock {A Primer on Stochastic Differential Geometry for Signal Processing}.
\newblock {\em Selected Topics in Signal Processing, IEEE Journal of},
  7(4):681--699, 2013.

\bibitem{Jiri2010thirty}
J.~Matou{\v{s}}ek.
\newblock {\em Thirty-three Miniatures: Mathematical and Algorithmic
  Applications of Linear Algebra}.
\newblock Student mathematical library. American Mathematical Society, 2010.

\bibitem{Mercer:1909gf}
J.~Mercer.
\newblock {Functions of Positive and Negative Type, and their Connection with
  the Theory of Integral Equations}.
\newblock {\em Philosophical Transactions of the Royal Society of London.
  Series A, Containing Papers of a Mathematical or Physical Character},
  209:415--446, Jan. 1909.

\bibitem{Newton:2002uz}
H.~J. Newton.
\newblock {A Conversation with Emanuel Parzen}.
\newblock {\em Statistical Science}, 17(3):357--378, Aug. 2002.

\bibitem{ParrBRT12}
W.~D. Parreira, J.~C.~M. Bermudez, C.~Richard, and J.-Y. Tourneret.
\newblock Stochastic behavior analysis of the gaussian kernel least mean square
  algorithm.
\newblock {\em IEEE Trans. on Sig. Proc.}, 60(5):2208--2222, 2012.

\bibitem{Parzen:1959uu}
E.~Parzen.
\newblock {Statistical inference on time series by Hilbert space methods}.
\newblock Technical Report~23, {Applied Mathematics and Statistics Laboratory,
  Stanford University}, Jan. 1959.

\bibitem{Parzen:1962ug}
E.~Parzen.
\newblock {Extraction and detection problems and reproducing kernel Hilbert
  spaces}.
\newblock {\em Journal of the Society for Industrial {\&} Applied Mathematics,
  Series A: Control}, 1(1):35--62, 1962.

\bibitem{parzen1963probability}
E.~Parzen.
\newblock Probability density functionals and reproducing kernel hilbert
  spaces.
\newblock In {\em Proceedings of the Symposium on Time Series Analysis}, volume
  196, pages 155--169. Wiley, New York, 1963.

\bibitem{Parzen:1970wz}
E.~Parzen.
\newblock {Statistical inference on time series by RKHS methods}.
\newblock In {\em Proc. Twelfth Biennial Sem. Canad. Math. Congr. on Time
  Series and Stochastic Processes; Convexity and Combinatorics (Vancouver,
  B.C., 1969)}, pages 1--37. Canad. Math. Congr., Montreal, Que., 1970.

\bibitem{Paulsen:rkhs}
V.~I. Paulsen.
\newblock An introduction to the theory of reproducing kernel {H}ilbert spaces.
\newblock \url{http://www.math.uh.edu/~vern/rkhs.pdf}, 2009.

\bibitem{PiciD88}
B.~Picinbono and P.~Duvaut.
\newblock Optimal linear-quadratic systems for detection and estimation.
\newblock {\em IEEE Trans. on Info. Theory}, 34(2):304--311, 1988.

\bibitem{PiciD90}
B.~Picinbono and P.~Duvaut.
\newblock Geometrical properties of optimal volterra filters for signal
  detection.
\newblock {\em IEEE Trans. on Info. Theory}, 36(5):1061--1068, 1990.

\bibitem{Pitcher:1960ub}
T.~S. Pitcher.
\newblock {Likelihood ratios of Gaussian processes}.
\newblock {\em Arkiv f\"or Matematik}, 4:35--44 (1960), 1960.

\bibitem{poor1994introduction}
H.~Poor.
\newblock {\em An Introduction to Signal Detection and Estimation}.
\newblock A Dowden \& Culver book. Springer, 1994.

\bibitem{Reny59}
A.~R{\'e}nyi.
\newblock On measures of dependence.
\newblock {\em Acta Math. Acad. Sci. Hungar.}, 10(2):441--451, september 1959.

\bibitem{RichBH09}
C.~Richard, J.~C.~M. Bermudez, and P.~Honeine.
\newblock Online prediction of time series data with kernels.
\newblock {\em IEEE Trans. on Signal Processing}, 57(3):1058--1067, Mar 2009.

\bibitem{Saitoh:1988vg}
S.~Saitoh.
\newblock {\em {Theory of reproducing kernels and its applications}}, volume
  189 of {\em Pitman Research Notes in Mathematics Series}.
\newblock Longman Scientific \& Technical, Harlow; copublished in the United
  States with John Wiley \& Sons, Inc., New York, 1988.

\bibitem{SchoS02}
B.~Sch{\"o}lkopf and A.~J. Smola.
\newblock {\em Learning with kernels}.
\newblock MIT Press, Cambridge, Ma, USA, 2002.

\bibitem{SlavBT14}
K.~Slavakis, P.~Bouboulis, and S.~Theodoridis.
\newblock {\em Academic Press Library in Signal Processing: Volume 1, Signal
  Processing Theory and Machine Learning}, chapter ch. 17, Online learning in
  reproducing kernel Hilbert spaces, pages 883--987.
\newblock Elsevier, 2014.

\bibitem{Small:2011wg}
C.~G. Small and D.~L. McLeish.
\newblock {\em {Hilbert Space Methods in Probability and Statistical
  Inference}}.
\newblock John Wiley {\&} Sons, Sept. 2011.

\bibitem{SongFG13}
L.~Song, K.~Fukumizu, and A.~Gretton.
\newblock Kernel embeddings of conditional distributions.
\newblock {\em IEEE Signal Processing Magazine}, 30(4):98--111, 2013.

\bibitem{SongSGBB12}
L.~Song, A.~Smola, A.~Gretton, J.~Bedo, and K.~Borgwardt.
\newblock Feature selection via dependence maximization.
\newblock {\em Journal of Machine Learning Research}, 13:1393--1434, 2012.

\bibitem{SripFL11}
B.~K. Sriperumbudur, K.~Fukumizu, and G.~R.~G. Lanckriet.
\newblock Universality, characteristic kernels and rkhs embeddings of measures.
\newblock {\em Journal of Machine Learning Research}, 12:2389--2410, 2011.

\bibitem{steinke2008kernels}
F.~Steinke and B.~Sch{\"o}lkopf.
\newblock Kernels, regularization and differential equations.
\newblock {\em Pattern Recognition}, 41(11):3271--3286, 2008.

\bibitem{Stei01}
I.~Steinwart.
\newblock On the influence of the kernel on the consistency of support vector
  machines.
\newblock {\em Journal of Machine Learning Research}, 2:67--93, 2001.

\bibitem{Stulajter:1978wt}
F.~Stulajter.
\newblock {RKHS approach to nonlinear estimation of random variables}.
\newblock In {\em Transactions of the Eighth Prague Conference on Information
  Theory, Statistical Decision Functions, Random Processes (Prague, 1978), Vol.
  B}, pages 239--246. Reidel, Dordrecht, 1978.

\bibitem{SunGS12}
Y.~Sun, F.~Gomez, and J.~Schmidhuber.
\newblock On the size of the online kernel sparsification dictionary.
\newblock In {\em proceedings of ICML, Edinburgh, Scotland}, 2012.

\bibitem{sussmann1997300}
H.~J. Sussmann and J.~C. Willems.
\newblock 300 years of optimal control: from the brachystochrone to the maximum
  principle.
\newblock {\em Control Systems, IEEE}, 17(3):32--44, 1997.

\bibitem{TikhA77}
A.~N. Tikhonov and V.~Y. Arsenin.
\newblock {\em Solutions of ill-posed problems}.
\newblock John Wiley\&Sons, 1977.

\bibitem{wahba1981spline}
G.~Wahba.
\newblock Spline interpolation and smoothing on the sphere.
\newblock {\em SIAM Journal on Scientific and Statistical Computing},
  2(1):5--16, 1981.

\bibitem{wahba1973interpolating}
G.~G. Wahba.
\newblock {\em Interpolating Spline Methods for Density Estimation}.
\newblock Department of Statistics, University of Wisconsin, 1973.

\bibitem{Whit89}
J.~Whittaker.
\newblock {\em Graphical models in applied multivariate statistics}.
\newblock Wiley\&Sons, 1989.

\bibitem{williams1991probability}
D.~Williams.
\newblock {\em Probability with Martingales}.
\newblock Cambridge mathematical textbooks. Cambridge University Press, 1991.

\bibitem{wong2011stochastic}
E.~Wong and B.~Hajek.
\newblock {\em Stochastic Processes in Engineering Systems}.
\newblock Springer Texts in Electrical Engineering. Springer New York, 2011.

\bibitem{Yao:1967tc}
K.~Yao.
\newblock {Applications of reproducing kernel Hilbert spaces --- bandlimited
  signal models}.
\newblock {\em Information and Control}, 11(4):429--444, 1967.

\end{thebibliography}

\end{document}